\documentclass[10pt,letterpaper]{amsart}
\usepackage{amsfonts}
\usepackage{amsmath}
\usepackage{amsthm,amscd}
\usepackage{tikz-cd}
\usepackage{mathtools}
\usepackage{amssymb}
\usepackage{mathrsfs}
\usepackage{enumerate}
\usepackage[alphabetic,initials]{amsrefs}
\usepackage[mathscr]{euscript}
\usepackage{etex}
\usepackage{hyperref}
\usepackage{svg}
\hypersetup{
    colorlinks=true,
    linkcolor=blue,
    filecolor=red,      
    urlcolor=teal,
    citecolor=teal,
}
\usepackage{color}

\usepackage{tikz}
\usepackage{xypic}
\usepackage{longtable}

\usepackage[T1]{fontenc}
\usepackage[utf8]{inputenc}

\usepackage{graphicx}

\allowdisplaybreaks

\calclayout

\theoremstyle{plain}\newtheorem{Theorem}{Theorem}[section]
\theoremstyle{plain}\newtheorem{Corollary}[Theorem]{Corollary}
\theoremstyle{plain}\newtheorem{Lemma}[Theorem]{Lemma}
\theoremstyle{plain}\newtheorem{Definition}[Theorem]{Definition}
\theoremstyle{plain}
\theoremstyle{plain}\newtheorem{Proposition}[Theorem]{Proposition}
\theoremstyle{plain}
\theoremstyle{plain}
\theoremstyle{plain}\newtheorem{Example}[Theorem]{Example}
\theoremstyle{plain}
\theoremstyle{plain}\newtheorem*{Theorem*}{Theorem}
\theoremstyle{plain}\newtheorem*{Problem*}{Problem}

\theoremstyle{plain}\newtheorem{theoremalph}{Theorem}[section]

\makeatletter
\newtheorem*{rep@theorem}{\rep@title}
\newcommand{\newreptheorem}[2]{%
\newenvironment{rep#1}[1]{%
 \def\rep@title{#2 \ref{##1}}%
 \begin{rep@theorem}}%
 {\end{rep@theorem}}}
\makeatother
\theoremstyle{plain}\newreptheorem{theorem}{Theorem}

\theoremstyle{remark}\newtheorem{remark}[Theorem]{Remark}
\theoremstyle{remark}\newtheorem{notation}[Theorem]{Notation}
\theoremstyle{remark}

\usepackage[colorinlistoftodos]{todonotes}%textwidth=35mm
\definecolor{antiquewhite}{rgb}{0.98, 0.92, 0.84}
\definecolor{buff}{rgb}{0.94, 0.86, 0.51}
\definecolor{palecopper}{rgb}{0.85, 0.54, 0.4}
\definecolor{fluorescentyellow}{rgb}{0.8, 1.0, 0.0}
\definecolor{bole}{rgb}{0.47, 0.27, 0.23}
\usepackage{amsmath,float}
\usetikzlibrary{decorations.markings}
\usepackage{graphicx}
\usepackage{subcaption}
\numberwithin{equation}{section}
\newcommand{\bA}{\mathbb{A}}

\newcommand{\bC}{\mathbb{C}}
\newcommand{\bD}{\mathbb{D}}

\newcommand{\bG}{\mathbb{G}}

\newcommand{\bP}{\mathbb{P}}
\newcommand{\bQ}{\mathbb{Q}}

\newcommand{\bZ}{\mathbb{Z}}
\newcommand{\clF}{\mathcal{F}}
\newcommand{\clG}{\mathcal{G}}

\newcommand{\clI}{\mathcal{I}}

\newcommand{\clL}{\mathcal{L}}
\newcommand{\clM}{\mathcal{M}}
\newcommand{\clN}{\mathcal{N}}
\newcommand{\clO}{\mathcal{O}}

\newcommand{\clQ}{\mathcal{Q}}

\newcommand{\clX}{\mathcal{X}}
\newcommand{\clY}{\mathcal{Y}}

\newcommand{\fm}{\mathfrak{m}}

\newcommand{\fp}{\mathfrak{p}}

\newcommand{\fA}{\mathfrak{A}}

\newcommand{\fS}{\mathfrak{S}}

\newcommand{\sC}{\mathscr{C}}

\newcommand{\sE}{\mathscr{E}}

\newcommand{\sS}{\mathscr{S}}

\newcommand{\sV}{\mathscr{V}}

\newcommand{\sX}{\mathscr{X}}
\newcommand{\sY}{\mathscr{Y}}

\newcommand{\clHom}{\mathcal{H}om}
\newcommand{\Et}{\normalfont{\text{\'EtSch}}}
\newcommand{\Mbar}{\overline{\mathcal{M}}}
\DeclareMathOperator\Spec{Spec}
\DeclareMathOperator\Proj{Proj}

\DeclareMathOperator\Hom{Hom}
\DeclareMathOperator\Mod{Mod}
\DeclareMathOperator\Art{Art}
\DeclareMathOperator\Aut{Aut}
\DeclareMathOperator\Sets{Sets}
\DeclareMathOperator\Def{Def}
\DeclareMathOperator\Hilb{Hilb}
\DeclareMathOperator\Sym{Sym}

\begin{document}

\author{Fatemeh Rezaee}
\address{University of Cambridge}
\curraddr{ETH Z\"urich}

\email{fr414@cam.ac.uk}
\author{Mohan Swaminathan}
\address{Stanford University}
\email{mohans@stanford.edu}
\title[Constructing smoothings of stable maps]{Constructing smoothings of stable maps}
\begin{abstract}
 Let $X$ be a smooth projective variety. Define a stable map $f:\sC\to X$ to be \emph{eventually smoothable} if there is an embedding $X\hookrightarrow\bP^N$ such that $(\sC,f)$ occurs as the limit of a $1$-parameter family of stable maps to $\bP^N$ with smooth domain curves. Via an explicit deformation-theoretic construction, we produce a large class of stable maps (called \emph{stable maps with model ghosts}), and show that they are eventually smoothable.
\end{abstract}

\maketitle

\setcounter{tocdepth}{1}

\section{Introduction}\label{sec:intro}

We work over a fixed algebraically closed field $k$ of characteristic $0$. Given a projective variety $X$, we will consider stable maps to it in the sense of Gromov and Kontsevich (see Definition \ref{def:prestable}). The motivating problem for this paper is the following.

\begin{Problem*}
    Given a stable map $f:\sC\to X$, determine necessary and sufficient conditions under which we can find an embedding $X\hookrightarrow\bP^N$, so that $(\sC,f)$ arises as a limit of stable maps from nonsingular curves to $\bP^N$.
\end{Problem*}

In the present paper, we derive an explicit sufficient condition and give several examples, in arbitrary genus, which satisfy this condition. In a follow up paper \cite{sequel}, we will investigate necessary conditions.\\\\
\indent This problem can be naturally interpreted in terms of the moduli stacks $\Mbar_{g,n}(X,\beta)$ of $n$-pointed stable maps to $X$ of genus $g$ and class $\beta$. These moduli stacks are important in enumerative geometry and physics, because when $X$ is non-singular, we can use virtual intersection theory on $\Mbar_{g,n}(X,\beta)$ to define the Gromov--Witten invariants of $X$.

The stack $\Mbar_{g,n}(X,\beta)$ is proper over $k$ and contains an open substack $\clM_{g,n}(X,\beta)$ parametrizing maps from non-singular curves to $X$. It is known (see \textsection\ref{subsec:lit-survey}) that when $g$ is positive, $\clM_{g,n}(X,\beta)$ is usually \emph{not} dense in $\Mbar_{g,n}(X,\beta)$, even when $X = \bP^N$ and the degree $c_1(\clO_{\bP^N}(1))\cap\beta$ is much bigger than the genus $g$. This stems from the fact that a stable map may contract an entire irreducible component of the domain curve to a point in the target variety. Following \cite{Zinger-sharp-compactness}, it is therefore natural to ask if we can explicitly describe a \emph{sharper} compactification $\Mbar^\circ_{g,n}(X,\beta)$ with the following properties. 
\begin{enumerate}[(i)]
    \item\label{cond:compact} $\Mbar^\circ_{g,n}(X,\beta)$ is (strictly) contained in $\Mbar_{g,n}(X,\beta)$ as a closed substack and contains $\clM_{g,n}(X,\beta)$ as an open substack.
    \item\label{cond:functorial} For any embedding $X\hookrightarrow X'$ of projective varieties, with the class $\beta'$ being the image of $\beta$, the stack $\Mbar^\circ_{g,n}(X,\beta)$ is the inverse image of $\Mbar^\circ_{g,n}(X',\beta')$ under the morphism
    \begin{align*}
        \Mbar_{g,n}(X,\beta)\hookrightarrow\Mbar_{g,n}(X',\beta').    
    \end{align*} 
    \item\label{cond:optimal} Given $(\sC,p_1,\ldots,p_n)\to X$ lying in $\Mbar^\circ_{g,n}(X,\beta)$, we can find $X\hookrightarrow X'$ as above such that $(\sC,p_1,\ldots,p_n)\to X\hookrightarrow X'$ lies in the closure of $\clM_{g,n}(X',\beta')$ in $\Mbar_{g,n}(X',\beta')$.
\end{enumerate}
Property \eqref{cond:compact} says that $\Mbar^\circ_{g,n}(X,\beta)$ is a compactification of $\clM_{g,n}(X,\beta)$ within $\Mbar_{g,n}(X,\beta)$. Property \eqref{cond:functorial} says that this compactification is \emph{functorial under embeddings}, while property \eqref{cond:optimal} says that it is \emph{optimal} (or \emph{sharp}) in a precise sense.

Our ultimate goal is to obtain an explicit description of a compactification $\Mbar^\circ_{g,n}(X,\beta)$ satisfying \eqref{cond:compact}--\eqref{cond:optimal} above by characterizing the stable maps which belong to it. An important application of the sharp compactification was the resolution of the genus one Bershadsky--Cecotti--Ooguri--Vafa (BCOV) conjecture in \cite{Zinger-BCOV-g1} (see \textsection\ref{subsec:lit-survey} for more information).

%%%%%%%%%%%%%%%%%%%%%%%%%%%%%%%%%%%%%%%%%%%%%%%%%%%
\subsection{Main results}\label{subsec:summary-of-results}

We start by briefly recalling some notions that are required for the discussion. As above, $X$ denotes a projective variety.

\begin{Definition}[Prestable curve, stable map]\label{def:prestable}
    A \emph{prestable curve} $\sC$ is a projective connected reduced curve with at worst ordinary nodes\footnote{By this, we mean a singularity which is \'etale locally given by $\Spec k[x,y]/(xy)$.} as singularities. A \emph{stable map} to $X$ is a morphism $f:\sC\to X$ from a prestable curve such that the group of automorphisms of $\sC$ which preserve $f$ is finite. It is also possible to consider a variant of this definition where one allows $\sC$ to have a finite ordered collection $p_1,\ldots,p_n\in\sC$ of non-singular marked points.
\end{Definition}

Denote by $\Mbar(X) = \bigsqcup_{g,\beta}\Mbar_{g,0}(X,\beta)$, the moduli stack of stable maps to $X$ (without marked points) and with no restriction on their genus or degree. Denote by $\clM(X)$ the open substack parametrizing maps defined on non-singular domain curves. We are interested in the question of when a stable map is {\bf smoothable} (Definition \ref{def:smoothability-of-map}), i.e., understanding when a stable map lies in the closure of $\clM(X)$ within $\Mbar(X)$. For a general target variety $X$, smoothability of a stable map is rather subtle and depends delicately on the geometry of $X$. We instead study the following more flexible notion, as in our motivating problem.

\begin{Definition}[Eventual smoothability]\label{def:eventual-smoothability}
    A stable map $f:\sC\to X$ is said to be \emph{eventually smoothable} if there exists an embedding $X\hookrightarrow\bP^N$ such that the point of $\Mbar(\bP^N)$ defined by the stable map $\sC\to X\hookrightarrow \bP^N$ lies in the closure of the open substack $\clM(\bP^N)$.
\end{Definition}

For a stable map $f:\sC\to X$, a {\bf ghost component} (Definition \ref{def:ghost-components}) is a maximal connected union of irreducible components of $\sC$ on which $f$ is constant. Eventual smoothability of a stable map is expected to depend only on its local behavior near each of its ghost components. Our first main result is a local-to-global principle in this vein.

\begin{theoremalph}[see Theorem \ref{thm:local-criterion-for-smoothability}, Local criterion for eventual smoothability]\label{thm:intro-local-criterion}
    If a non-constant stable map is locally formally smoothable, then it is eventually smoothable.
\end{theoremalph}

Our second main result concerns the eventual smoothability of a large class of stable maps, which we call {\bf stable maps with model ghosts}.

\begin{theoremalph}[see {Theorem \ref{thm:model-ghost-implies-eventual-smoothability}}]\label{thm:intro-model-ghost}
    Stable maps with model ghosts are eventually smoothable.
\end{theoremalph}

Let us now explain the meaning of the terms appearing in Theorems \ref{thm:intro-local-criterion} and \ref{thm:intro-model-ghost}.

In the statement of Theorem \ref{thm:intro-local-criterion}, \emph{local formal smoothability} (Definition \ref{def:local-smoothability-stable-map}) is a condition on the behavior of a stable map near its ghost components. This condition is related to the smoothability of the singularities of a certain non-prestable curve, of the same arithmetic genus as the domain of the stable map, obtained by collapsing the ghost components. 

In the statement of Theorem \ref{thm:intro-model-ghost}, \emph{stable maps with model ghosts} (Definition \ref{def:stable-map-with-model-ghosts}) are a class of stable maps whose behavior near any ghost component is governed by an explicit model which we now describe. Given any non-singular projective curve $C$ of positive genus with $n\ge 1$ marked points $p_1,\ldots,p_n\in C$ (and relatively prime integers $d_1,\ldots,d_n\ge 1$), we construct a corresponding \emph{model} $1$-parameter family $\{f_t\}_{t\in\bA^1}$ of stable maps, to a projective space, with the following properties (see \textsection\ref{subsec:model-sing-and-smoothing} for details). When $t\ne 0$, $f_t$ is a non-constant stable map with domain $C$, while $f_0$ is a stable map with exactly one ghost component $C$ and $n$ non-constant $\bP^1$ components attached to $C$ at the points $p_1,\ldots,p_n$. The number $1/d_i$ encodes the \emph{rate of smoothing}\footnote{This means that if the local equation of the central fibre (near the node $p_i$) is $z_iw_i = 0$, then the local equation of the smoothing has the form $z_iw_i=t^{d_i}$. The value $d_i$ is the \emph{rate of node formation, i.e., degeneration}, while $1/d_i$ is the \emph{rate of smoothing}.} at the node $p_i$ in the central fibre. Additionally, the restriction of $f_0$ to the $n$ non-constant $\bP^1$ components satisfies a non-trivial set of conditions on its Taylor expansions at the points $p_1,\ldots,p_n$. The simplest of these conditions is that the first derivatives, at $p_1,\ldots,p_n$, of the $n$ non-constant $\bP^1$ components have linearly dependent images. The full list of Taylor expansion conditions is governed by the (non)existence of rational functions on $C$, having poles of prescribed orders at $p_1,\ldots,p_n$ and regular elsewhere (see \textsection\ref{subsec:example-model-sing} for examples). In particular, the exact answer is sensitive to whether $(C,p_1,\ldots,p_n)\in\clM_{g,n}$ is a general point or whether it lies on some locus of special curves (e.g. hyperelliptic curves). With this in mind, we note that a stable map has \emph{model ghosts} if, near each ghost component, it satisfies the same conditions on its Taylor expansions as in the central fibre of a corresponding model family (Remark \ref{rem:model-ghost-in-coord}). We also allow ghost components of genus $0$ in the definition of \emph{stable maps with model ghosts}, without requiring any special behavior near these (Remark \ref{rem:genus-0-vs-higher}).

Once we verify that the models described above satisfy the local formal smoothability hypothesis, Theorem \ref{thm:intro-model-ghost} follows as a consequence of Theorem \ref{thm:intro-local-criterion}. The majority of this verification is carried out in Appendices \ref{appendix:analysis-model-smoothing} and \ref{appendix:genus-0-pinchings}.

To illustrate the scope of Theorem \ref{thm:intro-model-ghost}, we discuss several examples (in arbitrary genus) of stable maps with model ghosts in \textsection\ref{subsec:example-model-sing}, provide pictures where possible and explicitly work out the Taylor expansion conditions in each case. For instance, we have the following example which is of genus $>2$, and to the best of the authors' knowledge, does not directly follow from previous works on (eventual) smoothability of stable maps.

\begin{Example}[see \textsection\ref{subsubsec:all-d-bigger-than-1}, specifically Example \ref{exa:genus-3-with-2-points-d=2,3}, for more details]
    Let $(C,p_1,p_2)$ be a general non-singular $2$-pointed curve of genus $3$. Form a prestable curve $\sC$ by attaching two copies of $\bP^1$ to $C$, denoted $E_1$ and $E_2$, at the points $p_1$ and $p_2$ respectively. Let $f:\sC\to X$ be any morphism such that $C$ is a ghost component, $f|_{E_1}$ is non-constant with its first two derivatives vanishing at $p_1$ and $f|_{E_2}$ is non-constant with its first derivative vanishing at $p_2$. Then, $f:\sC\to X$ is a stable map with model ghosts and therefore, eventually smoothable.
\end{Example}

The notion of eventual smoothability has been studied -- although not using this terminology -- in the algebraic setting as well as in the setting of closed symplectic manifolds equipped with compatible almost complex structures. We emphasize that these prior works give {\bf non-trivial obstructions to eventual smoothability}, as in Examples \ref{exa:zinger-simple-case-non-smoothable} and \ref{exa:single-ghost-non-smoothable} below. In the symplectic setting, the projective embedding in Definition \ref{def:eventual-smoothability} can be replaced by an unobstructedness assumption on each non-constant component of the stable map.\footnote{More precisely, we make a \emph{jet transversality} assumption on each non-constant component. Due to multiply covered curves, it is unclear if allowing the almost complex structure on the target manifold to vary (in a finite dimensional family) would be a suitable replacement for the jet transversality assumption.} In \textsection\ref{subsec:lit-survey} below, we discuss the relation of our results to some of these earlier works.

\begin{Example}[Follows from {\cite[Definition 1.1 and Theorem 1.2]{Zinger-sharp-compactness}}]\label{exa:zinger-simple-case-non-smoothable}
    Let $f:\sC\to X$ be a stable map such that $\sC$ is of arithmetic genus $1$ and it has $n+1$ irreducible components, denoted by $C,E_1,\ldots,E_n$. Assume that $C$ is a non-singular curve of genus $1$, and that $f|_C$ is constant. Let $p_i\in\sC$ denote the point at which $E_i\simeq\bP^1$ meets $C$, for $1\le i\le n$. If $f|_{E_i}$ are immersions into $X$ near $p_i$, for $1\le i\le n$, and their tangent lines at $p_1,\ldots,p_n$ are linearly independent, then $f:\sC\to X$ is \emph{not} eventually smoothable.
\end{Example}

\begin{Example}[Follows from {\cite[Theorem 1.1]{DW-counting}} or {\cite[Theorem 1.1]{ekholm-shende-ghost}}]\label{exa:single-ghost-non-smoothable}
    Let $f:\sC\to X$ be a stable map such that $\sC$ has exactly two irreducible components, denoted by $C$ and $E$, which are non-singular and meet at a unique point $p\in\sC$. If $C$ is a positive genus ghost component and $f|_E$ is an immersion into $X$ near $p$, then $f:\sC\to X$ is \emph{not} eventually smoothable.
\end{Example}

%%%%%%%%%%%%%%%%%%%%%%%%%%%%%%%%%%%%%%%%%%%%%%%%%%%%
\subsection{Strategy of the proof}\label{subsec:proof-strategy}

The proof of Theorem \ref{thm:intro-local-criterion} is based on formal deformation theory. Given a stable map $f:\sC\to X$, we use the local formal smoothability hypothesis on it to construct a formal deformation of a neighborhood of each ghost component in $\sC$. These local deformations are then patched together into a global deformation of the stable map $f:\sC\to X$, yielding the desired smoothing. For this local-to-global step to work, we require a positivity assumption and it is here that we make essential use of the freedom to replace the original target variety $X$ by a projective space $\bP^N$.

As explained in \textsection\ref{subsec:summary-of-results}, the proof of Theorem \ref{thm:intro-model-ghost} involves the construction of explicit model $1$-parameter families of stable maps to which Theorem \ref{thm:intro-local-criterion} can be applied. To explain the construction of the model families, in a special case, let us consider a non-singular projective curve $C$ of genus $\ge 1$ and a marked point $p\in C$. Let $h_1,\ldots,h_r$ be a set of non-constant elements of the coordinate ring of the affine curve $C\setminus\{p\}$. In particular, for $1\le i\le r$, the rational function $h_i$ has a pole of order $m_i\ge 1$ at $p$ and is regular elsewhere on $C$. By re-ordering, we may assume that $m_r=\max_{1\le i\le r}m_i$. Now, consider the rational map
\begin{align*}
    &f:C\times\bA^1\dashrightarrow\bP^r\\
    &f(z,t) = [1:t^{m_1}h_1(z):\cdots:t^{m_r}h_r(z)].
\end{align*}
Then, $f$ determines a well-defined morphism on $C\times\bA^1\setminus\{(p,0)\}$. The indeterminacy of $f$ at $(p,0)$ is resolved by blowing up this point and this defines a family of stable maps over $\bA^1$ (see Figure \ref{fig:strategyOfProof}). The stable map $f_0$, which is the central fibre of this family, has domain curve given by the union of $C$ with a copy of $\bP^1$ attached at $p\in C$. From the expression for $f(z,0)$, it is clear that $f_0$ maps the whole of $C$ to the point $q = [1:0:\cdots:0]\in\bP^r$. An explicit computation, using the local equation $zv=tu$ for the blow-up of $C\times\bA^1$ at $(p,0)$, shows that the restriction of $f_0$ to the exceptional curve $\bP^1$ is given by
\begin{align*}
    f_0([u:v]) = [u^{m_r}:c_1v^{m_1}u^{m_r-m_1}:\cdots:c_r v^{m_r}],
\end{align*}
where $c_1,\ldots,c_r$ are some nonzero elements of the field $k$. The model family is obtained when we specialize the discussion above to the case when $m_1,\ldots,m_r$ generate the \emph{Weierstrass semigroup} of $C$ at $p$ (see \textsection\ref{subsubsec:monomial-curves} for more information).

\begin{figure}[ht]
    \centering
\includegraphics[width=7.3cm]{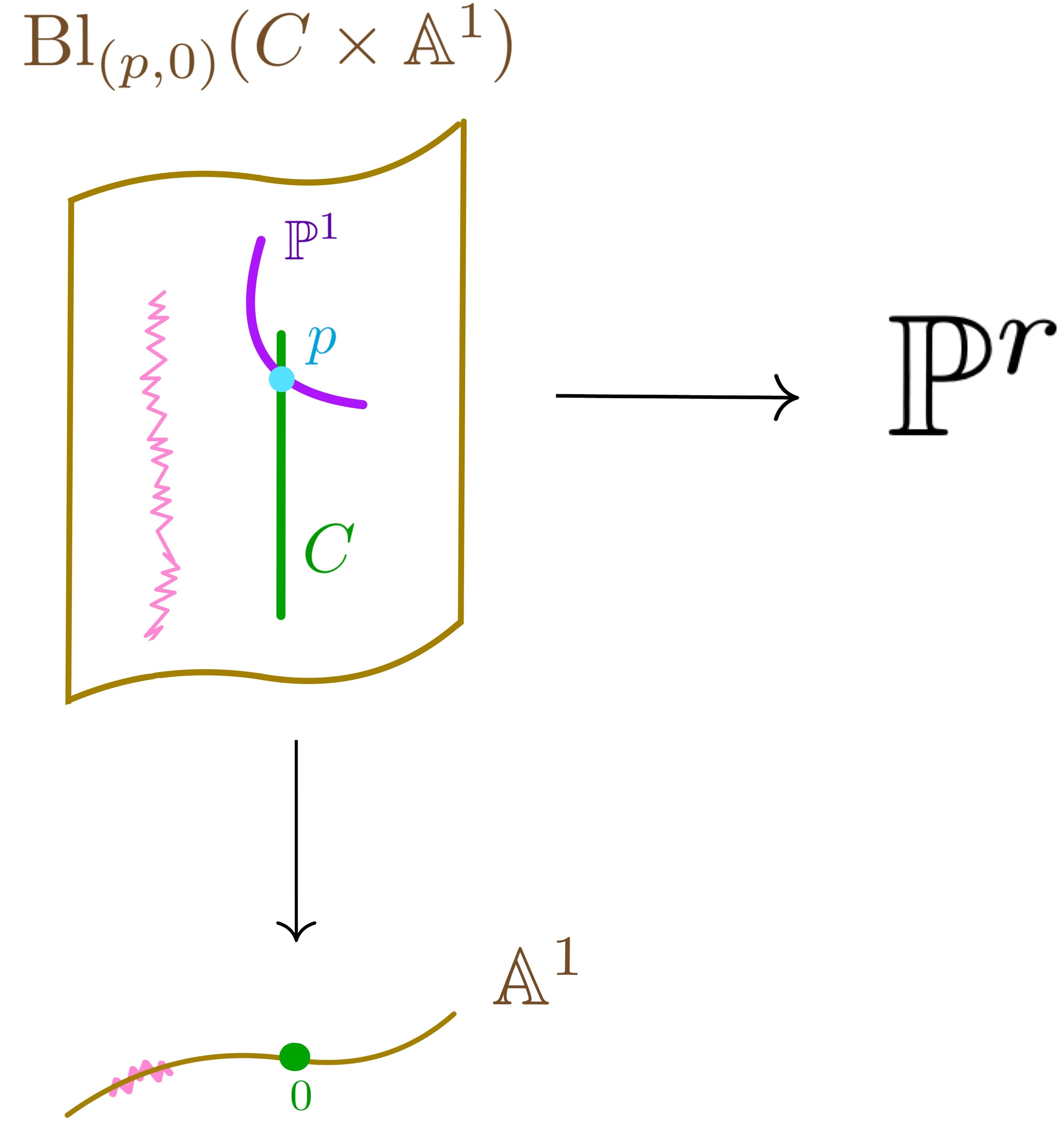}
    \caption{The model family of stable maps, from \textsection\ref{subsec:proof-strategy}, associated to $(C,p)$.}
    \label{fig:strategyOfProof}
\end{figure}

%%%%%%%%%%%%%%%%%%%%%%%%%%%%%%%%%%%%%%%%%%%%%%%%%%
\subsection{Relations to previous work}\label{subsec:lit-survey}

In the algebraic setting, some works which are closely related to the present paper are \cites{Pinkham-thesis, Pinkham-surf-sing} and \cites{Vakil-genus01, Vakil-stable-red}. The result of \cite[Lemma 5.9]{Vakil-genus01} describes a necessary condition for a genus $1$ stable map to be smoothable. When $X$ is a surface, the idea of relating the \emph{ghost components} of a stable map $f:\sC\to X$ to the \emph{singularities of the image} of $f$ is made explicit  in \cite[Lemma 2.4.1]{Vakil-stable-red}. Although the papers \cites{Pinkham-thesis, Pinkham-surf-sing} do not mention stable maps (for the simple reason that these were written much before stable maps were defined), they nevertheless serve as the main source of inspiration for the definition of the model families of \textsection\ref{subsec:model-sing-and-smoothing}. We also mention that there is an alternate compactification of $\clM(X)$, with a natural \emph{non-surjective} forgetful map to $\Mbar(X)$, parametrizing \emph{stable unramified maps}. This is constructed in \cite{KKO}; for an expository account, see \cite[Section $5\frac12$]{PT-13-over-2}.

In the symplectic setting as well, there are works with close relations to the present paper. For genus $1$ stable maps, eventual smoothability was studied in \cite{ionel-j-inv}, and subsequently, a complete characterization was given in \cite{Zinger-sharp-compactness}. \cite{VZ-desing} used this to construct an explicit desingularization of the closure of $\clM_{1,n}(\bP^N,d)$ in $\Mbar_{1,n}(\bP^N,d)$, which led to the resolution of the Bershadsky--Cecotti--Ooguri--Vafa (BCOV) conjecture for genus one Gromov--Witten invariants of the quintic $3$-fold in \cite{Zinger-BCOV-g1}. Subsequently, \cite{Ranganathan--Santos-Parker--Wise} used logarithmic geometry to interpret the desingularization in \cite{VZ-desing} as the solution to a natural moduli problem. A complete characterization of eventual smoothability for genus $2$ stable maps in the symplectic setting was given in \cite{Niu-thesis}; see also the related work \cite{hu-li-niu-genus-2} in the algebraic setting. We also mention that a modular desingularization of the closure of $\clM_{2,n}(\bP^N,d)$ in $\Mbar_{2,n}(\bP^N,d)$ was constructed in \cite{Battistella-Carocci}, by generalizing the method of \cite{Ranganathan--Santos-Parker--Wise}.

In genus $1$ and $2$, the characterizations given in \cite{Zinger-sharp-compactness} and \cite{Niu-thesis} lead to the description of a compactification of $\clM(X)$ which is \emph{strictly smaller} than $\Mbar(X)$. Following \cite{Zinger-sharp-compactness}, we refer to this as the \emph{sharp compactification}. By the methods of the present paper, we are able to produce eventual smoothings for \emph{most} of the stable maps occurring in the genus $1$ and $2$ sharp compactifications (see Examples \ref{exa:non-W-pt}, \ref{exa:hyperelliptic}, \ref{exa:genus-1-with-many-eff}, \ref{exa:genus-2-with-many-eff}, \ref{exa:genus-2-with-hyperelliptic-conjugate-points} and Remark \ref{rem:model-ghost-for-suspension}). Using the generalization of the notion of stable maps with model ghosts explained in Remark \ref{rem:etale-subtleties}, it is possible to cover \emph{all} the stable maps occurring in the genus $1$ and $2$ sharp compactifications. 

The hope is that the compactifications in \cite{Niu-thesis} and \cite{Battistella-Carocci} (and their analogues in higher genus, once defined) could shed light on the BCOV conjecture in genus bigger than one. However, let us point out that the BCOV conjecture in genus bigger than one has been resolved by two approaches not involving a sharp compactification. The first one is contained in \cites{NMSP2, NMSP3}, and its foundation is the theory of $N$-mixed-spin-$P$ fields developed in \cite{NMSP1} (generalizing mixed-spin-$P$ fields introduced in \cites{CLLL,CLLL2}). The second one is contained in \cites{Guo-Janda-Ruan-BCOV-g2, Guo-Janda-Ruan-BCOV-higher-g}, and its foundation is the theory of logarithmic gauged linear sigma model developed in \cites{log-GLSM1, log-GLSM2, log-GLSM3}.

An obstruction to eventual smoothability in the symplectic setting, for stable maps of any positive genus, was produced in \cite{DW-counting}, by generalizing some arguments of \cite{ionel-j-inv}, \cite{Zinger-sharp-compactness} and \cite{Niu-thesis}. A similar obstruction was produced in \cite{ekholm-shende-ghost} using a different argument.

%%%%%%%%%%%%%%%%%%%%%%%%%%%%%%%%%%%%%%%%%%%%%%%%%%%
\subsection{Future directions} 

In a sequel \cite{sequel} to the present paper, we will continue the study of eventual smoothability, focusing on obstructions. In particular, we will strengthen (in the algebraic setting) the obstructions found in \cite{DW-counting} and \cite{ekholm-shende-ghost} (in the symplectic setting), and use this to give many examples of stable maps which are not eventually smoothable. For some known examples of stable maps which are not eventually smoothable (based on the literature mentioned in \textsection\ref{subsec:lit-survey}), see Examples \ref{exa:zinger-simple-case-non-smoothable} and \ref{exa:single-ghost-non-smoothable} above.

%%%%%%%%%%%%%%%%%%%%%%%%%%%%%%%%%%%%%%%%%%%%%%%%%%%
\subsection*{Structure of the paper} 

In section \ref{sec:recall-defthy}, we review (without proof but with precise references) the main results from deformation theory that we need in this paper. In section \ref{sec:eventual-smoothability}, we prove Theorem \ref{thm:intro-local-criterion} (= Theorem \ref{thm:local-criterion-for-smoothability}), which gives a local criterion for eventual smoothability. In section \ref{sec:local-models}, we introduce the class of \emph{stable maps with model ghosts} and prove Theorem \ref{thm:intro-model-ghost} (= Theorem \ref{thm:model-ghost-implies-eventual-smoothability}), which states that these are eventually smoothable. Some technical details of this proof are deferred to Appendices \ref{appendix:analysis-model-smoothing} and \ref{appendix:genus-0-pinchings}. For the reader's convenience, at the beginning of each section, we provide a list of notations used in that section along with a reference to where each notation is introduced.

%%%%%%%%%%%%%%%%%%%%%%%%%%%%%%%%%%%%%%%%%%%%%%%%%%%%
\subsection*{Acknowledgements} 
We thank Ben Church, Amanda Hirschi, Patrick Kennedy-Hunt,  Adeel Khan, Andrew Kresch, T.R. Ramadas, Kyler Siegel, Sridhar Venkatesh and Sushmita Venugopalan for helpful discussions and correspondence. 
We would also like to thank Mark Gross, Eleny Ionel, Dominic Joyce, Melissa Liu, Rahul Pandharipande, John Pardon, Dhruv Ranganathan and Ravi Vakil for useful conversations and suggestions on a draft version of this paper.
We also thank the anonymous referee for their careful reading of the paper and for their useful suggestions.
This work came out of initial discussions the authors had while they were in residence at SLMath (formerly known as MSRI) in Berkeley, California, during the Fall 2022 semester (supported by the NSF under Grant No. DMS-1928930). F.R. was supported by UKRI grant No. EP/X032779/1 (in lieu of an MSCA grant), and was in residence at ETH Z\"urich during the final stages of writing. We acknowledge the use of GeoGebra \cite{geogebra5} and Goodnotes for figures, and Macaulay2 \cite{M2} for computations.

%%%%%%%%%%%%%%%%%%%%%%%%%%%%%%%%%%%%%%%%%%%%%%
\subsection*{Conventions} Unless explicitly stated otherwise, the schemes we consider will be of finite type (usually, quasi-projective) over a fixed algebraically closed field $k$ of characteristic $0$. Similarly, $\bA^n$ (resp. $\bP^n$) without subscripts is understood to be the $n$-dimensional affine space (resp. projective space) over $k$ and ring maps between $k$-algebras are understood to be $k$-algebra maps. For a finite dimensional $k$-vector space $V$, the algebra of polynomial functions on $V$ is denoted $\Sym V^\vee$. When tensor products and vector space dimensions are written without subscripts, they are understood to be over $k$.

The word \emph{variety} is reserved for a scheme which is irreducible and reduced. \emph{Morphisms} (resp. \emph{rational maps}) of schemes are defined everywhere (resp. on a dense open subset) and indicated by $\to$ (resp. $\dashrightarrow$). In particular, a \emph{birational morphism} is defined everywhere (though its inverse is typically a birational map which is defined only on a dense open subset). Following established terminology, we make an exception when the domain is a prestable curve and we say \emph{stable map} instead of stable morphism. A morphism of varieties is \emph{dominant} if its image is dense. For affine varieties, a dominant morphism corresponds to an injective map of coordinate rings. When $X$ is a scheme over a ring $A$, i.e., we have a morphism $X\to\Spec A$, its base change along a map $A\to B$ of rings will be denoted interchangeably by $X\otimes_AB$ and $X\times_{\Spec A}\Spec B$. We sometimes refer to a morphism of schemes $X\to S$ as the $S$-scheme $X$ (especially when the morphism is clear from context). A mophism $X\to Y$ between two $S$-schemes is called an $S$-morphism if it commutes with the projections to $S$.

For sheaves of $\clO_S$-modules $\clF,\clG$ on a scheme $S$, we let $\clHom_S(\clF,\clG)$ denote the sheaf of $\clO_S$-module homomorphisms between them. The space of global sections of $\clHom_S(\clF,\clG)$ is denoted $\Hom_S(\clF,\clG)$. For a point $p\in S$, we denote by $\clO_{S,p}$ (resp. $\fm_{S,p}$) the local ring (resp. maximal ideal) of $S$ at $p$. The $\fm_{S,p}$-adic completion of $\clO_{S,p}$ is denoted by $\clO_{S,p}^\wedge$ (with unique maximal ideal $\fm_{S,p}^\wedge$). The Zariski tangent space of $S$ at a closed point $p$ is denoted by $T_{S,p} = (\fm_{S,p}/\fm_{S,p}^2)^\vee$. When $S$ is non-singular of dimension $n$ near a closed point $p$, elements $z_1,\ldots,z_n\in\clO_{S,p}$ will be called \emph{local coordinates} for $S$ at $p$ if they generate the maximal ideal $\fm_{S,p}$. 

The notation $\bG_m = \Spec k[t^\pm]$ denotes the multiplicative group. Our convention is that the $\bG_m$-action on $\bA^1 = \Spec k[t]$ by scaling is of weight $1$ (and it corresponds to the grading on $k[t]$ under which $t$ has degree $1$).

More specific notations are introduced at the beginning of each section.

\section{Review of deformation theory}\label{sec:recall-defthy}

In this section, we review some background from deformation theory which will be used to study eventual smoothability of stable maps in \textsection\ref{sec:eventual-smoothability}. Throughout this section, we only recall the key statements and refer the reader to \cite{Har-DT} and \cite{stacks-project} for the proofs. In \textsection\ref{subsec:basic-dt-notions}, we recall some notions and constructions from deformation theory. These are applied in \textsection\ref{subsec:abstract-def} (resp. \textsection\ref{subsec:embedded-def}) to describe abstract (resp. embedded) deformations of schemes. Finally, in \textsection\ref{subsec:local-to-global} and \textsection\ref{subsec:top-inv-et-top} we record some useful facts which are needed to prove a local criterion for eventual smoothability (Theorem \ref{thm:local-criterion-for-smoothability}).

%%%%%%%%%%%%%%%%%%%%%%%%%%%%%%%%%%%%%%%%%%%%%%%%%%%%%%%%
\subsection*{Notation for this section}
   \begin{longtable}{ r l }
   $\Mod_B$ & Category of $B$-modules for a commutative ring $B$ (Notation \ref{not:ring-map-mod}).\\
   $T^i(B/A,M)$ & $T^i$ module for a ring map $A\to B$ and $M\in\Mod_B$ (Definition \ref{def:modules-T^i(B/A,M)}).\\
   $\clN_{Y/X}$& Normal sheaf $\clHom_Y(\clI/\clI^2,\clO_Y)$ of $Y$ in $X$ (Definition \ref{def:normal-sheaf}).\\ 
   $\fm_R$& Maximal ideal of a local $k$-algebra $R$ (Notation \ref{not:category-art}).\\
   $\Art_k$ & Category of local Artinian $k$-algebras with residue field $k$ (Notation \ref{not:category-art}).\\
   $\Art_k^\wedge$ & Category of complete local $k$-algebras $R$ such that $R/\fm_R^n\in\Art_k$ 
   \\& for all $n\ge 1$ (Notation \ref{not:category-art}).\\
   $D$ & Ring $k[t]/(t^2)$ of dual numbers (Notation \ref{not:category-art}).\\
   $h_S$ & Pro-representable functor associated to $S\in\Art_k^\wedge$ (Definition \ref{def:pro-rep}).\\
   $\Def_X$& Deformation functor of a scheme $X$ (Definition \ref{def:deformation-functor-of-scheme}).\\
   & If $X = \Spec B$, then $\Def_{X}(D) = T^1(B/k,B)$ (Proposition \ref{prop:deformation-of-B-over-D}).\\
   $\Hilb_{(Y,X)}$& Deformation functor of a closed subscheme $Y\subset X$ (Definition \ref{def:deformation-functor-of-subscheme}).\\
   & $\Hilb_{(Y,X)}(D) = H^0(Y,\clN_{Y/X})$ (Proposition \ref{prop:deformation-of-subscheme-over-D}).\\
   $\text{\'EtSch}/S$& Category of \'etale $S$-schemes for a scheme $S$ (Notation \ref{not:et-sch-cat}).
   \end{longtable}

%%%%%%%%%%%%%%%%%%%%%%%%%%%%%%%%%%%%%%%%%%%%%%%%%%%%%%%   
\subsection{Basic notions}\label{subsec:basic-dt-notions} To state the key results we need about abstract (resp. embedded) deformations of schemes, we review the definitions of the Lichtenbaum--Schlessinger $T^i$ functors (resp. normal sheaves). Additionally, we also recall some standard notions of deformation theory (such as functors of Artin rings).

\begin{notation}\label{not:ring-map-mod}
    Let $A\to B$ be a map of commutative rings with identity (not necessarily of finite presentation or over a base field). Let $\Mod_B$ denote the category of $B$-modules.
\end{notation}

We briefly describe the construction of the Lichtenbaum--Schlessinger functors $T^i(B/A,-)$, for $0\leq i \leq 2$. The reader is referred to \cite[Construction 3.1]{Har-DT} for more details.

To begin, choose a presentation $B = R/I$ of the $A$-algebra $B$ with $R = A[x]$ being a free polynomial algebra (with $x$ denoting the possibly infinite collection of variables) and $I\subset R$ being an ideal. Next, choose a free $R$-module $F$ fitting into a short exact sequence
\begin{align*}
    0\to Q\to F\xrightarrow{\alpha} I\to 0
\end{align*}
of $R$-modules. Let $F_0\subset Q$ be the $R$-sub-module generated by the elements $\alpha(f)f'-\alpha(f')f$ (called Koszul relations) ranging over $f,f'\in F$. Having made these choices, we can define the (truncated) cotangent complex of $A\to B$.

\begin{Definition} [Cotangent complex] \label{def:cotangent-complex} With notation as above, define the $B$-modules $L_i$ for $0\leq i \leq 2$ as follows.
\begin{align*}
     L_0&:=\Omega_{R/A}\otimes_R B\\
     L_1&:=F\otimes_RB=F/IF\\
     L_2&:=Q/F_0
\end{align*}
Here, $\Omega_{R/A}$ is the module of relative K\"ahler differentials of the $A$-algebra $R$. Next, define the $B$-module maps $d_i:L_i\to L_{i-1}$ for $i=1,2$ as follows. The map $d_1$ is defined to be the composition of the maps $L_1 = F/IF\to I/I^2$ induced by $\alpha$ and $I/I^2\to \Omega_{R/A}\otimes_R R/I = L_0$ induced by the derivation $R\to \Omega_{R/A}$. The map $d_2$ is induced by the inclusion $Q\subset F$.
These yield a complex
\begin{align*}
    L_2 \xrightarrow{d_2} L_1 \xrightarrow{d_1} L_0
\end{align*}
of $B$-modules, called the \emph{cotangent complex} of $A\to B$, denoted by $L_\bullet$.
\end{Definition}

Using the cotangent complex $L_\bullet$, we can define $T^i(B/A,M)$ for $0\le i\le 2$.

\begin{Definition}[$T^i$ modules]\label{def:modules-T^i(B/A,M)} With notation as above, for any $B$-module $M$, define
\begin{align*}
    T^i(B/A,M):=h^i(\Hom_B(L_\bullet,M))
\end{align*}
for $0\le i\le 2$.
\end{Definition}

\begin{remark}\label{rem:T^i-well-defined}
    Using \cite[Lemma 3.2 and Lemma 3.3]{Har-DT}, we see that $T^i(B/A,M)$ are well-defined $B$-modules, i.e., they are independent of the choices of the presentations $B = R/I$ and $I = F/Q$. Moreover, by \cite[Remark 3.3.1]{Har-DT}, the complex $L_\bullet$ gives a well-defined element of the derived category of $\Mod_B$.
\end{remark}

\begin{remark}\label{rem:T^i-functors}
    Using Definition \ref{def:modules-T^i(B/A,M)} and Remark \ref{rem:T^i-well-defined}, we obtain functors
    \begin{align*}
        T^i(B/A,-):\Mod_B\to\Mod_B
    \end{align*}
    for $0\le i\le 2$. One can show that $T^i$ functors are compatible with localization and this can be used to extend the definition to the non-affine case.
\end{remark}

When $B$ is a $k$-algebra of finite type and $M$ is a $B$-module, the following statement allows us to explicitly compute $T^i(B/k,M)$. The case $M = B$ is relevant to the study of abstract deformations of the affine scheme $\Spec B$.

\begin{Proposition}[{\cite[Proposition 3.10]{Har-DT}}]
\label{prop: s.e.s for T^i(B/k,M)}
    Let $A=k[x_1,\ldots,x_n]$, and let $B = A/I$ be a $k$-algebra obtained as a quotient of $A$. Then, for any $B$-module $M$, we have a natural exact sequence
    \begin{align*}
        0\to T^0(B/k,M)\to \mathrm{Hom}(\Omega_{A/k},M) \to \mathrm{Hom}(I/I^2,M) \to T^1(B/k,M) \to 0
    \end{align*}
    of $B$-modules. Furthermore, we have a  natural isomorphism 
    \begin{align*}
        T^2(B/A,M) = T^2(B/k,M)
    \end{align*} 
    of $B$-modules.
\end{Proposition} 

For the study of embedded deformations, the following notion is useful.

\begin{Definition}[Normal sheaf]\label{def:normal-sheaf} Let $X$ be a scheme, and let $Y\subset X$ be a closed subscheme with ideal sheaf $\clI:=\clI_{Y/X}$. The \emph{conormal sheaf} of $Y$ in $X$ is defined to be $\clI/\clI^2$. The \emph{normal sheaf} of $Y$ in $X$ is defined to be
 \begin{align*}
        \clN_{Y/X} := \clHom_X(\clI,\clO_Y)=\clHom_Y(\clI/\clI^2,\clO_Y).
\end{align*}
If $X$ is non-singular and $Y$ is a local complete intersection in $X$, then the $\clO_Y$-module $\clI/\clI^2$ and its dual $\clN_{Y/X}$ are locally free. In this case, $\clN_{Y/X}$ is called the \emph{normal bundle} of $Y$ in $X$.
\end{Definition}

Finally, we recall some standard notions which are required to state the key results on abstract and embedded deformations.
\begin{Definition}
    [Torsor and pseudotorsor] \label{def: torsor and pseudotorsor} Let $\fS$ be a set equipped with the action of a group $G$. We call $\fS$ a \emph{pseudotorsor} under the action of $G$, if $G$ acts freely and transitively on $\fS$. If $\fS$ is also non-empty, then it is said to be a \emph{torsor} under the action of $G$.
\end{Definition}

\begin{notation}\label{not:category-art}
    For a local $k$-algebra $R$, we denote its maximal ideal by $m_R$. Let $\Art_k$ denote the category of local Artinian $k$-algebras with residue field $k$. Similarly, let $\Art_k^\wedge$ denote the category of complete local $k$-algebras $R$ such that for all $n\ge 1$,  $R/\fm_R^n\in\Art_k$. Morphisms in $\Art_k$ and $\Art_k^\wedge$ are local $k$-algebra maps. Also, we introduce the notation $D := k[t]/(t^2)$, for the ring of dual numbers.
\end{notation}

\begin{Definition}[Extension of Artin rings]\label{def:ext-of-art}
    An \emph{extension of Artin rings} is a surjection $R'\twoheadrightarrow R$ in $\Art_k$. Defining $J := \ker(R'\twoheadrightarrow R)$, we get an associated short exact sequence
    \begin{align*}
        0\to J\to R'\to R\to 0.
    \end{align*}
    We say the extension $R'\twoheadrightarrow R$ is \emph{small} if we have $\fm_{R'}J = 0$.
\end{Definition}

\begin{remark}
    Our definition of small extensions is a slight generalization of the one given in \cite{Har-DT} where it is also required that $J$ be $1$-dimensional over $k$.
\end{remark}

\begin{Definition}[Functors of Artin rings]\label{def:functor-of-art}
    A \emph{functor of Artin rings} is a functor
    $F:\Art_k\to\Sets$. A natural transformation $F\to G$ between functors of Artin rings is said to be \emph{strongly surjective} if, for any extension $R'\twoheadrightarrow R$ in $\Art_k$, the induced map $F(R')\to F(R)\times_{G(R)}G(R')$ is surjective.
\end{Definition}

\begin{Definition}[Pro-representable functor]\label{def:pro-rep}
    For $S\in\Art_k^\wedge$, we define $h_S$ to be the functor of Artin rings which maps $R\in\Art_k$ to the set of local $k$-algebra maps $S\to R$. A functor $F$ of Artin rings is said to be \emph{pro-representable} if it is isomorphic to $h_S$ for some $S\in\Art_k^\wedge$. In this, case we also say $S$ \emph{pro-represents} $F$.
\end{Definition}

\begin{Definition}[Versality]\label{def:versal}
    Let $F$ be a functor of Artin rings, $S\in\Art_k^\wedge$ and consider a natural transformation $\eta: h_S\to F$. We say that $\eta$ is a
    \begin{enumerate}[\normalfont(i)]
        \item \emph{(formally) versal family} if $\eta$ is strongly surjective and $\eta(D)$ is a bijection,
        \item \emph{(formally) universal family} if $\eta$ is an isomorphism.
    \end{enumerate}
\end{Definition}

\begin{remark}[Formal deformations]
    When $F$ is the functor corresponding to deformations of a geometric object (as in Definitions \ref{def:deformation-functor-of-scheme} and \ref{def:deformation-functor-of-subscheme} below) and $S\in\Art_k^\wedge$, we will refer to a natural transformation $\eta:h_S\to F$ as a \emph{formal deformation} parametrized by $S$. This is because, by the Yoneda lemma, $\eta$ is equivalent to the data of a compatible collection of elements $\eta_n\in F(S/\fm_S^n)$ for $n\ge 1$.
\end{remark}

%%%%%%%%%%%%%%%%%%%%%%%%%%%%%%%%%%%%%%%%%%%%%%%%%%%%%%%%
\subsection{Deformations of schemes}\label{subsec:abstract-def}

We review (abstract) deformations of schemes and then specialize to the affine case. As usual, the schemes under consideration are assumed to be of finite type over $k$.

\begin{Definition}
    [Deformations of a scheme]\label{def:deformation-of-scheme} Let $X_0$ be a scheme and $R\in\Art_k$. A \emph{deformation of $X$ over $R$} is a pair $(X_R,i_R)$, where $X_R$ is a scheme which is flat over $R$ and $i_R:X_0\hookrightarrow X_R$ is a closed immersion inducing an isomorphism $X_0\simeq X_R\otimes_R k$ of schemes over $k$. 
    
    For fixed $X$ and $R$, two deformations $(X_R,i_R)$ and $(X'_R,i'_R)$ of $X_0$ over $R$ are considered \emph{equivalent}, if there is an isomorphism $f:X_R\xrightarrow{\simeq} X'_R$ of schemes over $R$ such that $f\circ i_R = i'_R$.
\end{Definition}

\begin{Definition}[Deformation functor of a scheme]\label{def:deformation-functor-of-scheme}
    Let $X_0$ be a scheme. The functor of Artin rings
    \begin{align*}
        \Def_{X_0}:\Art_k\to\Sets
    \end{align*}
    which maps any $R\in\Art_k$ to the set of equivalence classes of deformations of $X_0$ over $R$ is called the \emph{deformation functor of $X_0$}.
\end{Definition}

\begin{Definition}
    [Extensions of deformations of schemes] \label{def:extensions-of-deformations-of-schemes} Let $X_0$ be a scheme and let $R'\twoheadrightarrow R$ be a surjection in $\Art_k$ with kernel $J$ as in Definition \ref{def:ext-of-art}. Let $(X_R,i_R)$ be a deformation of $X_0$ over $R$. An extension of $(X_R,i_R)$ over $R'$ is a scheme $X_{R'}$ which is flat over $R'$ along with a closed immersion $X_R\hookrightarrow X_{R'}$ inducing an isomorphism $X_R\simeq X_{R'}\otimes_{R'} R$ of schemes over $R$. 
    
    Two such extensions are equivalent if there is an $R'$-isomorphism between them which is compatible with the closed immersion of $X_R$.
\end{Definition}

\begin{remark}
    In the situation of Definitions \ref{def:deformation-of-scheme} and \ref{def:extensions-of-deformations-of-schemes}, we usually omit the closed immersion $i_R$ from the notation when it is clear from the context.
\end{remark}

\begin{remark}[Restriction of deformation functors]\label{rem:restriction-of-def-functor}
    Since the maximal ideal $\fm_R$ of a ring $R\in\Art_k$ is nilpotent, any deformation of a scheme $X_0$ over $R$ has the same underlying topological space as $X_0$. In particular, if $U_0\subset X_0$ is an open subscheme, then it makes sense to restrict a deformation of $X_0$ to $U_0$. This gives a restriction map $\Def_{X_0}\to\Def_{U_0}$ between the deformation functors of $X_0$ and $U_0$.
\end{remark}

\begin{remark}[Deformations of affine schemes are affine]\label{rem:def-of-affine-is-affine}
    If $X_0 = \Spec B_0$ is an affine scheme, then any deformation of $X_0$ over $R\in\Art_k$ is necessarily affine \cite[Remark 1.1]{Artin-tifr-lectures}. In this case, Definitions \ref{def:deformation-of-scheme}, \ref{def:deformation-functor-of-scheme} and \ref{def:extensions-of-deformations-of-schemes} have obvious reformulations with rings in place of affine schemes.
\end{remark}

The deformations of an affine scheme over the ring $D$ of dual numbers can be characterized as follows.

\begin{Proposition}[{\cite[Corollary 5.2]{Har-DT}}]
\label{prop:deformation-of-B-over-D} 

Let $X_0 = \Spec B_0$ be an affine scheme. Then, there is a natural bijection
\begin{align*}
    \Def_{X_0}(D) = T^1(B_0/k,B_0).
\end{align*}
\end{Proposition}

To study higher order deformations of affine schemes, we have the following statement characterizing when extensions of deformations exist.

\begin{Proposition}[{\cite[Theorem 10.1]{Har-DT}}]
     \label{prop:def-obs-for-affine-schemes} Let $B_0$ be a $k$-algebra, $B_R\twoheadrightarrow B_0$ be a deformation of $B_0$ over $R\in\Art_k$ and let $R'\twoheadrightarrow R$ be a small extension in $\Art_k$ with kernel $J$ as in Definition \ref{def:ext-of-art}. Then, there is an obstruction in $T^2(B_0/k,B_0\otimes J)$ whose vanishing is equivalent to the existence of an extension $B_{R'}\twoheadrightarrow B_{R}$ of $B_R\twoheadrightarrow B_0$ to $R'$. Moreover, if extensions do exist, then the set of their equivalence classes is a torsor under the action of $T^1(B_0/k,B_0\otimes J)$.
\end{Proposition}

\begin{Example}[Versal deformation of a node {\cite[Theorem 14.1]{Har-DT}}]\label{exa:versal-def-of-node}
    Consider the algebra corresponding to a node, i.e., $B_0 = k[x,y]/(xy)$. Then, a formally versal family for $\Def_{B_0}$ is given by the restrictions of $k[x,y,t]/(xy-t)$ to $k[t]/(t^n)\in\Art_k$, for $n\ge 1$. We call $t$ the \emph{smoothing parameter} for the node.
\end{Example}

%%%%%%%%%%%%%%%%%%%%%%%%%%%%%%%%%%%%%%%%%%%%%%%%%%%%%%%%
\subsection{Deformations of closed subschemes}\label{subsec:embedded-def}

We review (embedded) deformations of closed subschemes. 

\begin{Definition}[Deformations of closed subschemes]\label{def:deformation-of-subscheme}
    Fix a closed subscheme $Y_0\subset X_0$ and $R\in\Art_k$. An \emph{embedded deformation of $Y_0$ in $X_0$ over $R$} is a closed subscheme $Y_R\subset X_0\otimes_k R$ such that $Y_R$ is flat over $R$, and $Y_R\otimes_R k = Y_0\subset X_0$.
\end{Definition}

\begin{Definition}[Deformation functor of a closed subscheme]\label{def:deformation-functor-of-subscheme}
    Let $Y_0\subset X_0$ be a closed subscheme. The functor of Artin rings
    \begin{align*}
        \Hilb_{(Y_0,X_0)}:\Art_k\to\Sets
    \end{align*}
    which maps any $R\in\Art_k$ to the set of embedded deformations of $Y_0$ in $X_0$ over $R$ is called the \emph{local Hilbert functor of $Y_0\subset X_0$}.
\end{Definition}

\begin{remark}[Restriction of local Hilbert functors]\label{rem:restriction-of-hilb-functor}
    If $U_0\subset X_0$ is an open subscheme and $W_0:=Y_0\cap U_0\subset U_0$, then any embedded deformation of $Y_0$ in $X_0$ defines an embedded deformation of $W_0$ in $U_0$, yielding a map $\Hilb_{(Y_0,X_0)}\to\Hilb_{(W_0,U_0)}$ between local Hilbert functors.
\end{remark}

\begin{remark}[Forgetting the embedding]\label{rem:forget-embedding}
    An embedded deformation of $Y_0$ in $X_0$ also yields a deformation of just $Y_0$ by forgetting the data of the embedding. This gives a forgetful map $\Hilb_{(Y_0,X_0)}\to\Def_{Y_0}$.
\end{remark}

The embedded deformations of $Y_0$ in $X_0$ over the ring $D$ of dual numbers can be characterized as follows.

\begin{Proposition}[{\cite[Theorem 2.4]{Har-DT}}] \label{prop:deformation-of-subscheme-over-D}
   For a closed subscheme $Y_0\subset X_0$, we have a natural bijection
   \begin{align*}
       \Hilb_{(Y_0,X_0)}(D) = H^0(Y_0,\clN_{Y_0/X_0}).
   \end{align*}
\end{Proposition}

\begin{Example}[{\cite[Corollary 2.5]{Har-DT}}]
\label{ex: Tangent space to the Hilbert scheme}   Let $Y_0\subset \bP^n$ be a closed subscheme with Hilbert polynomial $P(t)$. Let $H=\Hilb_{P(t)}(\bP^n)$ be the Hilbert scheme which parametrizes closed subschemes of $\bP^n$ with Hilbert polynomial equal to $P(t)$. Then, the Zariski tangent space of $H$ at the point $Y_0$ can be computed as
\begin{align*}
    T_{H,Y_0}=H^0(Y_0,\clN_{Y_0/\bP^n})
\end{align*}
using Proposition \ref{prop:deformation-of-subscheme-over-D}.
\end{Example}

To study higher order embedded deformations, we have the following statement characterizing when extensions of embedded deformations exist.

\begin{Proposition}[Special case of {\cite[Theorem 6.2]{Har-DT}}]
     \label{thm:def-obs-for-closed-subschemes}
     Let $Y_0\subset X_0$ be a closed subscheme and let $R'\twoheadrightarrow R$ be a small extension in $\Art_k$ with kernel $J$ as in Definition \ref{def:ext-of-art}. Let $Y_R\subset X_0\otimes_k R$ be a given element of $\Hilb_{(Y_0,X_0)}(R)$ and consider lifts of this element to the set $\Hilb_{(Y_0,X_0)}(R')$.
     \begin{enumerate}[\normalfont(i)]
         \item The set of lifts is a pseudotorsor under the action of $H^0(Y_0,\clN_{Y_0/X_0}\otimes_kJ)$.
         \item If there exist local lifts of $Y_R\subset X\otimes_k R$ to $R'$, then there is an obstruction in $H^1(Y_0, \clN_{Y_0/X_0}\otimes_kJ)$ whose vanishing is equivalent to the global existence of a lift. Moreover, if there is one global lift, then the set of all global lifts is a torsor under the action of $H^0(Y_0, \clN_{Y_0/X_0}\otimes_kJ)$. 
     \end{enumerate}
\end{Proposition}

%%%%%%%%%%%%%%%%%%%%%%%%%%%%%%%%%%%%%%%%%%%%%%%%%%%%%%%
\subsection{Local-to-global principle for deformations}\label{subsec:local-to-global} We record a useful result which allows one, under suitable assumptions, to lift local abstract deformations into global embedded deformations.

\begin{notation}
    For an affine scheme $U_0 = \Spec B_0$, write $T^1_{U_0} := T^1(B_0/k,B_0)$.
\end{notation}

\begin{Theorem}[{\cite[Theorem 29.7]{Har-DT}}] \label{thm: Local-to-global principle}
    Let $Z$ be a non-singular projective scheme and let $X_0\subset Z$ be a closed subscheme with isolated singularities. Let $U_0=\bigsqcup_i U_i$, where each $U_i\subset X_0$ is an affine open subscheme containing one and only one singularity of $X_0$. The morphism $U_0\to X_0$ induces a map
    \begin{align}\label{eqn:global-to-local-map}
        \Hilb_{(X_0,Z)}\to\Def_{U_0}
    \end{align}
    of functors of Artin rings.
    Suppose that $H^1(X_0,\clN_{X_0/Z}) = 0$ and that the map
    \begin{align}\label{eqn:linearized-global-to-local-map}
        H^0(X_0,\clN_{X_0/Z}) \to T^1_{U_0}
    \end{align}
    induced by \eqref{eqn:global-to-local-map}, combined with Propositions \ref{prop:deformation-of-subscheme-over-D} and \ref{prop:deformation-of-B-over-D}, is surjective. Then, the map \eqref{eqn:global-to-local-map} is strongly surjective in the sense of Definition \ref{def:functor-of-art}.
\end{Theorem}

We only require the following special case of Theorem \ref{thm: Local-to-global principle}.

\begin{Corollary}\label{cor:special-case-of-local-to-global}
    Using the notation of Theorem \ref{thm: Local-to-global principle}, consider the special case in which $X_0$ is a reduced curve and $Z = \bP^r\times\bP^s$. If we have 
    \begin{align}\label{eqn:local-to-global-positivity}
        H^1(X_0,\clO_{\bP^r}(1)|_{X_0}) = H^1(X_0,\clO_{\bP^s}(1)|_{X_0}) = 0,    
    \end{align}
    then the hypothesis of Theorem \ref{thm: Local-to-global principle} is satisfied, and as a result, the map \eqref{eqn:global-to-local-map} is strongly surjective.
\end{Corollary}
\begin{proof}
     Up to replacing $\bP^r\times\bP^s$ by $\bP^r$ and replacing the condition \eqref{eqn:local-to-global-positivity} by the condition $H^0(X_0,\clO_{\bP^r}(1)|_{X_0}) = 0$, this statement is shown in the proof of \cite[Proposition 29.9]{Har-DT}. To get the result for $Z = \bP^r\times\bP^s$, we replace the Euler exact sequence for $\bP^r$ by its analogue for $\bP^r\times\bP^s$, namely
    \begin{align*}
        0\to\clO_{X_0}\oplus\clO_{X_0}\to\clO_{\bP^r}(1)^{\oplus (r+1)}|_{X_0}\oplus\clO_{\bP^s}(1)^{\oplus (s+1)}|_{X_0}\to T_{\bP^r\times\bP^s}|_{X_0}\to 0.
    \end{align*}
    and use the condition \eqref{eqn:local-to-global-positivity} in place of $H^0(X_0,\clO_{\bP^r}(1)|_{X_0}) = 0$. The argument from \cite[Proposition 29.9]{Har-DT} then applies word for word.
\end{proof}

%%%%%%%%%%%%%%%%%%%%%%%%%%%%%%%%%%%%%%%%%%%%%%%%%%%%%%%%
\subsection{Topological invariance of \'etale topology}\label{subsec:top-inv-et-top} We record a useful result that allows us to lift deformations along \'etale morphisms.

\begin{notation}\label{not:et-sch-cat}
    For a scheme $S$, let $\Et/S$ denote the consider the category of \'etale morphisms $X\to S$ of schemes, considered as a full subcategory of the category of $S$-schemes.
\end{notation}

\begin{Theorem}[{\cite[\href{https://stacks.math.columbia.edu/tag/039R}{Tag 039R}]{stacks-project}}]\label{thm:top-inv-et-top} 
    Let $S$ be a scheme, and let $S_0 \subset S$ be a closed subscheme defined by a nilpotent ideal. Then, the functor
    \begin{align*}
        \Et/S &\to \Et/S_0\\ 
        X&\mapsto X\times_{S}S_0
    \end{align*}
    defines an equivalence of categories.
\end{Theorem}
\section{Local criterion for eventual smoothability}\label{sec:eventual-smoothability}

In this section, we discuss a sufficient condition for eventual smoothability of stable maps which depends only on the behavior of the stable map near its ghost components. In \textsection\ref{subsec:formal-smoothability-of-map}, we study the notion of \emph{formal smoothability} for stable maps. In \textsection\ref{subsec:pinching-a-curve}, we discuss the process of \emph{pinching} a prestable curve at a sub-curve. In \textsection\ref{subsec:smoothability-of-pinching} (resp. \textsection\ref{subsec:behavior-near-ghost-component}), we define a notion of \emph{local formal smoothability} for pinchings (resp. stable maps). Finally, in \textsection\ref{subsec:local-criterion-eventual-smoothability}, we state and prove the local criterion for eventual smoothability (Theorem \ref{thm:intro-local-criterion} in \textsection\ref{subsec:summary-of-results}).

%%%%%%%%%%%%%%%%%%%%%%%%%%%%%%%%%%%%%%%%%%%%%%%%%%%%%%%
\subsection*{Notation for this section} 
   \begin{longtable}{ r l }
    $X$ & Projective variety. \\
    $\Mbar(X)$ & Stack of stable maps to $X$ (Notation \ref{not:stable-map-moduli}). \\
    $\clM(X)$ & Stack of stable maps to $X$ with non-singular\\
    & domains (Notation \ref{not:stable-map-moduli}).\\
    $f:\sC\to X$ & Stable map to $X$. \\
    $\bD_n$ & Scheme $\Spec k[t]/(t^{n+1})$ for $n\ge 0$ (Notation \ref{not:1D-fat-points}).\\
     $\bD$ & Scheme $\Spec k[[t]]$ (Notation \ref{not:1D-fat-points}).\\
    $ \{\sC_n\to\bD_n,\,f_n:\sC_n\to X\}_{n\ge 0}$ &  Formal smoothing  of $f:\sC\to X$ (Definition \ref{def:formal-smoothability-of-map}).\\
    $\sS$ &Pinched curve (Definition \ref{def:pinching}, Figure \ref{fig: pinching}).\\
    $\nu:\sC\to\sS$ &Pinching morphism (Definition \ref{def:pinching}, Figure \ref{fig: pinching}).\\
    $s_i$ & Distinct closed points of $\sS$ for $1\le i\le r$ over \\ 
    & which $\nu$ is not an isomorphism (Definition \ref{def:pinching}).\\
    $C_i$ & Set-theoretic inverse image of $s_i$ under $\nu$ (Definition \ref{def:pinching}).\\
    $U_i$& Affine open neighborhood of $s_i\in\sS$ with $U_i\setminus\{s_i\}$ \\
    & non-singular (Notation \ref{not:zariski-nbd-of-pinched-point} and Figure  \ref{fig:etale-local-pinching}).\\
    $\sC_{U_i}$&Inverse image of $U_i$ under $\nu$ (Notation \ref{not:zariski-nbd-of-pinched-point}, Figure \ref{fig:etale-local-pinching}).\\
    $\tilde U_i$& Closure of $\sC_{U_i}\setminus C_i$ in $\sC_{U_i}$ (Notation \ref{not:zariski-nbd-of-pinched-point}).\\
    $V_i$& Reduced affine curve containing closed point $s'_i$ with \\ 
    & $V_i\setminus\{s_i'\}$ non-singular (Notation \ref{not:etale-nbd-of-pinched-point}, Figure \ref{fig:etale-local-pinching}).\\
    $\varphi_i:U_i\to V_i$& \'Etale morphism with $\varphi_i^{-1}(s_i') = \{s_i\}$ (Notation \ref{not:etale-nbd-of-pinched-point}, Figure \ref{fig:etale-local-pinching}).\\
    $\eta_i: \tilde V_i\to V_i$& Normalization of $V_i$ (Notation \ref{not:etale-nbd-of-pinched-point}, Figure \ref{fig:etale-local-pinching}).\\
    $\sC_{U_i,\varphi_i}$ & Curve obtained by gluing $C_i$ and $\tilde V_i$ (Notation \ref{not:etale-nbd-of-pinched-point}, Figure \ref{fig:etale-local-pinching}).\\
    $\nu_{U_i,\varphi_i}:\sC_{U_i,\varphi_i}\to V_i$ & \'Etale local version of pinching morphism \\
    & (Notation \ref{not:etale-nbd-of-pinched-point}, Remark \ref{rem:etale-base-change-pinching}, Figure \ref{fig:etale-local-pinching}).\\
    $\tilde\varphi_i:\sC_{U_i}\to\sC_{U_i,\varphi_i}$ & Morphism induced by $\varphi_i$ (Notation \ref{not:etale-nbd-of-pinched-point}, Figure \ref{fig:etale-local-pinching}).\\
    $C$ & Ghost sub-curve of $f:\sC\to X$ (Definition \ref{def:ghost-components}).\\
    & $C$ has connected components $C_1,\ldots,C_r$ (Figure \ref{fig: setup}).
   \end{longtable}

%%%%%%%%%%%%%%%%%%%%%%%%%%%%%%%%%%%%%%%%%%%%%%%%%%%%%%% 
\subsection{Formal smoothability of a stable map}\label{subsec:formal-smoothability-of-map}

Let $X$ be a projective variety.

\begin{notation}\label{not:stable-map-moduli}
    $\Mbar(X)$ denotes the stack of all stable maps to $X$, without marked points and with no restrictions on their genus or degree. The open substack of $\Mbar(X)$ parametrizing maps with non-singular domain curves is denoted by $\clM(X)$. Note that $\Mbar(X)$ is a Deligne--Mumford stack, locally of finite type over $k$ by \cite[Theorem 3.14]{behrend-manin}.
\end{notation}

\begin{Definition}[Smoothability of a stable map]\label{def:smoothability-of-map}
    A stable map $f:\sC\to X$ is said to be \emph{smoothable} if the point of $\Mbar(X)$ given by $(\sC,f)$ is contained in the closure of the open substack $\clM(X)$.
\end{Definition}

\begin{notation}\label{not:1D-fat-points}
    For integers $n\ge 0$, we write $\bD_n := \Spec k[t]/(t^{n+1})$. We also write $\bD := \Spec k[[t]]$. Observe that we have $\Spec k = \bD_0\subset\bD_1\subset\bD_2\subset\cdots\subset\bD$. Other than the closed point (contained in $\bD_n$ for every $n\ge 0$), the scheme $\bD$ also has a generic point. Note that $\bD$ is \emph{not} of finite type over $k$.
\end{notation}

The following definition is the analogue of Definition \ref{def:smoothability-of-map} from the perspective of formal deformations.

\begin{Definition}[Formal smoothability of a stable map]\label{def:formal-smoothability-of-map}
    A \emph{formal smoothing} of the stable map $f:\sC\to X$ is a formal deformation
    \begin{align}\label{eqn:intro-formal-smoothing}
        \{\sC_n\to\bD_n,\,f_n:\sC_n\to X\}_{n\ge 0}
    \end{align}
    of $(\sC,f)$, such that the following condition holds: 
    For any node $p\in\sC$, there is an affine open neighborhood $U\subset\sC$ such that 
    \begin{enumerate}[\normalfont(i)]
        \item $U\setminus\{p\}$ is non-singular, and
        \item $U$ is \emph{non-trivially deformed} by \eqref{eqn:intro-formal-smoothing}, i.e., there is an integer $n\ge 1$ such that the deformation of $U$, induced\footnote{Induced by restriction, in the sense of Remark \ref{rem:restriction-of-def-functor}.} by $\sC_n\to\bD_n$, is not isomorphic to the trivial deformation $U\times\bD_n\to\bD_n$.
    \end{enumerate}
    We say that the stable map $f:\sC\to X$ is \emph{formally smoothable}, if it has a formal smoothing.
\end{Definition}

\begin{remark}
    Note that a formal smoothing of $f:\sC\to X$ is a sequence of structure sheaves (and morphisms) defined on a \emph{single underlying topological space}.
\end{remark}

The next lemma shows that formal smoothability of a stable map implies its smoothability. The proof is similar to \cite[Proposition 29.5]{Har-DT}, where an analogous statement is proved for the Hilbert scheme in place of $\Mbar(X)$.

\begin{Lemma}[Formal smoothability implies smoothability]\label{lem:formal-implies-actual}
   If $f:\sC\to X$ is a formally smoothable stable map, then it is smoothable.
\end{Lemma}
\begin{proof}
    Using the algebraicity of $\Mbar(X)$, we will produce a morphism $\bD\to\Mbar(X)$ mapping the closed point of $\bD$ to $(\sC,f)$, and the generic point of $\bD$ into $\clM(X)$.
    
    Let $U\to\Mbar(X)$ be a representable \'etale morphism from a finite type affine $k$-scheme $U = \Spec R$, with a $k$-point $0\in U$ mapping to the $k$-point of $\Mbar(X)$ given by $f:\sC\to X$. Let $\fm\subset R$ be the maximal ideal corresponding to $0\in U$. If $\sC$ has exactly $N$ nodes, we can choose $U$ so that for each $1\le i\le N$, there is a regular function $f_i\in\fm$ which gives the smoothing parameter of the $i^\text{th}$ node. Let $U^\circ\subset U$ denote the inverse image of $\clM(X)$. Then, $U^\circ$ is the complement of the closed subscheme of $U$ defined by the vanishing of $\prod_{1\le i\le N} f_i$. The morphism $U\to\Mbar(X)$ yields a family $\sC_U\to U$ of prestable curves with a morphism $f_U:\sC_U\to X$ extending $f$.
    
    Now, the complete local ring $\clO_{U,0}^\wedge$ pro-represents the deformation functor\footnote{This functor is defined by mapping $R\in\Art_k$ to the set of isomorphism classes of $(\sC_R,f_R)$ where $\sC_R$ is a deformation of $\sC$ over $R$ (as in Definition \ref{def:deformation-of-scheme}) and $f_R:\sC_R\to X$ is a morphism extending $f$.} of the stable map $(\sC,f)$. Thus, the formal smoothing induces a compatible collection of local $k$-algebra maps 
    \begin{align*}
        \clO_{U,0}^\wedge\to k[t]/(t^{n+1})    
    \end{align*}
    for $n\ge 0$. Passing to the inverse limit, and composing with $R = \Gamma(U,\clO_U)\to \clO_{U,0}\to\clO_{U,0}^\wedge$, results in a $k$-algebra map
    \begin{align*}
        R\to k[[t]],
    \end{align*}
    whose kernel is a prime ideal $\fp$ contained in $\fm$. As each node of $\sC$ is non-trivially deformed by the formal smoothing, the explicit description of the versal deformation of a node (Example \ref{exa:versal-def-of-node}) shows that, for $1\le i\le N$, the image of the smoothing parameter $f_i$ under the map $R\to k[t]/(t^{n+1})$ is non-zero for some $n\ge 1$ (depending on $i$). In particular, $f_i\not\in\fp$, for $1\le i\le N$. As $\fp$ is prime, we also get $\prod_{1\le i\le N}f_i\not\in\fp$. Thus, the (not necessarily closed) point $\eta\in U$ corresponding to $\fp\subset R$, lies in $U^\circ$. We conclude by observing that $\eta$, and therefore $U^\circ$, contains $0$ in its closure.\footnote{This can be rephrased, avoiding any mention of non-closed points, as follows. Since we have $\prod_{1\le i\le N}f_i\not\in\fp$, the closed subvariety $\Spec(R/\fp)\subset U$ meets the open subscheme $U^\circ$ and contains $0$.}
\end{proof}

\begin{remark}\label{rem:actual-implies-formal}
    For completeness, we mention without proof that smoothability of a stable map implies its formal smoothability.
\end{remark}

\begin{remark}
    Due to Lemma \ref{lem:formal-implies-actual}, to show eventual smoothability (in the sense of Definition \ref{def:eventual-smoothability}) of a stable map $f:\sC\to X$, it is sufficient to produce an embedding $X\subset\bP^N$ so that the induced stable map $\sC\to X\subset\bP^N$ is formally smoothable.
\end{remark}

%%%%%%%%%%%%%%%%%%%%%%%%%%%%%%%%%%%%%%%%%%%%%%%%%%%%%%%%%%%%%%%
\subsection{Pinching a prestable curve}\label{subsec:pinching-a-curve}

Let $\sC$ be a prestable curve and let $C\subsetneq\sC$ be a union of some of the irreducible components of $\sC$. In our applications, $C$ will consist of the irreducible components of $\sC$ on which a given stable map $f:\sC\to X$ is constant. Suppose $C$ has $r\ge 1$ connected components, which we enumerate as $C_1,\ldots,C_r$.

\begin{Definition}[Pinching]\label{def:pinching}
    A \emph{pinching}\footnote{This terminology is borrowed from \cite[Example 29.10.3]{Har-DT}, where it is used to describe a construction of non-smoothable reduced irreducible curves appearing in \cite{mumford-path-4}.} of $\sC$ at $C$ is a proper reduced connected curve $\sS$ along with a surjective morphism $\nu:\sC\to\sS$ with the following properties.
    \begin{enumerate}[\normalfont(i)]
        \item There are distinct closed points $s_1,\ldots,s_r\in\sS$ such that the set-theoretic inverse image of $\{s_i\}$ under $\nu$ is $C_i$, for $1\le i\le r$, and $\nu$ induces an isomorphism 
        \begin{align*}
            \sC\setminus C\xrightarrow{\simeq}\sS\setminus\{s_1,\ldots,s_r\}.
        \end{align*}
        \item\label{cond:genus-conserved} For $1\le i\le r$, the vector space dimension of the stalk of $\nu_*\clO_\sC/\clO_\sS$ at $s_i$ coincides with the arithmetic genus $\dim H^1(C_i,\clO_{C_i})$ of $C_i$.
    \end{enumerate}
    In this situation, we call $\nu$ the \emph{pinching morphism}, and $\sS$ the \emph{pinched curve}.

       \begin{figure}[ht]
    \centering
\includegraphics[width=5cm]{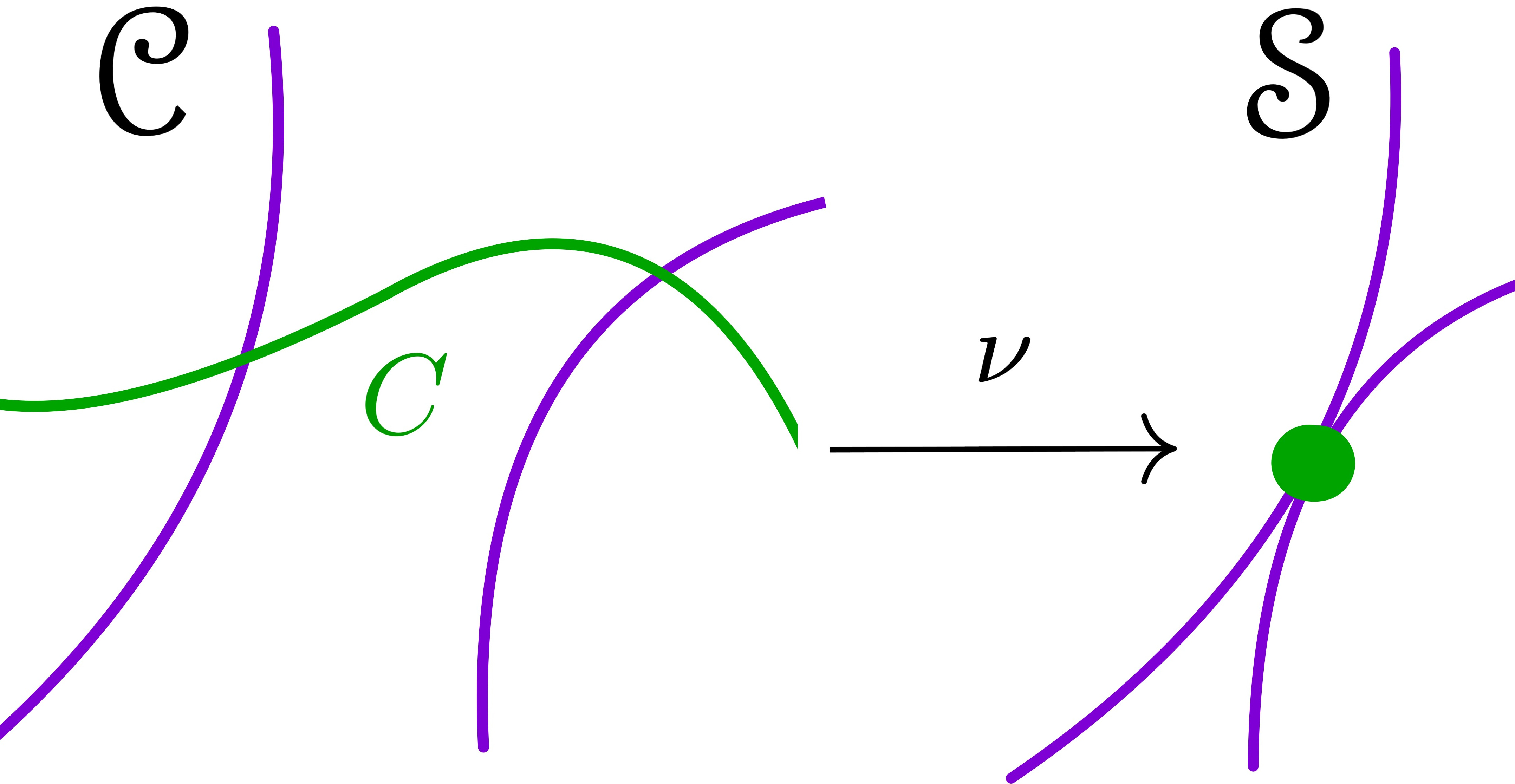}
    \caption{Pinching morphism (Definition \ref{def:pinching}). The first-order tangency of the two branches of $\sS$ means that $C$ is of genus $1$.}
    \label{fig: pinching}
\end{figure}
\end{Definition}

\begin{remark}
    In the situation of Definition \ref{def:pinching}, we have
    \begin{align}\label{eqn:pinching-genus-preserved}
        \chi(\clO_\sC) = \chi(R\nu_*\clO_\sC) = \chi(\clO_\sS) + \chi(\nu_*\clO_\sC/\clO_\sS) - \chi(R^1\nu_*\clO_\sC) = \chi(\clO_\sS),
    \end{align} 
  where $R\nu_*$ denotes the derived pushforward of coherent sheaves under the proper morphism $\nu$. The last equality in \eqref{eqn:pinching-genus-preserved} comes from condition \eqref{cond:genus-conserved} by adding up the $H^0$ contributions at $s_1,\ldots,s_r$ of the $0$-dimensional sheaves $\nu_*\clO_\sC/\clO_\sS$ and $R^1\nu_*\clO_\sC$. It follows that the curves $\sC$ and $\sS$ have the same arithmetic genus. (See Figure \ref{fig: pinching} for an example.)
\end{remark}

\begin{Example}[Genus $0$ pinching]\label{exa:genus-0-pinching}
    Following the notation of Definition \ref{def:pinching}, fix $1\le i\le r$. Denote the closure of $\sC\setminus C_i$ by $\Sigma_i\subset\sC$, and let $q_1,\ldots,q_n$ be an enumeration of $C_i\cap\Sigma_i$. Both $C_i$ and $\Sigma_i$ are non-singular at $q_j$ for $1\le j\le n$. We determine the structure of $\sS$ at $s_i$ when $H^1(C_i,\clO_{C_i}) = 0$, assuming that $q_1,\ldots,q_n$ lie in pairwise distinct irreducible components of $\Sigma_i$. Note that, since $H^1(C_i,\clO_{C_i}) = 0$, the curve $C_i$ is a tree of $\bP^1$ components.

    Restriction of regular functions gives an injective ring map 
    \begin{align*}
        (\nu_*\clO_\sC)_{s_i}\to\prod_{1\le j\le n}\clO_{\Sigma_i,q_j}.
    \end{align*}
    Its image is the sub-ring $\prod_{1\le i\le n}'\clO_{\Sigma_i,p_j}$ consisting of tuples $(h_j)_{1\le j\le n}$ such that the value $h_j(q_j)\in k$ is independent of $j$. Choosing local coordinates $x_j\in \clO_{\Sigma_i,q_j}$, this shows that $(\nu_*\clO_\sC)_{s_i}$ is a local $k$-algebra whose formal completion is identified with
    \begin{align*}
        k[[x_1,\ldots,x_n]]/(x_jx_{j'}\,:\,1\le j<j'\le n).
    \end{align*}
    From condition \eqref{cond:genus-conserved} of Definition \ref{def:pinching}, we find that the inclusion $\clO_{\sS,s_i} \subset (\nu_*\clO_\sC)_{s_i}$ is an equality. Thus, the pinched curve $\sS$, \'etale locally at $s_i$, is given by the union of the coordinate axes in $\bA^n = \Spec k[x_1,\ldots,x_n]$.
\begin{figure}[ht]
 \subcaptionbox*{}[.31\linewidth]{%
    \includegraphics[width=\linewidth]{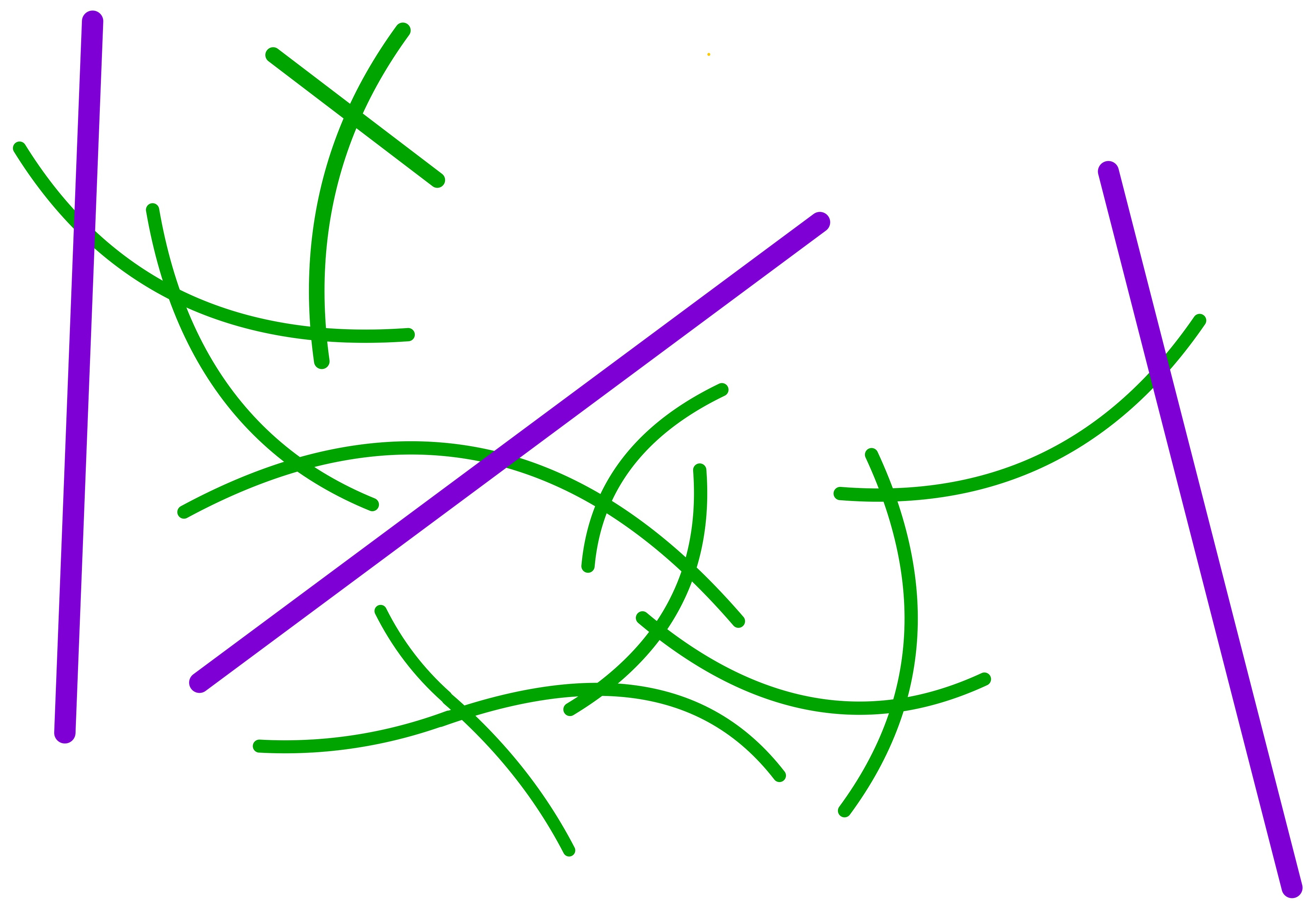}%
  }%
  \hskip15ex
  \subcaptionbox*{}[.24\linewidth]{%
    \includegraphics[width=\linewidth]{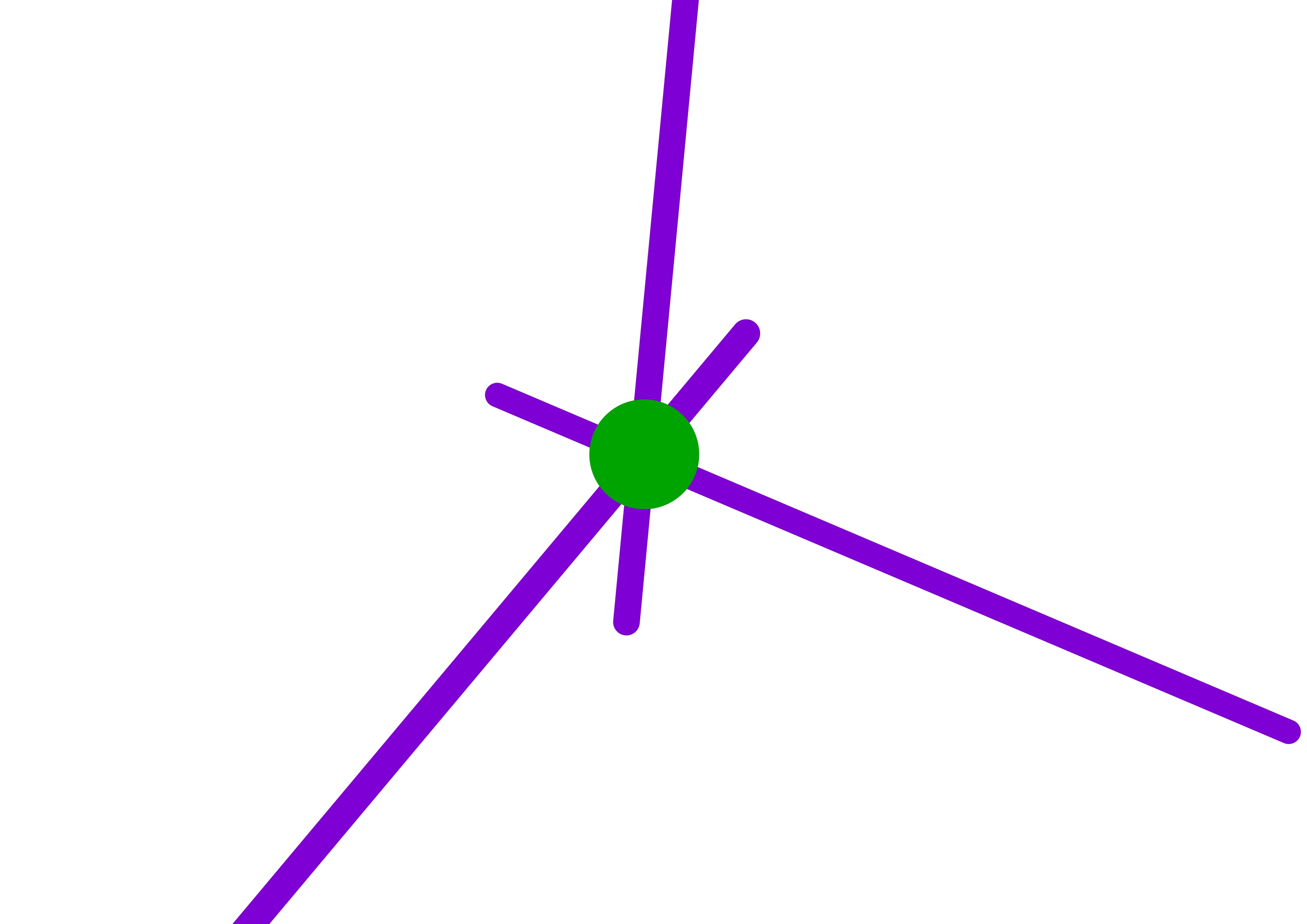}%
  }
  \caption{Depiction of $\sC$ (left) and $\sS$ (right), from Example \ref{exa:genus-0-pinching}, in the case $n=3$. The curve $\Sigma_i$ and its image in $\sS$ are in purple, while the curve $C_i$ and its image $s_i$ are in green. Note that $\sC$ can be embedded into a non-singular surface while $\sS$ cannot (its embedding dimension at $s_i$ is $3$).}
  \label{fig: pinched/un-pinched}
\end{figure}
\end{Example}

%%%%%%%%%%%%%%%%%%%%%%%%%%%%%%%%%%%%%%%%%%%%%%%%%%%%%%%%%%%%%%
\subsection{Local smoothability of a pinching}\label{subsec:smoothability-of-pinching}

Fix a pinching $\nu:\sC\to\sS$. We continue with the notation in Definition \ref{def:pinching}. In particular, the points $s_1,\ldots,s_r\in\sS$ are the respective images of $C_1,\ldots,C_r$ under $\nu$.

\begin{notation}\label{not:zariski-nbd-of-pinched-point}
    For $1\le i\le r$, let $U_i\subset\sS$ be an affine open neighborhood of $s_i$ such that $U_i\setminus\{s_i\}$ is non-singular. Define $\sC_{U_i} := \nu^{-1}(U_i)$. Let $\tilde U_i$ denote the closure of $\sC_{U_i}\setminus C_i$ inside $\sC_{U_i}$. Then, $\nu|_{\tilde U_i}:\tilde U_i\to U_i$ is seen to be the normalization of $U_i$.
\end{notation}

We would like to consider smoothings of the curves $\sC_{U_i}$ and $U_i$, which are compatible (in a sense to be made precise in Definition \ref{def:locally-smoothable-pinching}). We will see in Theorem \ref{thm:local-criterion-for-smoothability} that the existence of such compatible smoothings is related to eventual smoothability for stable maps.

\begin{remark}
    For technical reasons, it will be useful to carry out the discussion after replacing $U_i$ by an analytically isomorphic curve singularity. With this in mind, we introduce the following notation.
\end{remark}

\begin{notation}\label{not:etale-nbd-of-pinched-point}
    For $1\le i\le r$, let $V_i$ be a reduced affine curve with a closed point $s'_i$ such that $V_i\setminus\{s_i'\}$ is non-singular. Let $\varphi_i:U_i\to V_i$ be an \'etale morphism with $\varphi_i^{-1}(s_i') = \{s_i\}$. Let $\eta_i: \tilde V_i\to V_i$ be the normalization of $V_i$. Then, the \'etale morphism $\tilde\varphi_i:\tilde U_i\to\tilde V_i$ (induced by $\varphi_i$) restricts to a bijection 
    \begin{align}\label{eqn:etale-normalization-bijection}
        C_i\cap\tilde U_i = (\nu|_{\tilde U_i})^{-1}(s_i)\xrightarrow{\simeq}\eta^{-1}_i(s'_i).
    \end{align}
    Define the curve $\sC_{U_i,\varphi_i}$, with only ordinary nodes as singularities, as follows. Take the disjoint union $\tilde V_i\sqcup C_i$, and glue each point of $\eta_i^{-1}(s_i')\subset\tilde V_i$ to the corresponding point of $C_i\cap\tilde U_i\subset C_i$, under the bijection given by \eqref{eqn:etale-normalization-bijection}, to obtain $\sC_{U_i,\varphi_i}$ (see Figure \ref{fig:etale-local-pinching}). There is a natural morphism
    \begin{align}\label{eqn:etale-transfer-pinching}
        \nu_{U_i,\varphi_i}:\sC_{U_i,\varphi_i}\to V_i    
    \end{align}
    defined by $\eta_i$ on $\tilde V_i\subset \sC_{U_i,\varphi_i}$, and the constant morphism $s'_i$ on $C_i\subset \sC_{U_i,\varphi_i}$.
    
    There is an \'etale morphism $\sC_{U_i}\to\sC_{U_i,\varphi_i}$ given by combining the identity morphism $C_i\to C_i$ and $\tilde\varphi_i:\tilde U_i\to\tilde V_i$. By a minor abuse of notation, we denote this also by 
    \begin{align*}
        \tilde\varphi_i:\sC_{U_i}\to\sC_{U_i,\varphi_i}.
    \end{align*}
\end{notation}

\begin{remark}[\'Etale local version of pinching]\label{rem:etale-base-change-pinching}
    We have the following commutative fibre product diagram relating the various morphisms introduced in Notation \ref{not:etale-nbd-of-pinched-point}.
    \begin{equation}\label{eqn:fibre-product-pinching}
        \begin{tikzcd}
            \sC_{U_i} \arrow[r,"\tilde\varphi_i"] \arrow[d, "\nu"] & \sC_{U_i,\varphi_i} \arrow[d, "\nu_{U_i,\varphi_i}"] \\
            U_i\arrow[r, "\varphi_i"] & V_i
        \end{tikzcd}
    \end{equation}

         \begin{figure}[ht]
    \centering
\includegraphics[width=11cm]{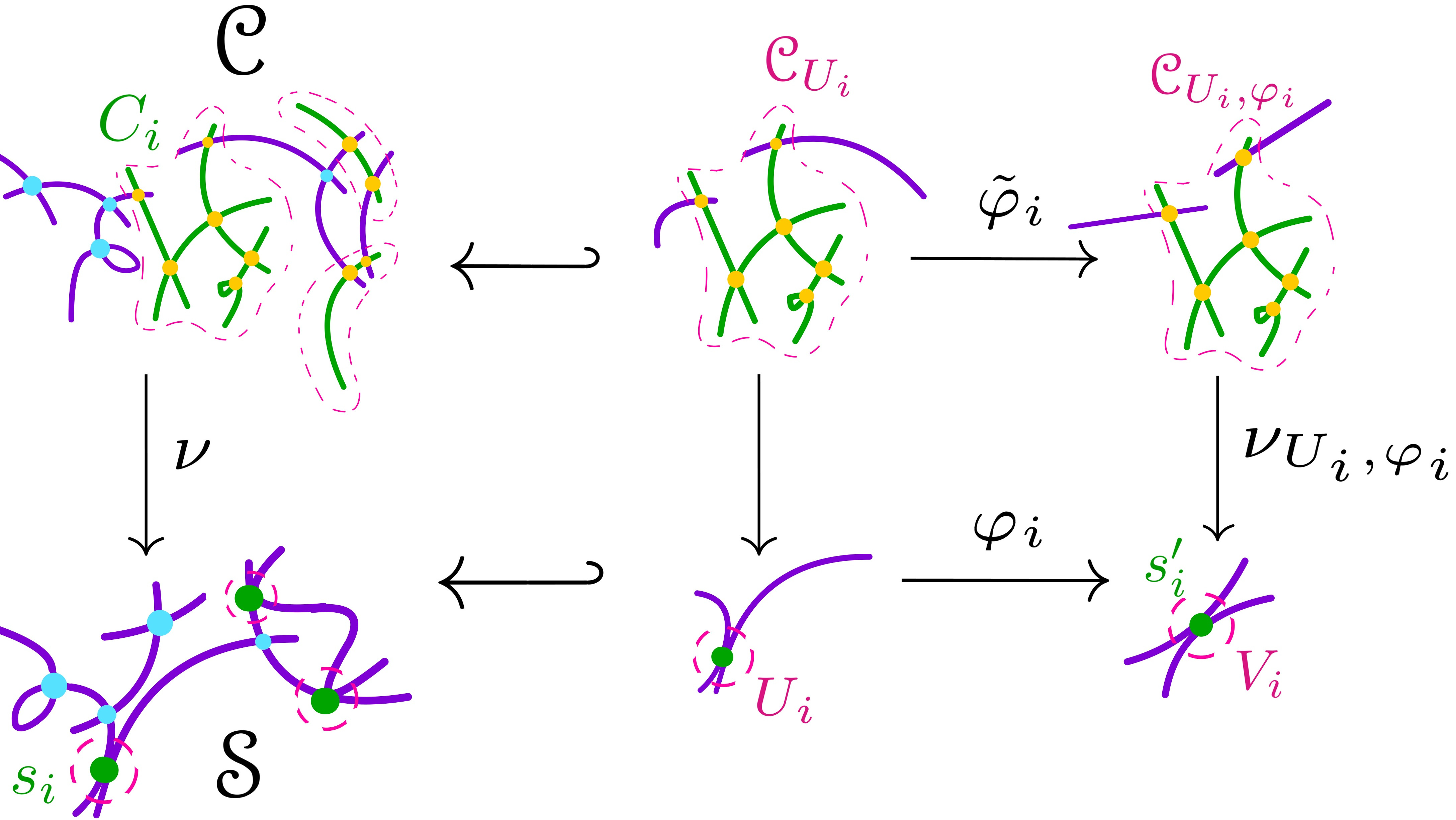}
    \caption{Depiction of the objects appearing in Remark \ref{rem:etale-base-change-pinching}. In the rightmost column, corresponding to $\nu_{U_i,\varphi_i}$, $C_i$ (top) and $s_i'$ (bottom) are shown in green while $\tilde V_i$ (top) and $V_i$ (bottom) are shown in purple.}
    \label{fig:etale-local-pinching}
\end{figure}
\end{remark}

We now introduce a notion of smoothability for a pinching.

\begin{Definition}[Local smoothability of pinching]\label{def:locally-smoothable-pinching}
    Fix $1\le i\le r$. We say that the pinching $\nu:\sC\to\sS$ is \emph{smoothable near $s_i$}, if there exist an affine open neighborhood $U_i$ of $s_i$, and an \'etale morphism $\varphi_i:(U_i,s_i)\to (V_i,s_i')$ such that the following holds, where we use Notations \ref{not:zariski-nbd-of-pinched-point} and \ref{not:etale-nbd-of-pinched-point} established above.

   First, we have the following data:
    
    \begin{enumerate}[\normalfont(i)]
        \item a non-singular affine curve $B_i$ with a closed point $0\in B_i$,
        \item a flat family $\clX_i\to B_i$ with fibre $\sC_{U_i,\varphi_i}$ over $0\in B_i$,
        \item a flat family $\clY_i\to B_i$ with fibre $V_i$ over $0\in B_i$, and
        \item  a proper morphism $\psi_i:\clX_i\to\clY_i$ of $B_i$-schemes, restricting to $\nu_{U_i,\varphi_i}:\sC_{U_i,\varphi_i}\to V_i$ over $0\in B_i$.
    \end{enumerate}

    Second, we require the following compatibility condition for this data. The morphism $\clY_i\to B_i$ is smooth over $B_i\setminus\{0\}$, and $\psi_i$ induces an isomorphism 
    \begin{align}\label{eqn:local-smoothable-punctured-iso}
        \clX_i\setminus C_i\xrightarrow{\simeq}\clY_i\setminus\{s_i'\}.
    \end{align}
    If this condition holds for each $1\le i\le r$, then we say that the pinching $\nu:\sC\to\sS$ is \emph{locally smoothable}.
\end{Definition}

\begin{remark}
   Intuitively, the compatibility condition says that if we smooth $V_i$ using $\clY_i/B_i$ then, by semistable reduction, we can convert this into the smoothing $\clX_i/B_i$ of $\sC_{U_i,\varphi_i}$ (with the curve $C_i$ appearing in place of the point $s_i'$). From this point of view, condition \eqref{cond:genus-conserved} of Definition \ref{def:pinching} intuitively says that $s_i'\in V_i$ provides the same local contribution to the arithmetic genus as $C_i\subset\sC_{U_i,\varphi_i}$.
\end{remark}

The following definition is the analogue of Definition \ref{def:locally-smoothable-pinching} from the perspective of formal deformations.

\begin{Definition}[Local formal smoothability of pinching]\label{def:locally-formally-smoothable-pinching}
    Fix $1\le i\le r$. We say that the pinching $\nu:\sC\to\sS$ is \emph{formally smoothable near $s_i$}, if there exist an affine open neighborhood $U_i$ of $s_i$, and an \'etale morphism $\varphi_i:(U_i,s_i)\to (V_i,s_i')$ such that the following holds, where we use Notations \ref{not:zariski-nbd-of-pinched-point} and \ref{not:etale-nbd-of-pinched-point} established above.

    First, we have the following data:
    \begin{enumerate}[\normalfont(i)]
        \item a formal deformation $\{\clX_{i,n}\to\bD_n\}_{n\ge 0}$ of $\sC_{U_i,\varphi_i}$,
        \item a formal deformation $\{\clY_{i,n}\to\bD_n\}_{n\ge 0}$ of $V_i$, and
        \item a compatible collection of $\bD_n$-morphisms $\psi_{i,n}:\clX_{i,n}\to\clY_{i,n}$ for $n\ge 0$, restricting to $\nu_{U_i,\varphi_i}:\sC_{U_i,\varphi_i}\to V_i$ for $n=0$.
    \end{enumerate}
    Second, we require the following compatibility condition for this data. The formal deformation $\{\clX_{i,n}\to\bD_n\}_{n\ge 0}$ induces a non-trivial deformation of each node of $\sC_{U_i,\varphi_i}$, and $\psi_{i,n}$ induces an isomorphism
    \begin{align}\label{eqn:local-formal-smoothable-punctured-iso}
        \clX_{i,n}\setminus C_i\to\clY_{i,n}\setminus\{s_i'\}
    \end{align}
    for $n\ge 0$. If this condition holds for each $1\le i\le r$, then we say that the pinching $\nu:\sC\to\sS$ is \emph{locally formally smoothable}.
\end{Definition}

The next lemma gives the relation between Definitions \ref{def:locally-smoothable-pinching} and \ref{def:locally-formally-smoothable-pinching}, which is also shown in Diagram \eqref{eqn:big-diagram}.

\begin{Lemma}\label{lem:actual-implies-formal-for-pinching}
    Assume we are in the situation of Definition \ref{def:locally-smoothable-pinching}. If the pinching $\nu:\sC\to\sS$ is smoothable near $s_i$, then it is formally smoothable near $s_i$.
\end{Lemma}
\begin{proof}
    Fix choices of $\varphi_i:(U_i,s_i)\to(V_i,s_i')$, $0\in B_i$, and $\psi_i:\clX_i\to\clY_i$ over $B_i$ (as in Definition \ref{def:locally-smoothable-pinching}). Fix an identification $\clO_{B_i,0}^\wedge = k[[t]]$, by choosing a local coordinate in $\clO_{B_i,0}$. This gives $\bD = \Spec k[[t]] = \Spec \clO_{B_i,0}^\wedge$, and $\bD_n = \Spec k[t]/(t^{n+1})\subset\bD$.

    For $n \geq 0$, the inclusions $\bD_n\subset\bD = \Spec\clO_{B_i,0}^\wedge$ define the collection of flat families
    \begin{align*}
        \clY_{i,n} &:= \clY_i\times_{B_i}\bD_n\to\bD_n, \\
        \clX_{i,n} &:= \clX_i\times_{B_i}\bD_n\to\bD_n,
    \end{align*}
    and morphisms 
    \begin{align*}
        \psi_{i,n} := \psi_i\times_{B_i}\bD_n : \clX_{i,n}\to \clY_{i,n}.
    \end{align*} 
    For $n=0$, $\psi_{i,n}$ reduces to the morphism $\nu_{U_i,\varphi_i}:\sC_{U_i,\varphi_i}\to V_i$ from \eqref{eqn:etale-transfer-pinching}. The relation between these families is summarized in the following commutative diagram.

\begin{equation}\label{eqn:big-diagram}
\begin{tikzcd}
    [row sep=scriptsize, column sep=scriptsize]
    & V_i\arrow[dd] \arrow[rr,hook]  & & \clY_{i,1} \arrow[dd]  \arrow[r,hook]&\cdots \arrow[r,hook]  &  \clY_{i,n}\arrow[dd] \arrow[r,hook]&\cdots \arrow[r,hook]  & \clY_{i}  \arrow[dd] \\
    \sC_{U_i,\varphi_i} \arrow[ur,"\nu_{U_i,\varphi_i}"]\arrow[rr, crossing over,hook]  \arrow[dd]  & & \clX_{i,1} \arrow[ur,"\psi_{i,1}"]\arrow[r,hook]&\cdots \arrow[r,hook] &\clX_{i,n}  \arrow[ur,"\psi_{i,n}"] \arrow[dd] \arrow[r,hook]&\cdots \arrow[r,hook] &\clX_{i} \arrow[dd]\arrow[ur,"\psi_{i}"] \\
    &\{0\}   \arrow[rr,hook] & & \bD_1\arrow[r,hook]&\cdots \arrow[r,hook]  & \bD_{n}\arrow[r,hook]&\cdots \arrow[r,hook] &B_i\\
   \{0\}  \arrow[ur,equal] \arrow[rr,hook] & & \bD_1  \arrow[ur,equal] \arrow[from=uu, crossing over]\arrow[r,hook]&\cdots \arrow[r,hook]& \bD_{n}\arrow[ur,equal]\arrow[r,hook]&\cdots \arrow[r,hook]  & B_i\arrow[ur,equal]\\
\end{tikzcd}
\end{equation}

    The isomorphisms \eqref{eqn:local-formal-smoothable-punctured-iso} follow from \eqref{eqn:local-smoothable-punctured-iso}. To show that $\{\clX_{i,n}\to\bD_n\}_{n\ge 0}$ non-trivially deforms every node of $\sC_{U_i,\varphi_i}$, it suffices to check that $\clX_i\to B_i$ is smooth over $B_i\setminus\{0\}$. Definition \ref{def:locally-smoothable-pinching} says that $\clY_i\to B_i$ is smooth over $B_i\setminus\{0\}$, and that $\psi_i$ identifies $\clX_i|_{B_i\setminus\{0\}}$ with $\clY_i|_{B_i\setminus\{0\}}$ and thus, we are done.
\end{proof}

%%%%%%%%%%%%%%%%%%%%%%%%%%%%%%%%%%%%%%%%%%%%%%%%%%%%%%%%%%%%%%%
\subsection{Behavior near ghost components}\label{subsec:behavior-near-ghost-component}

Let $f:\sC\to X$ be a non-constant stable map with the target $X$ being a projective variety.

\begin{Definition}[Ghost]\label{def:ghost-components}
    The \emph{ghost sub-curve} of $(\sC,f)$, denoted by $C$, is defined to be the union of those irreducible components of $\sC$ on which $f$ restricts to a constant morphism. Each connected component of $C$ is called a \emph{ghost component} of $(\sC,f)$, or  simply a \emph{ghost}, when $(\sC,f)$ is clear from context.
\end{Definition}

We now introduce another notion of smoothability for stable maps, whose relation to eventual smoothability will be established in the next subsection.

\begin{Definition}[Local formal smoothability of stable map]\label{def:local-smoothability-stable-map}
    The non-constant stable map $f:\sC\to X$ is called \emph{locally formally smoothable} if the morphism $f$ factors through some locally formally smoothable pinching $\nu:\sC\to\sS$ of $\sC$ at its ghost sub-curve.
\end{Definition}

\begin{remark}\label{rem:no-ghost-vacuous}
    Note that the condition of local formal smoothability of a stable map is vacuously satisfied when there are no ghosts.
\end{remark}

%%%%%%%%%%%%%%%%%%%%%%%%%%%%%%%%%%%%%%%%%%%%%%%%%%%%%%%%%%%
\subsection{Local criterion for eventual smoothability}\label{subsec:local-criterion-eventual-smoothability}

The following result, stated as Theorem \ref{thm:intro-local-criterion} in \textsection\ref{subsec:summary-of-results}, explains the relation between Definition \ref{def:local-smoothability-stable-map} and eventual smoothability (Definition \ref{def:eventual-smoothability}), and is an application of the local-to-global principle stated in \textsection\ref{subsec:local-to-global}.

\begin{Theorem}[Local criterion for eventual smoothability]\label{thm:local-criterion-for-smoothability}
    If the non-constant stable map $f:\sC\to X$ is locally formally smoothable, then it is eventually smoothable.
\end{Theorem}

The remainder of this section is devoted to the proof of this theorem. To clarify the argument, we will divide it into five parts.

%%%%%%%%%%%%%%%%%%%%%%%%%%%%%%%%%%%%%%%%%%%%%%%%%%%%%%%%%%%%%%%%%%%
\subsubsection{Set up for the proof}\label{subsubsec:local-criterion-proof-1}

We fix the following data, which will be used in the proof.

\begin{enumerate}[\normalfont(i)]
    \item The ghost sub-curve $C=\bigsqcup_{i=1}^{r}C_i$  of $\sC$, where $C_i$ are its connected components.
    \item A locally formally smoothable pinching, $\nu:\sC\to\sS$, of $\sC$ at $C$, and a factorization $f = g\circ\nu$. Denote the image of $C_i$ under $\nu$ by $s_i\in\sS$, for $1\le i\le r$.
    \begin{equation*}
        \begin{tikzcd}
         C_i \arrow[d,mapsto] &     \sC \arrow[d,"\nu"] \arrow[dr,"f"] & \\
           s_i& \sS \arrow[r,"g"] & X
        \end{tikzcd}
    \end{equation*}

    \item For $1\le i\le r$, choices of $\varphi_i:(U_i,s_i)\to(V_i,s_i')$, and $\{\psi_{i,n}:\clX_{i,n}\to\clY_{i,n}\}_{n\ge 0}$, as in Definition \ref{def:locally-formally-smoothable-pinching}.
    \item An enumeration $p_1,\ldots,p_m$ of the nodes of $\sC\setminus C = \sS\setminus\{s_1,\ldots,s_r\}$.
    \item For $1\le j\le m$, an affine open neighborhood $U_{r+j}\subset\sC\setminus C$ of $p_j$, with $U_{r+j}\setminus\{p_j\}$ being non-singular. 
\end{enumerate}

We will refer to the nodes of $\sC$ that are contained in $C$ as \emph{internal} nodes. Similarly, we refer to the nodes of $\sC\setminus C$ as \emph{external} nodes. See Figure \ref{fig: setup}.

    \begin{figure}[ht]
    \centering
\includegraphics[width=8.5cm]{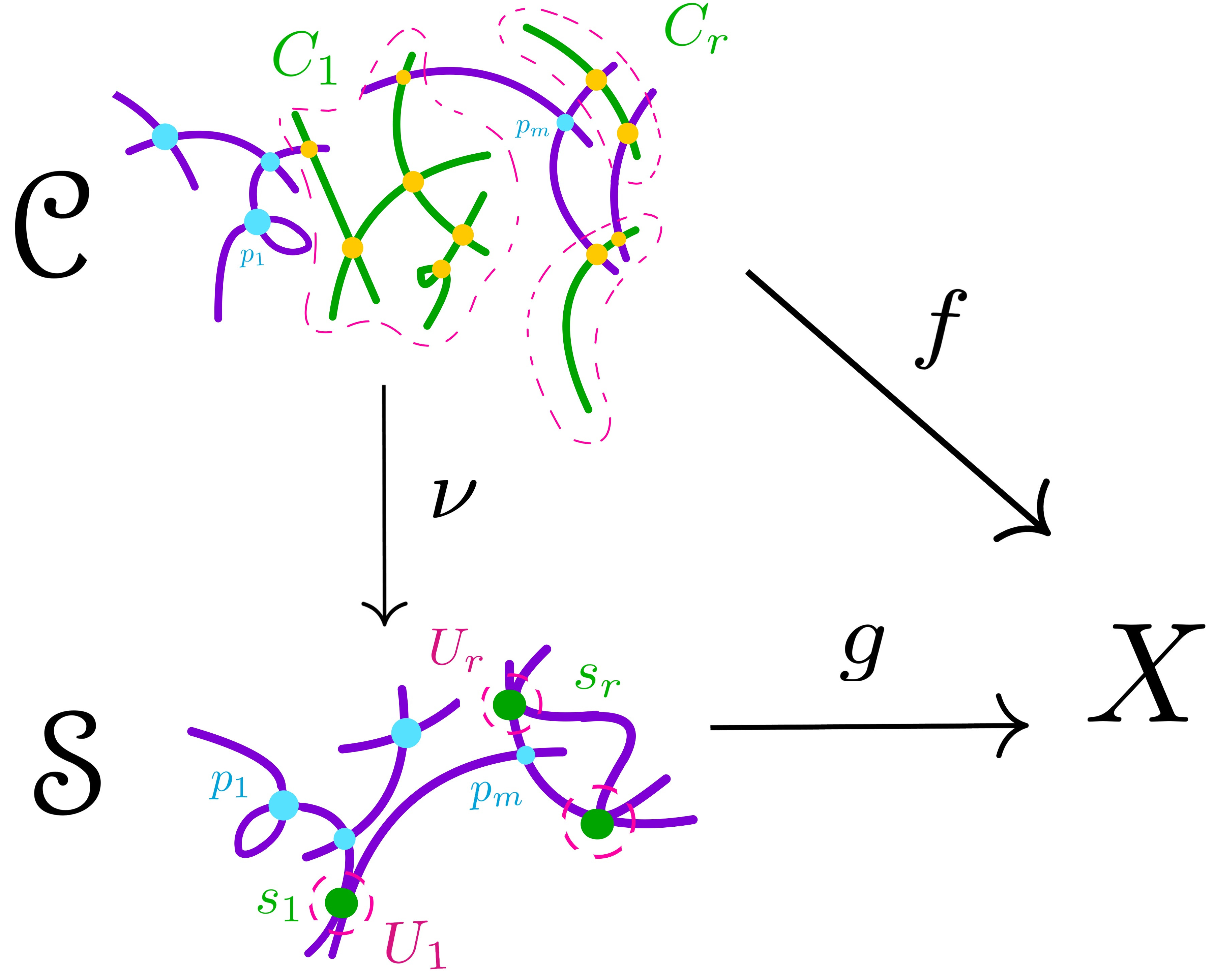}
    \caption{Depiction of the set up for the proof of Theorem \ref{thm:local-criterion-for-smoothability}. The ghost sub-curve $C$ is shown in green, while $\sC\setminus C = \sS\setminus\{s_1,\ldots,s_r\}$ is shown in purple. Internal (resp. external) nodes are shown in yellow (resp. blue).}
    \label{fig: setup}
\end{figure}

%%%%%%%%%%%%%%%%%%%%%%%%%%%%%%%%%%%%%%%%%%%%%%%%%%%%%%%%%%%%%%%%%%%
\subsubsection{Local deformation at the internal nodes}\label{subsubsec:local-criterion-proof-2} 

Fix $1\le i\le r$.
Consider the \'etale morphism $\varphi_i:U_i\to V_i$, and note that $V_i\subset\clY_{i,n}$ is the inclusion of a closed subscheme defined by a nilpotent ideal, for each $n\ge 1$. Applying Theorem \ref{thm:top-inv-et-top} (topological invariance of \'etale topology), we obtain a sequence of \'etale morphisms
\begin{align*}
    \varphi_{i,n}:\widetilde\clY_{i,n}\to\clY_{i,n},
\end{align*}
for $n\ge 0$, which successively extend\footnote{Explicitly, this means that we have $U_i = \widetilde\clY_{i,0}\subset\cdots\subset\widetilde\clY_{i,n}\subset\widetilde\clY_{i,n+1}\subset\cdots$ and that the restriction of $\varphi_{i,n+1}$ to $\widetilde\clY_{i,n}$ is the morphism $\varphi_{i,n}$, for $n\ge 0$. Moreover, we have $\varphi_{i,0} = \varphi_i:U_i\to V_i$. We also use the phrase \emph{successive extension} to indicate analogous situations appearing later in this proof.} the \'etale morphism $\varphi_i:U_i\to V_i$. Being a composition of flat morphisms, $\widetilde\clY_{i,n}\to\clY_{i,n}\to\bD_n$ is also flat. In particular, $\{\widetilde\clY_{i,n}\to\bD_n\}_{n\ge 0}$ defines a formal deformation of $U_i$.
    
For each $n\ge 0$, define the scheme $\widetilde\clX_{i,n}$ (along with the morphisms $\widetilde\varphi_{i,n}$ and $\widetilde\psi_{i,n}$), by the following commutative fibre product diagram.
\begin{equation}\label{eqn:fibre-prod-for-etale-lift}
    \begin{tikzcd}
        \widetilde\clX_{i,n} \arrow[r,"\widetilde\varphi_{i,n}"] \arrow[d, "\widetilde\psi_{i,n}"] & \clX_{i,n} \arrow[d, "\psi_{i,n}"] \\
            \widetilde\clY_{i,n}\arrow[r, "\varphi_{i,n}"] & \clY_{i,n}
    \end{tikzcd}
\end{equation}
Since $\widetilde\varphi_{i,n}$ is obtained by base change from the \'etale morphism $\varphi_{i,n}$, it is \'etale. As above, the composition $\widetilde\clX_{i,n}\to\clX_{i,n}\to\bD_n$ is flat. Using Remark \ref{rem:etale-base-change-pinching}, $\{\widetilde\clX_{i,n}\to\bD_n\}_{n\ge 0}$ defines a formal deformation of $\sC_{U_i}$ and the morphisms $\widetilde\varphi_{i,n}$ successively extend the morphism $\tilde\varphi_i:\sC_{U_i}\to\sC_{U_i,\varphi_i}$. The morphisms $\widetilde\psi_{i,n}$ commute with the projections to $\bD_n$. Moreover, by Remark \ref{rem:etale-base-change-pinching}, the morphisms $\widetilde\psi_{i,n}$ successively extend $\nu:\sC_{U_i}\to U_i$ (the restriction of the pinching morphism $\nu:\sC\to\sS$ over $U_i$). We may summarize the argument so far by saying that Diagram \eqref{eqn:fibre-product-pinching} is successively extended by Diagrams \eqref{eqn:fibre-prod-for-etale-lift} for $n\ge 0$. This is also shown in Diagram \eqref{eqn:successive-extn-etale-lift} below.

\begin{equation}\label{eqn:successive-extn-etale-lift}
\begin{tikzcd}
    [row sep=scriptsize, column sep=scriptsize]
    & \sC_{U_i,\varphi_i}\arrow[dd, "\nu_{U_i,\varphi_i}" near start] \arrow[r,hook]&\cdots \arrow[r,hook]  & \clX_{i,n} \arrow[dd, crossing over, "\psi_{i,n}" near start] \arrow[r,hook] &\cdots\\
    \sC_{U_i} \arrow[ur,"\tilde\varphi_i"] \arrow[dd, "\nu" near start]  \arrow[r,hook]&\cdots \arrow[r,hook] & \widetilde\clX_{i,n} \arrow[ur,"\widetilde\varphi_{i,n}"] \arrow[r,hook] &\cdots\\
    & V_i \arrow[r,hook]&\cdots \arrow[r,hook]  & \clY_{i,n} \arrow[r,hook] &\cdots\\
    U_i \arrow[ur,"\varphi_i"] \arrow[r,hook]&\cdots \arrow[r,hook]  & \widetilde\clY_{i,n}  \arrow[ur,"\varphi_{i,n}"] \arrow[from=uu, "\widetilde\psi_{i,n}" near start, crossing over] \arrow[r,hook] &\cdots  \\
\end{tikzcd}
\end{equation}
To understand the statement of the next lemma, note that the scheme $\widetilde\clX_{i,n}$ (resp. $\widetilde\clY_{i,n}$) has the same underlying topological space as the reduced scheme $\sC_{U_i}$ (resp. $U_i$), but a different structure sheaf, for $n\ge 1$.
    
\begin{Lemma}\label{lem:etale-lifts-of-XY-birational}
    For each $n\ge 0$, the morphism $\widetilde\psi_{i,n}$ induces an isomorphism 
    \begin{align*}
        \widetilde\clX_{i,n}\setminus C_i\xrightarrow{\simeq}\widetilde\clY_{i,n}\setminus \{s_i\}.
    \end{align*}
\end{Lemma}
\begin{proof}
Observe that we have $\varphi_{i,n}^{-1}(s_i') = \{s_i\}\subset\widetilde\clY_{i,n}$ and $\widetilde\varphi_{i,n}^{-1}(C_i) = C_i\subset\widetilde\clX_{i,n}$. Therefore, the following is a commutative fibre product diagram.
    \begin{equation}\label{eqn:fibre-prod-for-etale-lift-punctured}
        \begin{tikzcd}
            \widetilde\clX_{i,n}\setminus C_i \arrow[r,"\widetilde\varphi_{i,n}"] \arrow[d, "\widetilde\psi_{i,n}"] & \clX_{i,n}\setminus C_i \arrow[d, "\psi_{i,n}"] \\
            \widetilde\clY_{i,n}\setminus\{s_i\}\arrow[r, "\varphi_{i,n}"] & \clY_{i,n}\setminus\{s_i'\}
        \end{tikzcd}
    \end{equation}
The right vertical arrow in \eqref{eqn:fibre-prod-for-etale-lift-punctured} is an isomorphism, by Definition \ref{def:locally-formally-smoothable-pinching}.
It follows that the left vertical arrow in \eqref{eqn:fibre-prod-for-etale-lift-punctured} is also an isomorphism.
\end{proof}

The next lemma will be used to verify condition \eqref{cond:genus-conserved} of Definition \ref{def:formal-smoothability-of-map}.

\begin{Lemma}\label{lem:nodes-deformed-before-gluing}
    The formal deformation $\{\widetilde\clX_{i,n}\to\bD_n\}_{n\ge 0}$ induces a non-trivial deformation of each node of $\sC_{U_i}$.
\end{Lemma}
\begin{proof}
    The \'etale morphism $\tilde\varphi_i:\sC_{U_i}\to\sC_{U_i,\varphi_i}$ gives a bijection between the nodes of $\sC_{U_i}$ and $\sC_{U_i,\varphi_i}$. The $\bD_n$-morphisms $\widetilde\varphi_{i,n}:\widetilde\clX_{i,n}\to\clX_{i,n}$ are \'etale and successively extend $\tilde\varphi_i$. Thus, we are done once we recall that $\{\clX_{i,n}\to\bD_n\}_{n\ge 0}$ non-trivially deforms every node of $\sC_{U_i,\varphi_i}$ (by Definition \ref{def:locally-formally-smoothable-pinching}).
\end{proof}

%%%%%%%%%%%%%%%%%%%%%%%%%%%%%%%%%%%%%%%%%%%%%%%%%%%%%%%%%%%%%%%
\subsubsection{Local deformation at the external nodes}\label{subsubsec:local-criterion-proof-3}

Fix $1\le j\le m$, and let $U_{r+j}$ be the affine open neighborhood of the external node $p_j$ chosen in \textsection\ref{subsubsec:local-criterion-proof-1}. Fix a formal deformation 
\begin{align*}
    \{\widetilde\clY_{r+j,n}\to\bD_n\}_{n\ge 0}
\end{align*} 
of $U_{r+j}$, which non-trivially deforms the node $p_j$.

%%%%%%%%%%%%%%%%%%%%%%%%%%%%%%%%%%%%%%%%%%%%%%%%%%%%%%%%%%%%%%%%
\subsubsection{Global embedded deformation}\label{local-criterion-proof-4}

Choose an embedding $X\subset\bP^N$ and replace $X$ by $\bP^N$. The morphism $g:\sS\to\bP^N$ (obtained by factorizing $f:\sC\to X$ through the pinching morphism $\nu:\sC\to\sS$) is non-constant on each irreducible component of $\sS$. Therefore, $g$ is a finite morphism and the line bundle $L = g^*\clO_{\bP^N}(1)$ is ample on $\sS$. 

Using a sufficiently large $d$-uple embedding, we can embed $\bP^N$ into a larger projective space, which has the effect of replacing $L$ by a large power $L^{\otimes d}$. Thus, we may assume that $L$ is very ample, and that $H^1(\sS,L) = 0$, by Serre's vanishing theorem \cite[Theorem III.5.2]{Har77}. Since $L$ is very ample, we can construct (possibly after enlarging $N$ again) a projective embedding $\iota:\sS\to\bP^N$ defined by sections of $L$. Both $g:\sS\to\bP^N$ and $\iota:\sS\to\bP^N$ pull $\clO_{\bP^N}(1)$ back to $L$. Now, the morphism 
\begin{align}\label{eqn:new-embedding-of-pinched}
    (g,\iota):\sS\to\bP^N\times\bP^N    
\end{align}
is a closed embedding. From this point on, we will use the embedding \eqref{eqn:new-embedding-of-pinched} to regard $\sS\subset\bP^N\times\bP^N$ as a closed subscheme.

Since $H^1(\sS,L) = 0$, we may now apply the special case (Corollary \ref{cor:special-case-of-local-to-global}) of the local-to-global principle for deformations (Theorem \ref{thm: Local-to-global principle}) to $\sS\subset\bP^N\times\bP^N$. From this, we get an embedded formal deformation
\begin{align*}
    \{\sS_n\subset\bP^N\times\bP^N\times\bD_n\}_{n\ge 0}
\end{align*}
of the closed subscheme $\sS\subset\bP^N\times\bP^N$ such that the induced formal deformation of $U_i\subset\sS$ is isomorphic to $\{\widetilde\clY_{i,n}\to\bD_n\}_{n\ge 0}$ for $1\le i\le r+m$. We represent these isomorphisms by a sequence of open embeddings
\begin{align*}
    \{ \alpha_{i,n}:\widetilde\clY_{i,n}\hookrightarrow\sS_n\}_{n\ge 0},
\end{align*}
for $1\le i\le r+m$, which successively extend the open inclusion $U_i\subset\sS$.

%%%%%%%%%%%%%%%%%%%%%%%%%%%%%%%%%%%%%%%%%%%%%%%%%%%%%%%%%%%%%%%
\subsubsection{Completing the proof}\label{subsubsec:local-criterion-proof-5}

Consider any $n\ge 0$. The open subschemes $\alpha_{i,n}(\widetilde\clY_{i,n})$ for $1\le i\le r$ and $\sS_n\setminus\{s_1,\ldots,s_r\}$ together cover the scheme $\sS_n$. Thus, $\sS_n$ can be built by gluing the $r+1$ schemes given by $\widetilde\clY_{i,n}$ for $1\le i\le r$ and $\sS_n\setminus\{s_1,\ldots,s_r\}$, along open subschemes, using suitable transition functions. We may express this by saying that the natural $\bD_n$-morphism, from the coequalizer of the diagram
\begin{align}\label{eqn:colim-to-define-S}
    \bigsqcup_{1\le i\le r}\left(\widetilde\clY_{i,n}\setminus\{s_i'\}\right)\sqcup\bigsqcup_{1\le i<i'\le r}\left(\widetilde\clY_{i,n}\times_{\sS_n}\widetilde\clY_{i',n}\right)\rightrightarrows\left(\sS_n\setminus\{s_1,\ldots,s_r\}\right)\sqcup\bigsqcup_{1\le i\le r}\widetilde\clY_{i,n}
\end{align}
to the scheme $\sS_n$, is an isomorphism.

For $1\le i\le r$, notice that the point $s_i\in\sS_n$ does not occur on the \emph{overlap} of any two elements of this open cover of $\sS_n$. Thus, for $1\le i\le r$, we may replace $\widetilde\clY_{i,n}$ by $\widetilde\clX_{i,n}$ and glue this new collection of $r+1$ schemes using the \emph{same transition functions} (and the isomorphisms $\widetilde\clX_{i,n}\setminus C_i\xrightarrow{\simeq}\widetilde\clY_{i,n}\setminus\{s_i\}$ provided by Lemma \ref{lem:etale-lifts-of-XY-birational}) obtain a new scheme $\sC_n\to\bD_n$. As before, we may express this by saying that the $\bD_n$-scheme $\sC_n$ is defined to be the coequalizer of the diagram
\begin{align}\label{eqn:colim-to-define-C}
    \bigsqcup_{1\le i\le r}\left(\widetilde\clY_{i,n}\setminus\{s_i'\}\right)\sqcup\bigsqcup_{1\le i<i'\le r}\left(\widetilde\clY_{i,n}\times_{\sS_n}\widetilde\clY_{i',n}\right)\rightrightarrows\left(\sS_n\setminus\{s_1,\ldots,s_r\}\right)\sqcup\bigsqcup_{1\le i\le r}\widetilde\clX_{i,n}
\end{align}
of $\bD_n$-schemes and open embeddings. Note that the left side of \eqref{eqn:colim-to-define-C} is the same as that of \eqref{eqn:colim-to-define-S} and only the right side has been modified. We also obtain a morphism 
\begin{align*}
    \nu_n:\sC_n\to\sS_n
\end{align*}
given by gluing together the morphisms $\widetilde\psi_{i,n}:\widetilde\clX_{i,n}\to\widetilde\clY_{i,n}$ for $1\le i\le r$ with the identity morphism of $\sS_n\setminus\{s_1,\ldots,s_r\}$. Again, this can be seen arising as the coequalizer (or colimit) of the natural morphism from the diagram \eqref{eqn:colim-to-define-C} to the diagram \eqref{eqn:colim-to-define-S}. Since $\sC_n$ is covered by open subschemes which are flat over $\bD_n$, we conclude that $\sC_n\to\bD_n$ is flat.

Letting $n\ge 0$ vary, we obtain a formal deformation $\{\sC_n\to\bD_n\}_{n\ge 0}$ of $\sC$, along with morphisms $\{\nu_n:\sC_n\to\sS_n\}_{n\ge 0}$, which successively extend the pinching morphism $\nu:\sC\to\sS$. 

\begin{Lemma}\label{lem:nodes-deformed-after-gluing}
    The formal deformation $\{\sC_n\to\bD_n\}_{n\ge 0}$ induces a non-trivial deformation of each node of $\sC$.
\end{Lemma}
\begin{proof}
    For the internal nodes, this follows from Lemma \ref{lem:nodes-deformed-before-gluing}. For the external nodes, this follows from the choice of $\{\widetilde\clY_{r+j,n}\to\bD_n\}_{n\ge 0}$ in \textsection\ref{subsubsec:local-criterion-proof-3} for $1\le j\le m$.
\end{proof}

Now, we can complete the proof of Theorem \ref{thm:local-criterion-for-smoothability}. For $n\ge 0$, recalling that $\sS_n\subset\bP^N\times\bP^N\times\bD_n$ is a subscheme, we get a morphism
\begin{align*}
    f_n:\sC_n\to\bP^N,
\end{align*} 
by composing $\nu_n:\sC_n\to\sS_n$ with the projection from $\sS_n$ to the first $\bP^N$ coordinate. By construction, the morphisms $f_n:\sC_n\to\bP^N$ for $n\ge 0$ successively extend $f:\sC\to\bP^N$. We have therefore produced a formal smoothing of $(\sC,f)$ after replacing $X$ by $\bP^N$. An appeal to Lemma \ref{lem:formal-implies-actual} finishes the proof.\qed 
\section{Stable maps with model ghosts}\label{sec:local-models}

In this section, we introduce a class of stable maps which we show to be eventually smoothable. In \textsection\ref{subsec:model-sing-and-smoothing}, we describe a class of pinchings (called \emph{model pinchings}) and study their properties. In \textsection\ref{subsec:stable-maps-with-model-ghosts}, we define and study the class of \emph{stable maps with model ghosts}. The proof that these are eventually smoothable (Theorem \ref{thm:intro-model-ghost} in \textsection\ref{subsec:summary-of-results}) relies on some technical details worked out in Appendices \ref{appendix:analysis-model-smoothing} and \ref{appendix:genus-0-pinchings}. Finally, in \textsection\ref{subsec:example-model-sing}, we give several examples of model pinchings (and the corresponding examples of stable maps with model ghosts).

%%%%%%%%%%%%%%%%%%%%%%%%%%%%%%%%%%%%%%%%%%%%%%%%%%%%%%%%%
\subsection*{Notation for this section}
\begin{longtable}{ r l }

   $C$& Smooth projective curve of genus $g(C)\ge 1$ (Notation \ref{not:model-sing-input}).\\
   $n$ & Positive integer denoting number of points on $C$ (Notation \ref{not:model-sing-input}).\\
   $p_i$& Distinct closed points on $C$ for $1\le i\le n$ (Notation \ref{not:model-sing-input}).\\
   $d_i$& Relatively prime positive integers for $1\le i\le n$ (Notation \ref{not:model-sing-input}).\\
   $\Delta$& $\bQ$-divisor $\sum\frac 1{d_i}\cdot p_i$ on $C$ (Notation \ref{not:model-sing-input}).\\
   $U_0$ & Complement of $\{p_1,\ldots,p_n\}\subset C$ (Notation \ref{not:model-sing-input}).\\
   $R_0$ & Coordinate ring of the affine variety $U_0$ (Notation \ref{not:model-sing-input}).\\
   $A$  & Graded algebra $\bigoplus_{m\ge 0}A_m$ with $A_m = H^0(C,\clO_C(\lfloor m\Delta\rfloor)$ (Definition \ref{def:model-smoothing}).\\
   & $A$ is a finitely generated $k$-algebra (Lemma \ref{lem:A-finite-generation}).\\   
   $t$ & Fixed element of $A_1\subset A$ and coordinate on $\bA^1$ (Definition \ref{def:model-smoothing}). \\
   $\pi_\sY:\sY\to\bA^1$ & $\bG_m$-equivariant family with $\sY = \Spec A$ (Definition \ref{def:model-smoothing}, Figure \ref{fig:appendixA}).\\
   & $\pi_\sY$ is flat with general fibre $U_0$ (Lemma \ref{lem:flatness-and-smoothing}). \\
   & $\sY$ is a normal surface (Corollary \ref{cor:contraction-normal}).\\
   $\fA$ & Graded algebra $\bigoplus_{m\ge 0}\fA_m$ with $\fA_m = A_m/A_{m-1}$ (Definition \ref{def:model-sing}).\\
   $\sY_0$ & Model singularity, $\sY_0=\Spec(\fA)$ (Definition \ref{def:model-sing}, Figure \ref{fig:Partial normalization}).\\
   &  $\sY_0=\pi_\sY^{-1}(0)$ (Lemma \ref{lem:model-sing}). $\sY_0$ inherits $\bG_m$-action from $\sY$.\\
   $A_+$ & Maximal ideal $\bigoplus_{m>0}A_m$ of $A$ (Definition \ref{def:cone-point}).\\
   $\overline{c}$ & Closed point of $\sY_0$ corresponding to $A_+\subset A$ (Definition \ref{def:cone-point}).\\
    $\gamma:E\to\sY_0$ & Seminormalization of $\sY_0$ (Definition \ref{def:partial-norm}, Figure \ref{fig:Partial normalization}).\\
   $Q$ & Graded coordinate ring $\bigoplus_{m\ge 0}Q_m$ of $E$ (Definition \ref{def:partial-norm}).\\
   & $\fA\subset Q$ is a graded sub-algebra and $\dim Q/\fA = g(C)$ (Corollary \ref{cor:coord-ring-of-curve-sing-final}).\\
   $\overline{E}_i$ & Projective line compactifying $T_{C,p_i}^\vee$ for $1\le i\le n$ (Notation \ref{not:compactified-model-sing}).\\
   $\overline{\sY}_0$ & Compactification of $\sY_0$, adds $n$ points at infinity (Notation \ref{not:compactified-model-sing}).\\
   $\nu:\overline{\sX}_0\to\overline{\sY}_0$ & Model pinching (Definition \ref{def:model-pinching}).\\
   $\sX_0$ & Inverse image of $\sY_0$ under $\nu$ (Notation \ref{not:compactified-model-sing}, Definition \ref{def:model-pinching}).\\
   $f:\sC\to X$ & Stable map with model ghosts (Definition \ref{def:stable-map-with-model-ghosts}). \\
   $\sE$ & Union of irreducible components where $f$ is non-constant (Notation \ref{not:stable-map-eff-comp}).
\end{longtable}

%%%%%%%%%%%%%%%%%%%%%%%%%%%%%%%%%%%%%%%%%%%%%%%%%%%%%%%%
\subsection{Model curve singularity and its smoothing}\label{subsec:model-sing-and-smoothing}

We will construct a class of locally smoothable pinchings (in the sense of Definitions \ref{def:pinching} and \ref{def:locally-smoothable-pinching}). These will be called \emph{model pinchings}. We begin by describing the input data required for the construction.

\begin{notation}[Input for construction of model pinching]\label{not:model-sing-input}
    Fix a smooth projective curve $C$ of genus $g(C)\ge 1$. Also fix an integer $n\ge 1$, a collection of distinct closed points $p_1,\ldots,p_n\in C$ and integers $d_1,\ldots,d_n\ge 1$ with $\gcd(d_1,\ldots,d_n)=1$. Define the $\bQ$-divisor 
    \begin{align}\label{eqn:ample-Q-div}
        \Delta := \sum_{1\le i\le n} \frac1{d_i}\cdot p_i
    \end{align}
    on $C$. Denote $C\setminus\{p_1,\ldots,p_n\}$ by $U_0 = \Spec R_0$
\end{notation}

We now define the model singularity $\sY_0$ (and its smoothing $\pi_\sY:\sY\to\bA^1$), associated to the data $C$, $\{p_i\}_{1\le i\le n}$ and $\{d_i\}_{1\le i\le n}$ fixed in Notation \ref{not:model-sing-input}.

\begin{Definition}[Model smoothing]\label{def:model-smoothing}
   Define $\sY := \Spec A$, where the graded $k$-algebra $A$ is defined by
    \begin{align*}
        A := \bigoplus_{m\ge 0} H^0(C,\clO_C(\lfloor m\Delta\rfloor))
    \end{align*}
    with $\lfloor m\Delta\rfloor := \sum\lfloor\frac{m}{d_i}\rfloor\cdot p_i$. For $m\ge 0$, denote by $A_m$ the summand of $A$ in grading $m$. Multiplication in $A$ is given by the tensoring via the maps
    \begin{align*}
        A_m\otimes A_{m'}\to A_{m+m'}
    \end{align*}
    for $m,m'\ge 0$, which are well-defined since $\lfloor m\Delta\rfloor + \lfloor m'\Delta\rfloor\le\lfloor(m+m')\Delta\rfloor$.
    The element $1\in H^0(C,\clO_C)$ determines an element $t\in A_1$ via the inclusion $\clO_C\subset\clO_C(\lfloor\Delta\rfloor)$ of coherent sheaves. This yields an injective graded ring map $k[t]\to A$ which determines a dominant $\bG_m$-equivariant morphism
    \begin{align*}
        \pi_{\sY}:\sY\to\bA^1 = \Spec k[t].
    \end{align*}
    Note that the $\bG_m$-actions on $\sY$ and $\bA^1$ are induced by the gradings on their respective coordinate rings.
\end{Definition}

See Figure \ref{fig:appendixA} (in Appendix \ref{appendix:analysis-model-smoothing}) for a depiction of $\pi_\sY:\sY\to\bA^1$.

\begin{remark}\label{rem:model-smoothing-subalg}
    For $m\ge 0$, we have a natural inclusion $A_m\subset R_0 = \Gamma(U_0,\clO_{U_0})$ given by restricting sections of $\clO_C(\lfloor m\Delta\rfloor)$ to $U_0$. From this, we get an inclusion $A = \bigoplus_{m\ge 0} A_m\subset\bigoplus_{m\ge 0} R_0 = R_0[t]$ of graded $k[t]$-algebras. Another related observation which will be useful is that $R_0$ is the increasing union of $A_m$ for $m\ge 0$. Here, the inclusion $A_m\subset A_{m+1}$, for $m\ge 0$, comes from the corresponding inclusion of coherent sheaves $\clO_C(\lfloor m\Delta\rfloor)\subset\clO_C(\lfloor (m+1)\Delta\rfloor)$. The increasing filtration $\{A_m\}_{m\ge 0}$ of $R_0$ respects multiplication in the sense that we have $A_m\cdot A_{m'}\subset A_{m+m'}\subset R_0$ for $m,m'\ge 0$.
\end{remark}

Next, we introduce the model curve singularity.

\begin{Definition}[Model singularity]\label{def:model-sing}
    Define $\sY_0 := \Spec\fA$, where the $k$-algebra
    \begin{align*}
        \fA := \bigoplus_{m\ge 0} A_m/A_{m-1}
    \end{align*}
    is the associated graded\footnote{For convenience, we set $A_{-1} = 0$.} ring for the increasing filtration $\{A_m\}_{m\ge 0}$ of $R_0$ described in Remark \ref{rem:model-smoothing-subalg}. We call $\sY_0$ the \emph{model singularity} determined by $C$, $\{p_i\}_{1\le i\le n}$ and $\{d_i\}_{1\le i\le n}$.
\end{Definition}

To justify the notation $\sY_0$, we have the following lemma.

\begin{Lemma}\label{lem:model-sing}
    We have a natural identification $\sY_0 = \pi_\sY^{-1}(0)$.
\end{Lemma}
\begin{proof}
    The canonical quotient maps $A_m\twoheadrightarrow A_m/A_{m-1}$ for $m\ge 0$ together give a surjective map $A\twoheadrightarrow\fA$ of graded rings. To complete the proof, we show that the homogeneous ideal $\ker(A\twoheadrightarrow\fA) =: I = \bigoplus I_m$ is generated by $t$. For any $m\ge 1$, the restriction of the multiplication map
    \begin{align*}
        A_1\otimes A_{m-1}\to A_m
    \end{align*}
    to the subspace $t\otimes A_{m-1}\simeq A_{m-1}$ recovers the inclusion $A_{m-1}\subset A_m$ from Remark \ref{rem:model-smoothing-subalg}. Since $I_m\subset A_m$ is the same as $A_{m-1}\subset A_m$, this shows that $I$ is generated by $t\in A_1$ as required.
\end{proof}

\begin{remark}[Coprimality condition]
    Relaxing the coprimality condition $\gcd(d_1,\ldots,d_n) = 1$ in Notation \ref{not:model-sing-input} leads to no new examples of model singularities. Indeed, for any integer $d>1$, if we replace $(d_1,\ldots,d_n)$ by $(dd_1,\ldots,dd_n)$, then the $\bQ$-divisor $\Delta$ from \eqref{eqn:ample-Q-div} gets replaced by $\frac1d\Delta$ and the algebra $A$ (from Definition \ref{def:model-smoothing}) gets replaced by $A\otimes_{k[t]}k[t^{1/d}]$. Thus, the fibre over $0$ remains isomorphic to $\fA$.
\end{remark}

\begin{Lemma}\label{lem:A-finite-generation}
    $A$ is a finitely generated $k$-algebra.
\end{Lemma}
\begin{proof}
    Choose $d\gg 1$ to be a large integer divisible by all the $d_i$. Then, $\Delta'=d\Delta$ is an effective divisor which is also very ample. For each integer $0\le r< d$, define
    \begin{align*}
        S_r := \bigoplus_{m\ge 0} H^0(C,\clO_C(m\Delta'+\lfloor r\Delta\rfloor)). 
    \end{align*}
    We claim that $S_0$ is a finitely generated $k$-algebra and that $S_r$ is finitely generated as an $S_0$-module for $1\le r<d$. Since $A$ is the direct sum of $S_r$ over $0\le r<d$, the claim implies the lemma. Embed $C$ into some projective space $\bP^s = \Proj k[x_0,\ldots,x_s]$ using $\Delta'$. By pushing forward along this embedding, regard $\clO_C(\lfloor r\Delta\rfloor)$ as a coherent sheaf on $\bP^s$, for $0\le r<d$. Moreover, regard $\clO_C(\Delta')$ as the restriction of $\clO_{\bP^s}(1)$ to $C$. The claim now follows from assertion (5) of \cite[\href{https://stacks.math.columbia.edu/tag/01YS}{Tag 01YS}]{stacks-project} which shows that, for any coherent sheaf $\clF$ on $\bP^s$, the graded vector space $\bigoplus_{m\ge 0} H^0(\bP^s,\clF(m))$ is finitely generated as a module over $k[x_0,\ldots,x_s] = \bigoplus_{m\ge 0} H^0(\bP^s,\clO_{\bP^s}(m))$.
\end{proof}

\begin{Corollary}\label{cor:Y-affine-variety}   
    $\sY$ is an affine variety.
\end{Corollary}
\begin{proof}
    Remark \ref{rem:model-smoothing-subalg} implies that $A$ is an integral domain. Combining this with Lemma \ref{lem:A-finite-generation}, we see that $\sY = \Spec A$ is an affine variety.
\end{proof}

The next lemma shows that $\pi_\sY:\sY\to\bA^1$ is indeed a smoothing of $\sY_0$.

\begin{Lemma}[$\sY$ is a smoothing of $\sY_0$]\label{lem:flatness-and-smoothing}
    The morphism $\pi_\sY:\sY\to\bA^1$ is flat with $1$-dimensional fibres and we have a natural $\bG_m$-equivariant identification
    \begin{align}\label{eqn:Y-generic-trivial}
        \sY|_{\bA^1\setminus\{0\}} = U_0\times(\bA^1\setminus\{0\})
    \end{align}
    of families over $\bA^1\setminus\{0\}$.
\end{Lemma}
\begin{proof}
    \cite[\href{https://stacks.math.columbia.edu/tag/00HD}{Tag 00HD}]{stacks-project} shows that a module $M$ over a ring $R$ is flat if and only if $I\otimes_R M\to R\otimes_RM = M$ is injective for all ideals $I\subset R$. Thus, a module over the principal ideal domain $k[t]$ is flat if and only if it is torsion-free. Since the ring map $k[t]\to A$ is injective and $A$ is an integral domain, we get the flatness of $\pi_\sY$.

    To prove the second assertion, recall from Remark \ref{rem:model-smoothing-subalg} that we have an inclusion $A\subset R_0[t]$ of graded $k[t]$-algebras and that $R_0$ is the increasing union of $A_m$ for $m\ge 0$. Thus, adjoining the inverse of $t$ to $A$ yields the graded $k[t^{\pm1}]$-algebra $R_0[t^{\pm1}]$, and as a result, we get the identification \eqref{eqn:Y-generic-trivial}.
\end{proof}

\begin{remark}[Flatness over a DVR]\label{rem:flatness-over-dvr}
    The argument used in the proof of Lemma \ref{lem:flatness-and-smoothing} shows that a dominant morphism from a variety to a non-singular curve is automatically flat. The main point is that the local ring of a non-singular curve at a closed point is a discrete valuation ring (and therefore, a principal ideal domain).
\end{remark}

\begin{remark}
    It follows from Lemma \ref{lem:flatness-and-smoothing} that $\sY$ is a surface and $\sY_0$ is a curve.
\end{remark}

\begin{remark}[Motivation for the definitions of $\sY$ and $\sY_0$]\label{rem:model-motivation-pinkham-surf-sing}
    Definitions \ref{def:model-smoothing} and \ref{def:model-sing} are inspired by a special case of \cite[Theorem 5.1]{Pinkham-surf-sing}. In this theorem, Pinkham provides a completely explicit description of the graded coordinate ring of any $\bG_m$-equivariant affine normal surface singularity (in terms of the geometry of a $\bG_m$-equivariant resolution of singularities).
\end{remark}

We will compute examples of the model singularity $\sY_0$ (obtained by taking specific choices of $C$, $n$, $p_i$ and $d_i$) in \textsection\ref{subsec:example-model-sing}. For this, it will be helpful to have an explicit algebraic description of the normalization of $\sY_0$. After some preliminaries, this is given in Lemma \ref{lem:justifying-partial-norm}

\begin{Definition}\label{def:cone-point}
    Let $\overline{c}\in\sY$ be the closed point corresponding to the maximal ideal $A_+ := \bigoplus_{m>0} A_m \subset A$. Since $t\in A_1\subset A_+$, we have $\overline{c}\in\pi_{\sY}^{-1}(0) = \sY_0$.
\end{Definition}

\begin{notation}
    Define the coherent sheaves $\clQ_m$ on $C$ for integers $m\ge 0$ as follows. Set $\clQ_0 = \clO_C$, and for $m\ge 1$, define $\clQ_m$ by requiring
    \begin{align}\label{eqn:ses-Q}
        0\to\clO_C(\lfloor(m-1)\Delta\rfloor)\to \clO_C(\lfloor m\Delta\rfloor)\to\clQ_m\to 0
    \end{align}
    to be a short exact sequence.
\end{notation}

For $m\ge 1$, the key point to note about the coherent sheaf $\clQ_m$ is that its support is contained in $\{p_1,\ldots,p_n\}\subset C$, and its stalks are given by\footnote{The computation depends on the following assertion: for any $m\in\bZ$ and closed point $p\in C$, we have $\clO_C(m\cdot p)\otimes_{\clO_C}(\clO_{C,p}/\fm_{C,p}) = (T_{C,p})^{\otimes m}$. Note that $\clO_C(-p)$ is the ideal sheaf of $p$ in $C$, and we have $\clO_C(-p)\otimes_{\clO_C}(\clO_{C,p}/\fm_{C,p}) = \fm_{C,p}/\fm_{C,p}^2 = T_{C,p}^\vee$. The assertion now follows since the invertible sheaf $\clO_C(m\cdot p)$ is the $(-m)^\text{th}$ tensor power of the invertible sheaf $\clO_C(-p)$.}
\begin{align}\label{eqn:Q-local-stalks}
    (\clQ_m)_{p_i} = \begin{cases}
        0 & \text{when } d_i\nmid m \\
        (T_{C,p_i})^{\otimes\frac m{d_i}} & \text{when } d_i\mid m
    \end{cases}
\end{align}
for $1\le i\le n$.

\begin{remark}\label{rem:curve-sing-incl-in-Q}
    The long exact sequence in cohomology associated to \eqref{eqn:ses-Q} gives an inclusion $\fA_m\subset H^0(C,\clQ_m)$ for $m\ge 1$. Moreover, when $m=0$, we have $\fA_0 = H^0(C,\clQ_0) = k$. Taking the direct sum over $m\ge 0$, we obtain an inclusion
    \begin{align}\label{eqn:curve-sing-incl-in-Q}
        \fA\subset \bigoplus_{m\ge 0} H^0(C,\clQ_m).
    \end{align}
\end{remark}

\begin{Definition}[Seminormalization of $\sY_0$]\label{def:partial-norm}
    Consider the $k$-vector space
    \begin{align}
        V := \bigoplus_{1\le i\le n}T_{C,p_i}^\vee     
    \end{align}
    as the affine variety $\Spec(\Sym V^\vee)$ equipped with the $\bG_m$-action given by
    \begin{align*}
        \lambda\cdot(v_1,\ldots,v_n) = (\lambda^{d_1}v_1,\ldots,\lambda^{d_n}v_n).
    \end{align*}
    Define $E\subset V$ to be the reduced closed subscheme given by the union of the coordinate lines $T^\vee_{C,p_i}$. Denote by $Q = \bigoplus_{m\ge 0} Q_m$ the coordinate ring of the affine scheme $E$, graded using the $\bG_m$-action inherited from $V$, with $Q_0 := k$ and
    \begin{align}\label{eqn:partial-normalization-coord-ring}
        Q_m := \bigoplus_{d_i\mid m} (T_{C,p_i})^{\otimes\frac m{d_i}}
    \end{align}
    for $m\ge 1$. For $m\ge 1$, denote by $Q_{m,i}$ the summand of $Q_m$ in \eqref{eqn:partial-normalization-coord-ring} corresponding to $1\le i\le n$, with $Q_{m,i} = 0$ when $d_i\nmid m$. The computation \eqref{eqn:Q-local-stalks} yields the identification
    \begin{align}\label{eqn:coord-ring-E-and-Q}
        Q_m = H^0(C,\clQ_m)
    \end{align}
    for $m\ge 0$. Combining \eqref{eqn:curve-sing-incl-in-Q} and \eqref{eqn:coord-ring-E-and-Q} gives an inclusion $\fA\subset Q$, which is a map of graded $k$-algebras. The $\bG_m$-equivariant morphism of affine schemes 
     \begin{align}\label{eqn:partial-normalization}
    \gamma:E\to\sY_0 
    \end{align}
    corresponding to the inclusion $\fA\subset Q$ of rings is called the \emph{seminormalization} of $\sY_0$.

    \begin{figure}[ht]
    \centering
\includegraphics[width=8.3cm]{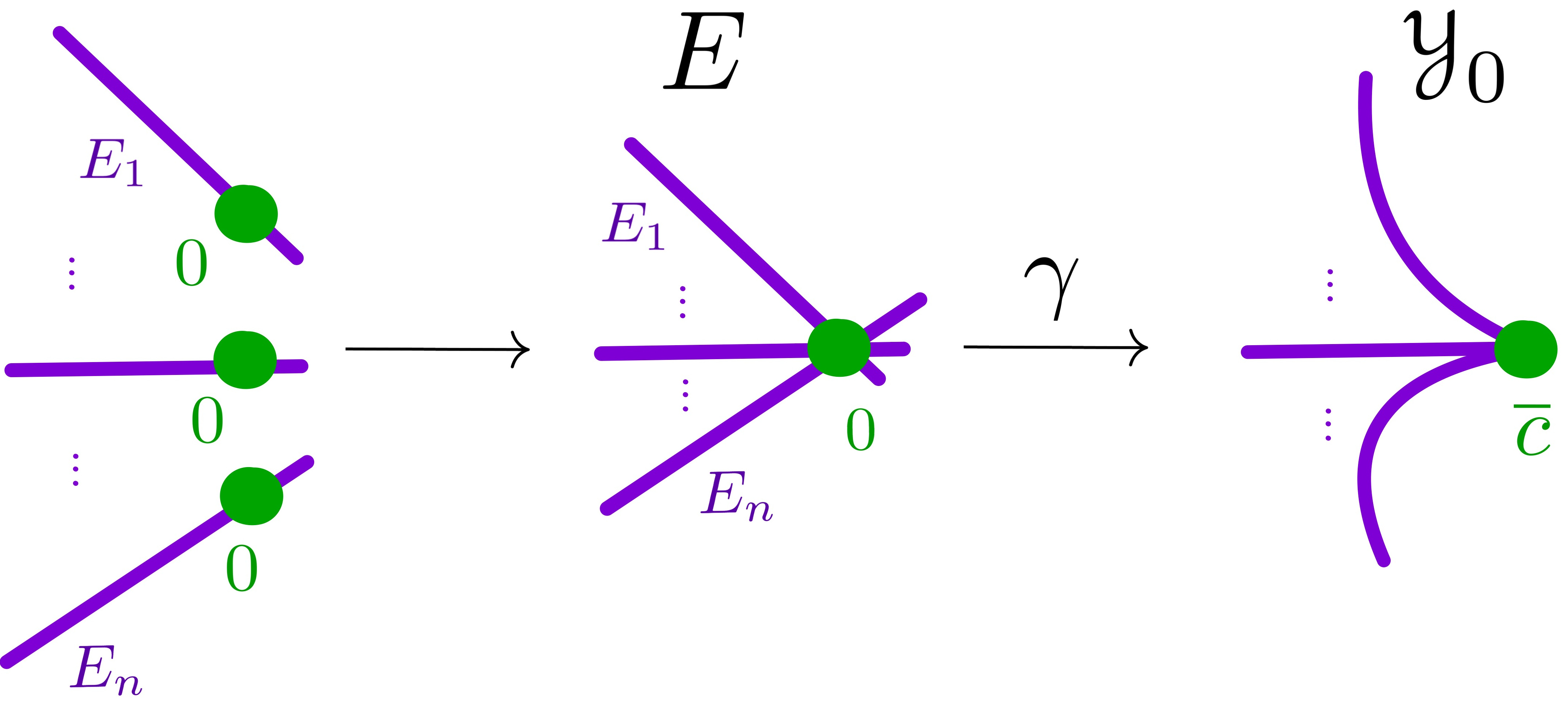}
    \caption{Seminormalization of the model singularity $\sY_0$.}
    \label{fig:Partial normalization}
\end{figure}
\end{Definition}

The next lemma (which depends on the results of Appendix \ref{appendix:analysis-model-smoothing}) is the justification for calling $\gamma$ the \emph{seminormalization} of $\sY_0$ (see Figure \ref{fig:Partial normalization}).

\begin{Lemma}[Normalization of $\sY_0$]\label{lem:justifying-partial-norm}
    The morphism $\gamma$ from \eqref{eqn:partial-normalization} is bijective on points and induces an isomorphism 
    \begin{align}
        E\setminus\{0\}\simeq\sY_0\setminus\{\overline{c}\}.    
    \end{align}
    Thus, the natural morphism $\bigsqcup_{1\le i\le n}T_{C,p_i}^\vee\to E$ followed by $\gamma$ yields the normalization morphism of $\sY_0$.
\end{Lemma}
\begin{proof}
    This is the content of Lemma \ref{lem:partial-normalization-ring-map} and the paragraph preceding it.
\end{proof}

\begin{Corollary}\label{cor:coord-ring-of-curve-sing-final}
    The inclusion $\fA_m \subset Q_m$ from Definition \ref{def:partial-norm} is an equality if
    \begin{align}\label{eqn:coord-ring-cutoff-point}
        \sum_{i=1}^n\left\lfloor\frac{m-1}{d_i}\right\rfloor > 2\cdot g(C)-2.
    \end{align}
    The coherent sheaf $(\gamma_*\clO_{E}/\clO_{\sY_0})$ is zero on $\sY_0\setminus\{\overline{c}\}$ and its stalk at $\overline{c}$ is identified with the $g(C)$-dimensional vector space $Q/\fA$.
\end{Corollary}
\begin{proof}
    For $m\ge 1$, the long exact sequence in cohomology associated to \eqref{eqn:ses-Q} shows that $Q_m/\fA_m$ is isomorphic to the kernel of the surjection
    \begin{align}\label{eqn:decreasing-cohomologies}
        H^1(C,\clO_C(\lfloor (m-1)\Delta\rfloor))\twoheadrightarrow H^1(C,\clO_C(\lfloor m\Delta\rfloor)).
    \end{align}
    The first assertion now follows since \eqref{eqn:coord-ring-cutoff-point} implies, by Serre duality, that the domain of the map \eqref{eqn:decreasing-cohomologies} is zero. In particular, it follows that $Q/\fA = \bigoplus Q_m/\fA_m$ is a finite dimensional $k$-vector space. Counting dimensions via \eqref{eqn:decreasing-cohomologies} and summing over $m$, we obtain a telescoping series, which shows that 
    \begin{align}\label{eqn:generalized-g-gaps}
        \dim Q/\fA = \dim H^1(C,\clO_C) = g(C).
    \end{align}
    Next, observe that the finite morphism $\gamma$ is bijective on points and an isomorphism over $\sY_0\setminus\{\overline c\}$. As a result, the sheaf $(\gamma_*\clO_{E}/\clO_{\sY_0})$ is zero away from $\overline{c}$. Its stalk at $\overline{c}$ is therefore identified with its space $Q/\fA$ of global sections. Now, \eqref{eqn:generalized-g-gaps} completes the proof of the second assertion.
\end{proof}

We now introduce the model pinching.

\begin{notation}\label{not:compactified-model-sing}
    For $1\le i\le n$, the projective line $\overline{E}_i$ is defined to be the compactification of $T_{C,p_i}^\vee$. Let $\overline{\sY}_0$ be the scheme obtained by gluing $\sY_0$ and $\bigsqcup_{1\le i\le n}(\overline{E}_i\setminus\{0\})$ along the isomorphism
    \begin{align*}
        \sY_0\setminus\{\overline{c}\}\simeq E\setminus\{0\} = \bigsqcup_{1\le i\le n}(T_{C,p_i}^\vee\setminus\{0\})
    \end{align*}
    induced by $\gamma$. Note that $\overline\sY_0$ is a reduced proper curve with an isolated singularity at $\overline{c}$.

    Define the prestable curve $\overline{\sX}_0$ as follows. To obtain $\overline{\sX}_0$, We take $C\sqcup\bigsqcup_{1\le i\le n}\overline{E}_i$ and glue $p_i\in C$ to $0\in T_{C,p_i}^\vee\subset\overline{E}_i$, for $1\le i\le n$. Define $\sX_0\subset\overline{\sX}_0$ to be the complement of the $n$ points at infinity (one in each $\overline{E}_i$).
\end{notation}

\begin{Definition}[Model pinching]\label{def:model-pinching}
    Define the morphism
    \begin{align*}
        \nu:\sX_0\to\sY_0
    \end{align*}
    as the following composition. First, map $\sX_0$ to $E$ via the identity morphism on each $E_i$ and the constant morphism $0$ on $C$; then apply $\gamma:E\to\sY_0$. This extends to a morphism $\overline{\sX}_0\to\overline{\sY}_0$, which is still denoted by $\nu$. 
    
    We call $\nu:\overline{\sX}_0\to\overline{\sY}_0$ the \emph{model pinching} associated to $C$, $\{p_i\}_{1\le i\le n}$ and $\{d_i\}_{1\le i\le n}$.
\end{Definition}

Lemma \ref{lem:justifying-partial-norm} and Corollary \ref{cor:coord-ring-of-curve-sing-final} show that $\nu:\overline{\sX}_0\to\overline{\sY}_0$ is a pinching of $\overline{\sX}_0$ at $C$ (in the sense of Definition \ref{def:pinching}). This justifies the terminology in Definition \ref{def:model-pinching}. The significance of the model pinching comes from the next proposition. The proof of this proposition depends on the detailed analysis of $\pi_\sY:\sY\to\bA^1$ carried out in Appendix \ref{appendix:analysis-model-smoothing}.

\begin{Proposition}[Local smoothability of model pinching]\label{prop:local-smoothability-of-model-pinching}
    The pinching $\nu:\overline{\sX}_0\to\overline{\sY}_0$ is locally smoothable in the sense of Definition \ref{def:locally-smoothable-pinching}.
\end{Proposition}
\begin{proof}
    To begin, note that $\pi_\sY:\sY\to\bA^1$ is smooth over $\bA^1\setminus\{0\}$ (by Lemma \ref{lem:flatness-and-smoothing}) and has fibre $\sY_0$ over $0\in\bA^1$ (by Lemma \ref{lem:model-sing}). Proposition \ref{prop:local-smoothability-of-model-pinching-appx} produces a proper birational morphism $\Phi:\sX\to\sY$ such that the following properties hold.
    \begin{enumerate}[(i)]
        \item The composition $\sX\to\sY\to\bA^1$ is flat and $\Phi\times_{\bA^1}\{0\}$ is identified with $\nu:\sX_0\to\sY_0$.
        \item $\Phi$ restricts to an isomorphism over $\sY\setminus\{\overline c\}$.
    \end{enumerate}
    
    We claim that the existence of a morphism $\Phi:\sX\to\sY$ as above implies the local smoothability of the pinching $\nu:\overline{\sX}_0\to\overline{\sY}_0$. Indeed, we may take the identity morphism of $(\sY_0,\overline{c})$, the pointed curve $0\in\bA^1$ and the morphism $\Phi:\sX\to\sY$ to play the respective roles of $\varphi_i:(U_i,s_i)\to(V_i,s_i')$, $0\in B_i$ and $\psi_i:\clX_i\to\clY_i$ from Definition \ref{def:locally-smoothable-pinching}.
\end{proof}

\begin{remark}
    For the reader's convenience, we explicitly describe the $\sY$-scheme $\sX$ mentioned in the proof of Proposition \ref{prop:local-smoothability-of-model-pinching}. More details can be found in Lemma \ref{lem:inverse-contraction}. 
    
    Via the isomorphism \eqref{eqn:Y-generic-trivial}, the coordinate projection $U_0\times(\bA^1\setminus\{0\})\to U_0$ defines a rational map $\sY\dashrightarrow C$. Let $\Gamma\subset\sY\times C$ be the closure of the graph of this rational map. The $\sY$-scheme $\sX$ can then be described as the normalization of the $\sY$-scheme $\Gamma$.
\end{remark}

%%%%%%%%%%%%%%%%%%%%%%%%%%%%%%%%%%%%%%%%%%%%%%%%%%%%%%%%%%
\subsection{Stable maps with model ghosts}\label{subsec:stable-maps-with-model-ghosts} We will now introduce a class of stable maps, called \emph{stable maps with model ghosts}, which we later prove to be eventually smoothable. To understand the definition of this class of stable maps, the reader should recall the definitions of ghost components (Definition \ref{def:ghost-components}) and pinchings (Definition \ref{def:pinching}).

\begin{notation}\label{not:stable-map-eff-comp}
    Let $X$ be a projective variety and let $f:\sC\to X$ be a non-constant stable map. Let $\sE$ denote the non-empty union of those irreducible components of $\sC$ on which $f$ is non-constant.
\end{notation}

\begin{Definition}[Stable map with model ghosts]\label{def:stable-map-with-model-ghosts}
    We say that $f:\sC\to X$ is a \emph{stable map with model ghosts} if $f$ factors through a pinching $\nu_f:\sC\to\sS_f$ of $\sC$ at its ghost sub-curve, satisfying the following condition at each of its ghost components. 
    
    For a ghost component $C$ of $f$, let $s\in\sS_f$ be its image under $\nu_f$. Enumerate the points of $C\cap\sE$ as $p_1,\ldots,p_n$. Then, $p_1,\ldots,p_n$ lie on pairwise distinct irreducible components of $\sE$. Further, exactly one of the following two conditions holds.
    \begin{enumerate}[\normalfont(a)]
        \item $C$ is of arithmetic genus $0$, i.e., $H^1(C,\clO_C) = 0$, but possibly singular.
        \item $C$ is non-singular of genus $g(C)\ge 1$. Moreover, there exist integers $d_1,\ldots d_n\ge 1$ with $\gcd(d_1,\ldots,d_n) = 1$ such that, near $s\in\sS_f$, the pinching $\nu_f$ is an \'etale pullback of the model pinching
        \begin{align*}
            \nu:\overline{\sX}_0\to\overline{\sY}_0    
        \end{align*}
        associated to $C$, $\{p_i\}_{1\le i\le n}$ and $\{d_i\}_{1\le i\le n}$ by Definition \ref{def:model-pinching}.
        
        More precisely, there is an affine open neighborhood $U\subset\sS_f$ of $s$, an \'etale morphism $\varphi_s:U\to\sY_0$ with $\varphi_s^{-1}(\overline{c}) = \{s\}$, and a commutative diagram
        \begin{equation}\label{eqn:pullback-of-model-pinching}
            \begin{tikzcd}
                \nu_f^{-1}(U) \arrow[r,"\simeq"] \arrow[dr,"\nu_f"] & U\times_{\sY_0}\sX_0 \arrow[d] \arrow[r]  & \sX_0 \arrow[d,"\nu"] \\
                & U \arrow[r,"\varphi_s"] & \sY_0
            \end{tikzcd}
        \end{equation}
        where the \'etale morphism $\nu_f^{-1}(U)\to\sX_0$ restricts on $C$ to the identity morphism.
    \end{enumerate}

    \begin{figure}[ht]
    \centering
\includegraphics[width=8.8cm]{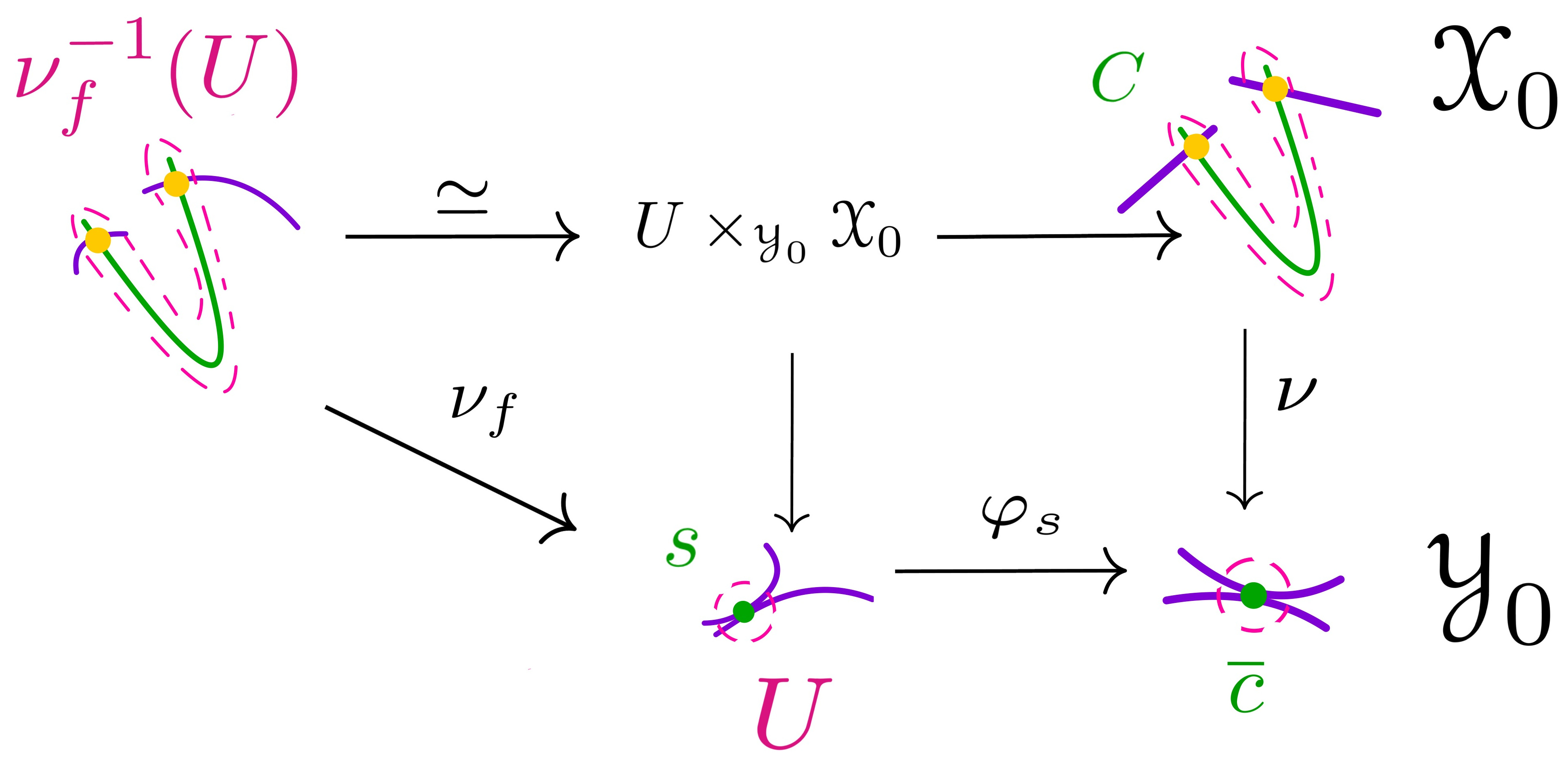}

    \caption{Local picture of the pinching $\nu_f:\sC\to\sS_f$, from Definition \ref{def:stable-map-with-model-ghosts}, with $C$ (shown in green) being a non-singular curve of genus $g(C)\ge 1$.}
    \label{fig: stable-map with model ghost}
\end{figure}
\end{Definition}

\begin{remark}[Model ghost condition is preserved by embeddings]\label{rem:model-ghost-functorial}
    If $\sC\to X$ is a stable map with model ghosts and $X\hookrightarrow X'$ is an embedding of $X$ into another projective variety $X'$, then the composition $\sC\to X\hookrightarrow X'$ is also a stable map with model ghosts.
\end{remark}

\begin{remark}[Genus $0$ versus higher genus pinchings]\label{rem:genus-0-vs-higher}
    Example \ref{exa:genus-0-pinching} shows that the local structure of a genus $0$ pinching is uniquely determined. This is why there is no need to specify a local picture like \eqref{eqn:pullback-of-model-pinching} in condition (a) of Definition \ref{def:stable-map-with-model-ghosts}.
\end{remark}

Definition \ref{def:stable-map-with-model-ghosts} is stated in terms of first finding suitable pinching $\nu_f:\sC\to\sS_f$ (through which $f:\sC\to X$ factors) and then finding \'etale morphisms $\varphi_s$ (which exhibit $\nu_f$ locally as a pullback of model pinchings). The next remark gives a detailed reformulation of this definition in terms of local coordinates.

\begin{remark}[Model ghost condition in coordinates]\label{rem:model-ghost-in-coord}
    We continue with the notation of Definition \ref{def:stable-map-with-model-ghosts}. Constructing the \'etale morphism $\nu_f^{-1}(U)\to\sX_0$ from \eqref{eqn:pullback-of-model-pinching}, near a ghost component $C\subset\sC$, amounts to specifying local coordinates on $\sE$ at $p_i$ (with values in $T_{C,p_i}^\vee$) for $1\le i\le n$, i.e., specifying elements
    \begin{align}\label{eqn:coords-comparing-to-model}
        \zeta_i\in \fm_{\sE,p_i}\otimes T_{C,p_i}^\vee
    \end{align}
    with nonzero images in $T_{\sE,p_i}^\vee\otimes T_{C,p_i}^\vee$ for $1\le i\le n$. 
    
    Near the point $s\in\sS_f$, we know that $\sS_f\setminus\{s\}$ should be identified (by the pinching morphism $\nu_f$) with $\sE\setminus\{p_1,\ldots,p_n\}$, near $C\subset\sC$. If we have $H^1(C,\clO_C) = 0$, then we can reconstruct the pinching $\nu_f$, near $s$, as in Example \ref{exa:genus-0-pinching}. When $C$ is non-singular of genus $g(C)\ge 1$, the key point is that the elements \eqref{eqn:coords-comparing-to-model}, together with the model pinching $\nu:\overline{\sX}_0\to\overline{\sY}_0$ determined by $\{d_i\}_{1\le i\le n}$, are sufficient to \emph{reconstruct} the local ring 
    \begin{align}\label{eqn:actual-from-model-local-ring}
        \clO_{\sS_f,s}\subset\prod_{1\le i\le n}\clO_{\sE,p_i}.
    \end{align}
    This, of course, would allow us to also reconstruct the pinching $\nu_f$ near $s\in\sS_f$. Explicitly, \eqref{eqn:actual-from-model-local-ring} is determined as follows. 
    \begin{enumerate}[(i)]
        \item Introduce the local sub-ring $\prod_{1\le i\le n}'\clO_{\sE,p_i}\subset\prod_{1\le i\le n}\clO_{\sE,p_i}$, as in Example \ref{exa:genus-0-pinching}, consisting of tuples $(h_i)_{1\le i\le n}$ such that $h_i(p_i)\in k$ is independent of $i$. Denote the maximal ideal of this local ring by $\fm'$.
        \item The elements \eqref{eqn:coords-comparing-to-model} provide a local $k$-algebra map 
        \begin{align}\label{eqn:intermediate-map-zeta}
            \zeta:Q\to\textstyle\prod_{1\le i\le n}'\clO_{\sE,p_i},
        \end{align}
        where $Q$ is the coordinate ring of the union $E$ of the lines $\{T_{C,p_i}^\vee\}_{1\le i\le n}$ with all their origins identified (Definition \ref{def:partial-norm}). Denote the maximal ideal of $Q$ defining $0\in E$ by $\fm_{Q,0}$. Then, the map $\zeta$ from \eqref{eqn:intermediate-map-zeta}, after passing to the $\fm_{Q,0}$-adic completion of $Q$ and the $\fm'$-adic completion of $\prod_{1\le i\le n}'\clO_{\sE,p_i}$, induces an isomorphism $\zeta^\wedge$ of complete local $k$-algebras. Note that $\zeta^\wedge$ sets up a bijection between sub-algebras of $Q$ containing a power of $\fm_{Q,0}$ and sub-algebras of $\prod_{1\le i\le n}'\clO_{\sE,p_i}$ containing a power of $\fm'$.
        \item Recall that we have the sub-algebra $\clO_{\sY_0,\overline{c}} = \fA\subset Q$ (Definition \ref{def:partial-norm}), corresponding to the local ring of the model singularity $\sY_0$ at its unique singular point $\overline{c}$. Corollary \ref{cor:coord-ring-of-curve-sing-final} shows that $\fA$ contains a power of $\fm_{Q,0}$ and is of codimension $g(C)$ in $Q$. Thus, as noted above in (ii), $\zeta^\wedge$ allows us to transfer $\fA\subset Q$ to obtain a sub-algebra of $\prod_{1\le i\le n}'\clO_{\sE,p_i}$, containing a power of $\fm'$. This is exactly \eqref{eqn:actual-from-model-local-ring}.
    \end{enumerate} 
   
   Denote the image of the ghost $C$ under $f$ by $q\in X$. The condition that $f$ factors through $\nu_f$ then takes the following  form in terms of the above discussion: the ring map $\clO_{X,q}\to\prod_{1\le i\le n}'\clO_{\sE,p_i}$, induced by $f|_\sE$, must have image contained in the sub-ring $\clO_{\sS_f,s}$.  This translates to $g(C)$ linearly independent conditions on the Taylor expansion of $f|_\sE$ at the points $p_i$ for $1\le i\le n$, using the coordinates \eqref{eqn:coords-comparing-to-model}. We make these Taylor expansion conditions explicit in the examples given in \textsection\ref{subsec:example-model-sing}.
\end{remark}

It turns out that stable maps with model ghosts are always eventually smoothable. Before stating and proving a theorem to this effect, we need the following auxiliary proposition (whose proof depends on Appendix \ref{appendix:genus-0-pinchings}). To understand the statement of this proposition, the reader should recall the discussion of local smoothability of pinchings given in \textsection\ref{subsec:smoothability-of-pinching}.

\begin{Proposition}[Local smoothability of genus $0$ pinching]\label{prop:local-smoothability-genus-0-pinching}
    Let $\nu:\sC\to\sS$ be a pinching. Let $s_1,\ldots,s_r\in\sS$ be the points over which $\nu$ fails to be an isomorphism. For $1\le i\le r$, let $C_i\subset\sC$ be the set-theoretic inverse image of $s_i$ under $\nu$. 
    
    Fix $1\le i\le r$. Denote the closure of $\sC\setminus C_i$ by $\Sigma_i\subset\sC$. Suppose that no irreducible component of $\Sigma_i$ meets $C_i$ more than once and that $H^1(C_i,\clO_{C_i}) = 0$. Then, the pinching $\nu:\sC\to\sS$ is smoothable near $s_i$, in the sense of Definition \ref{def:locally-smoothable-pinching}.
\end{Proposition}
\begin{proof}
    Let $\Sigma_{i,1},\ldots,\Sigma_{i,n}$ be an enumeration of the irreducible components of $\Sigma_i$ which meet $C_i$. By assumption, for $1\le j\le n$, the intersection $C_i\cap\Sigma_{i,j}$ consists of a single point, which we denote by $p_j$. Since we are investigating \emph{local} smoothability at $s_i$, we may replace $\sC$ by the sub-curve $C_i\cup\bigcup_{1\le j\le n}\Sigma_{i,j}$. It is enough to exhibit $\nu$ as an \'etale pullback (near $s_i$) of another locally smoothable pinching. By choosing a local coordinate in $\clO_{\Sigma_{i,j},p_j}$, for $1\le j\le n$, this means we may replace each $\Sigma_{i,j}$ by a copy of $\bP^1$.
    
    Let us summarize the reductions so far. We now have the pinching of a genus $0$ prestable curve at connected sub-curve. Moreover, the complement of this sub-curve consists of $n$ pairwise disjoint copies of $\bA^1$. We are left to show that this special kind of pinching is locally smoothable. The last assertion is exactly the content of Proposition \ref{prop:local-smoothability-genus-0-pinching-appx}.
\end{proof}

We now come to the main result on stable maps with model ghosts, which was stated as Theorem \ref{thm:intro-model-ghost} in \textsection\ref{subsec:summary-of-results}.

\begin{Theorem}[Stable maps with model ghosts are eventually smoothable]\label{thm:model-ghost-implies-eventual-smoothability}
    Let $f:\sC\to X$ be a stable map with model ghosts in the sense of Definition \ref{def:stable-map-with-model-ghosts}. Then, it is eventually smoothable in the sense of Definition \ref{def:eventual-smoothability}.
\end{Theorem}
\begin{proof}
    We have a pinching $\nu_f:\sC\to\sS_f$, as in Definition \ref{def:stable-map-with-model-ghosts}, through which $f$ factors. Using the local criterion for eventual smoothability (Theorem \ref{thm:local-criterion-for-smoothability}), it suffices to check that $\nu_f$ is \emph{formally smoothable} (in the sense of Definition \ref{def:locally-formally-smoothable-pinching}) near the image $s\in\sS_f$ of each ghost component $C\subset\sC$. By Lemma \ref{lem:actual-implies-formal-for-pinching}, it suffices to check \emph{local smoothability} of $\nu_f$ (in the sense of Definition \ref{def:locally-smoothable-pinching}) instead of local formal smoothability.

    With this in mind, fix a ghost component $C\subset\sC$ with image $s\in\sS_f$ under $\nu_f$. When $C$ is of genus $0$, smoothability of $\nu_f$ near $s$ follows from Proposition \ref{prop:local-smoothability-genus-0-pinching}. When $C$ is of positive genus (and therefore non-singular by Definition \ref{def:stable-map-with-model-ghosts}), smoothability of $\nu_f$ near $s$ follows from Proposition \ref{prop:local-smoothability-of-model-pinching}. This completes the proof.
\end{proof}

\begin{remark}[Generalizing the notion of stable maps with model ghosts]\label{rem:etale-subtleties}
    We expect that it is possible to generalize Definition \ref{def:stable-map-with-model-ghosts} by dropping the requirement that the points $p_1,\ldots,p_n\in\sE$ lie on pairwise distinct irreducible components of $\sE$. This requires replacing the Zariski neighborhood $U$ in \eqref{eqn:pullback-of-model-pinching} by an \emph{\'etale neighborhood} of $s\in\sS_f$. The reason for this is that a globally irreducible curve with an analytically reducible singularity\footnote{Recall that $\sS_f$ is said to be \emph{analytically reducible} at the closed point $s\in\sS_f$ if the completed local ring $\clO_{\sS_f,s}^\wedge$ is not an integral domain. This can happen even when $\clO_{\sS_f,s}$ is itself an integral domain.} remains irreducible Zariski locally (and becomes reducible only \emph{\'etale locally}, see Figure \ref{fig:etale-local-reducibility}). 

    \begin{figure}[ht]
    \centering
\includegraphics[width=4cm]{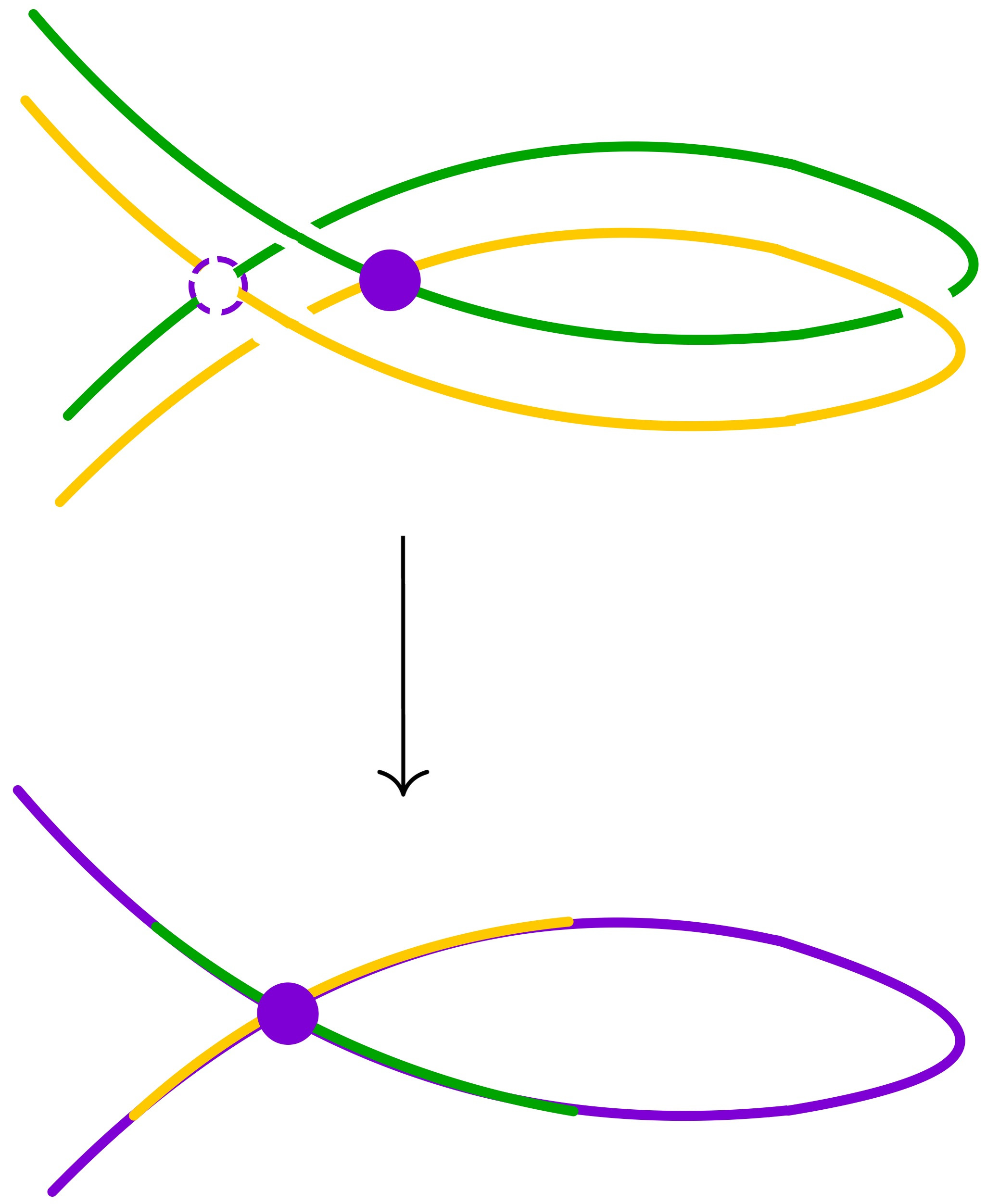}
    \caption{\'Etale local reducibility of an irreducible nodal curve. The analytic branches at the node (shown in yellow and green) become separate irreducible components after passing to an \'etale neighborhood of the node.}
    \label{fig:etale-local-reducibility}
\end{figure}
    
    For Theorem \ref{thm:model-ghost-implies-eventual-smoothability} to remain valid with this generalized definition, we also have to replace the Zariski neighborhoods $(U_i,s_i)$ appearing in the definition of local formal smoothability of a pinching (Definition \ref{def:locally-formally-smoothable-pinching}) by \'etale neighborhoods, and correspondingly, generalize the proof of the local criterion for eventual smoothability (Theorem \ref{thm:local-criterion-for-smoothability}). Working with \'etale neighborhoods would also allow us to drop the assumption in Proposition \ref{prop:local-smoothability-genus-0-pinching} stating that no irreducible component of $\Sigma_i$ meets $C_i$ more than once. 
\end{remark}

%%%%%%%%%%%%%%%%%%%%%%%%%%%%%%%%%%%%%%%%%%%%%%%%%%%%%%%%
\subsection{Examples}\label{subsec:example-model-sing}
We now compute the model singularity $\sY_0$ for some interesting choices of $C$, $\{p_i\}_{1\le i\le n}$ and $\{d_i\}_{1\le i\le n}$. We also explain what it concretely means for a stable map $f:\sC\to X$ with a ghost component $C$, to factor through a pinching of $\sC$ at $C$ which arises as an \'etale pullback of the model pinching $\overline{\sX}_0\to\overline{\sY}_0$. As mentioned in Remark \ref{rem:model-ghost-in-coord}, this factorization condition (which we refer to, in the examples below, as the \emph{model ghost condition}) is expressed in terms of the Taylor expansions of $f|_\sE$ at $p_1,\ldots,p_n\in C\cap\sE$. 

In the subsections below, we start with an abstract discussion, which is then specialized to obtain some concrete computational examples.

%%%%%%%%%%%%%%%%%%%%%%%%%%%%%%%%%%%%%%%%%%%%%%%%%%%%%%%%
\subsubsection{Monomial curves and Weierstrass points}\label{subsubsec:monomial-curves}

This set of examples is motivated by the results in \cite[\textsection{13}]{Pinkham-thesis}. Suppose $n = 1$ and write $p = p_1$. We then have $d_1 = 1$, $\Delta = p$ and
\begin{align*}
    A_m = H^0(C,\clO_C(m\cdot p))
\end{align*}
for $m\ge 0$. Define $W\subset\bZ_{\ge 0}$ to be the \emph{Weierstrass semigroup} of $C$ at $p$, i.e., we have $m\in W$ if and only if there exists a rational function on $C$ which is regular on $C\setminus\{p\}$ and has a pole of order exactly $m\ge 0$ at $p$. Using the long exact sequence in cohomology associated to \eqref{eqn:ses-Q} and the isomorphism \eqref{eqn:Q-local-stalks}, we get
\begin{align*}
    \fA = \bigoplus_{m\in W}(T_{C,p})^{\otimes m}.
\end{align*}
Via the choice of a nonzero tangent vector $x\in T_{C,p}$, we may identify $\fA$ with the sub-algebra $k[x^W]\subset k[x]$ generated by the monomials $\{x^m\}_{m\in W}$. The scheme $\Spec k[x^W]$ is called the \emph{monomial curve} associated to the semigroup $W$. Observe that the last assertion of Corollary \ref{cor:coord-ring-of-curve-sing-final} recovers the familiar fact (known as Weierstrass' gap theorem) that $\bZ_{\ge 0}\setminus W$ consists of exactly the $g(C)$ \emph{gap values} of $C$ at $p$.

\begin{Example}[Non-Weierstrass point]\label{exa:non-W-pt}
    If $p\in C$ is not a Weierstrass point, we have $W = \{0\}\cup\{j\in\bZ:j>g\}$, where $g:=g(C)\ge 1$. Note that the condition of being a non-Weierstrass point is vacuous when $g=1$.
    
    The semigroup $W$ is generated by the integers $j$ satisfying $g+1\le j\le 2g + 1$. Thus, $\fA$ can be identified with the image of the ring map $k[u_1,\ldots,u_{g+1}]\to k[x]$ given by $u_i\mapsto x^{g+i}$. Geometrically, this means that $\sY_0 = \Spec\fA$ is the image of the morphism $\bA^1\to\bA^{g+1}$ given by $x\mapsto(x^{g+1},\ldots,x^{2g+1})$.

\begin{figure}[ht]
  \subcaptionbox*{}[.14\linewidth]{%
    \includegraphics[width=\linewidth]{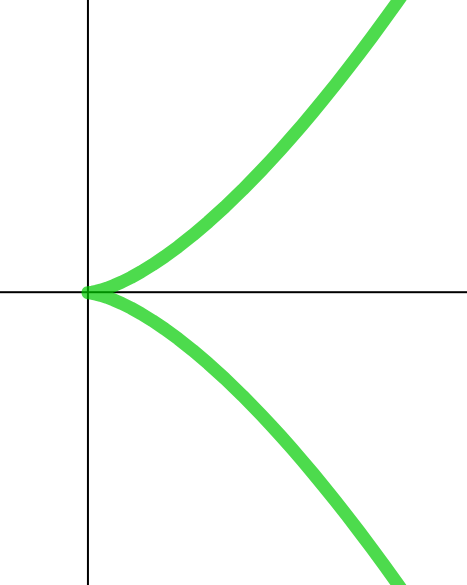}%
  }%
 \hskip18ex
  \subcaptionbox*{}[.31\linewidth]{%
    \includegraphics[width=\linewidth]{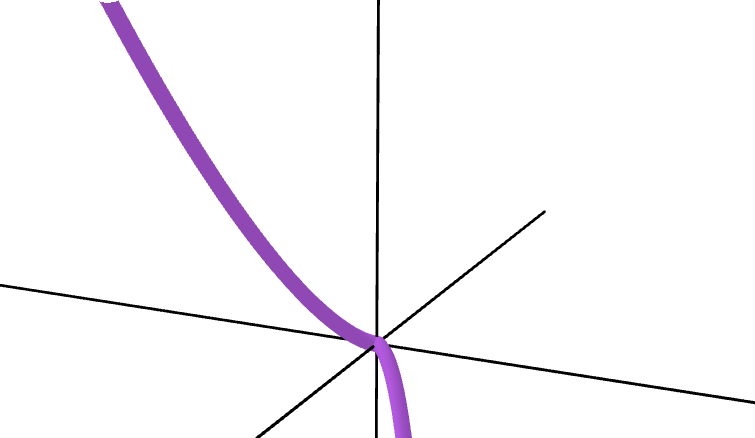}%
  }
  \caption{The model singularity, from Example \ref{exa:non-W-pt}, associated to a general point on a curve of genus $1$ (left) or genus $2$ (right).}
  \label{fig: non-W}
\end{figure}

  The model ghost condition takes a simple form in this case. It just states that the first $g$ derivatives of $f|_\sE$ at $p$ are required to vanish.
\end{Example}

\begin{Example}[Hyperelliptic fixed point]\label{exa:hyperelliptic}
    Assume $g:=g(C)\ge 2$, and that we have a degree $2$ morphism $C\to\bP^1$. Let $\iota$ be the generator of $\Aut(C/\bP^1) = \bZ/2\bZ$ and let $p\in C$ be a fixed point of $\iota$. When $g = 2$, the canonical linear series defines a degree $2$ morphism to $\bP^1$, in this case, $p$ is one of the six Weierstrass points on $C$.
    
    The semigroup $W$ is generated by $2$ and $2g + 1$. This means $\fA$ is the image of $k[u,v]\to k[x]$ given by $u\mapsto x^2$ and $v\mapsto x^{2g + 1}$ and it follows that $\sY_0 = \Spec \fA = \Spec k[u,v]/(v^2 - u^{2g + 1})$ is a higher order cusp.

    \begin{figure}[H]
    \centering
    \includegraphics[width=2.1cm]{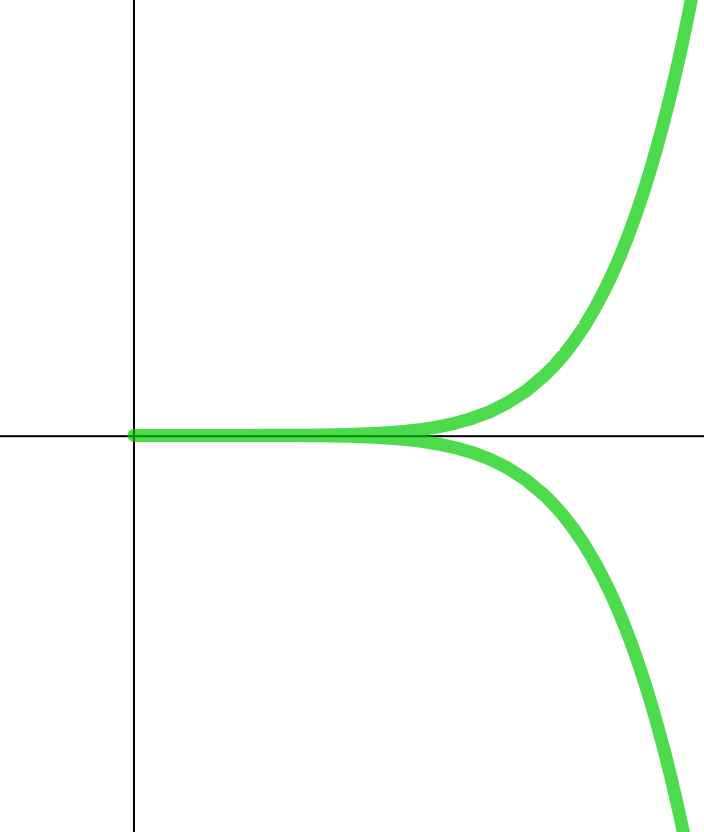}
    \caption{The higher order cusp $v^2=u^{13}$, from Example \ref{exa:hyperelliptic}, associated to a fixed point of the involution on a hyperelliptic curve of genus $6$.}
    \label{fig: hyperelliptic 2,2g+1}
\end{figure}

In this case, the model ghost condition says that there is a local coordinate $z\in \clO_{\sE,p}$ such that the Taylor expansion of $f|_\sE$ at $p$ contains no odd power of $z$ smaller than $z^{2g+1}$.

\end{Example}

\begin{Example}\label{exa:special-genus-3}
    Suppose $g:=g(C)\ge 3$ and let $p\in C$ be a point with Weierstrass semigroup $W = \bZ_{\ge 0}\setminus\{1,\ldots,g-1,g+1\}$. Such $(C,p)\in\clM_{g,1}$ exist by \cite[Theorem 1]{Eisenbud-Harris-Wpoint}. 
    
    The semigroup $W$ is generated by $g$, $2g+1$ and the integers $j$ satisfying $g+2\le j\le 2g-1$. As in Example \ref{exa:non-W-pt}, this set of generators shows that $\sY_0$ can be identified with the image of the morphism $\bA^1\to\bA^g$ given by $x\mapsto (x^g,x^{g+2},\ldots,x^{2g-1},x^{2g+1})$.
    
    For an explicit example, take $C\subset\bP^2$ be the non-singular genus $3$ curve defined by the polynomial $X_0^3X_2-X_1^4 + X_2^4$, with $[X_0:X_1:X_2]$ being the homogeneous coordinates on $\bP^2$. Let $p\in C$ be the point $[0:1:1]$. The local intersection multiplicity at $p$ of $C$ with any line in $\bP^2$ takes values in the set $\{0,1,3\}$. Since $C$ is a \emph{canonical curve},\footnote{This means that the embedding $C\subset\bP^2$ is defined by the linear system associated to the canonical bundle $\omega_C$ of the curve.} it follows that $W = \bZ_{\ge 0}\setminus\{1,2,4\}$. The semigroup $W$ is generated by $3,5$ and $7$. Thus, $\sY_0$ is the image of the morphism $\bA^1\to\bA^3$ given by $x\mapsto(x^3,x^5,x^7)$.
    \begin{figure}[ht]
    \centering
\includegraphics[width=2.8cm]{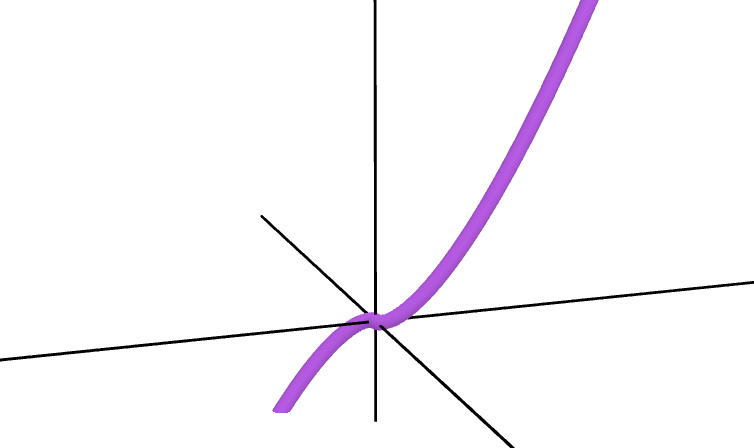}
    \caption{The model singularity, from Example \ref{exa:special-genus-3}, associated to a point on a curve of genus $3$ with Weierstrass semigroup $\bZ_{\ge 0}\setminus\{1,2,4\}$.}
    \label{fig: parametric curve 3,5,7}
\end{figure}

In this case, the model ghost condition can be divided into two assertions. First, we have the vanishing of the first $g-1$ derivatives of $f|_\sE$ at $p$. Second, there is a local coordinate $z\in\clO_{\sE,p}$ such that the Taylor expansion of $f|_\sE$ doesn't contain the power $z^{g+1}$.

\end{Example}

%%%%%%%%%%%%%%%%%%%%%%%%%%%%%%%%%%%%%%%%%%%%%%%%%%%%%%%%%
\subsubsection{Several marked points compared to genus}\label{subsubsec:ghost-with-many-eff}

Suppose that we have $n\ge 2\cdot g(C)-1$ and $d_1=\cdots=d_n=1$. In this case, we have $\Delta = \sum p_i$ and
\begin{align*}
    A_m = H^0(C,\clO_C(m\Delta))
\end{align*}
for $m\ge 0$. We have $\fA_0 = Q_0 = k$, and using Corollary \ref{cor:coord-ring-of-curve-sing-final} and $n>2\cdot g(C) - 2$, we get $\fA_m = Q_m = \bigoplus_{1\le i\le n} (T_{C,p_i})^{\otimes m}$ for $m\ge 2$. To complete the description of $\fA$, we are left to determine $\fA_1$. For this, use the long exact sequence in cohomology associated \eqref{eqn:ses-Q} with $m=1$ to obtain the short exact sequence
\begin{align}\label{eqn:ghost-with-many-eff}
    0\to\fA_1\to\bigoplus_{1\le i\le n} T_{C,p_i}\to H^1(C,\clO_C)\to 0.
\end{align}
Here, we have used $n>2\cdot g(C) - 2$ and Serre duality to get $H^1(C,\clO_C(\Delta)) = 0$. Now, using \eqref{eqn:ghost-with-many-eff} and Serre duality, we see that a tuple $(v_i)\in\bigoplus_{1\le i\le n} T_{C,p_i}$ lies in $\fA_1$ if and only if we have $\sum_i\alpha_{p_i}(v_i) = 0$ for all global regular $1$-forms $\alpha$ on $C$.

\begin{Example}[Genus $1$ with several points]\label{exa:genus-1-with-many-eff}
    Assume $g(C) = 1$ and $n\ge 2$. Using a trivialization of the tangent bundle $T_C$, we obtain canonical isomorphisms $T_{C,p_i} = T_{C,p_j}$ for $1\le i,j\le n$. Choose nonzero vectors $x_i\in T_{C,p_i}$ which correspond to each other under these isomorphisms. 
    
    Then, $\fA$ is the sub-algebra of
    $Q = k[x_1,\ldots,x_n]/(x_ix_j:1\le i<j\le n)$ generated by the elements $x_i - x_j$ for $1\le i,j\le n$ and $x_i^2$ for $1\le i\le n$. Breaking the symmetry between $x_1,\ldots,x_n$, we obtain $\fA$ as the image of $k[y_1,\ldots,y_n]\to Q$ given by $y_1\mapsto x_1^2$ and $y_i\mapsto x_1 - x_i$ for $2\le i\le n$. 
    
    Geometrically, this means that $\sY_0$ is embedded in $\bA^n = \Spec k[y_1,\ldots,y_n]$ and its $n$ irreducible components are the last $n-1$ coordinate axes and the image of the morphism $\bA^1\to\bA^n$ given by $x_1\mapsto(x_1^2,x_1,\ldots,x_1)$.
\begin{figure}[ht]
 \subcaptionbox*{}[.15\linewidth]{%
    \includegraphics[width=\linewidth]{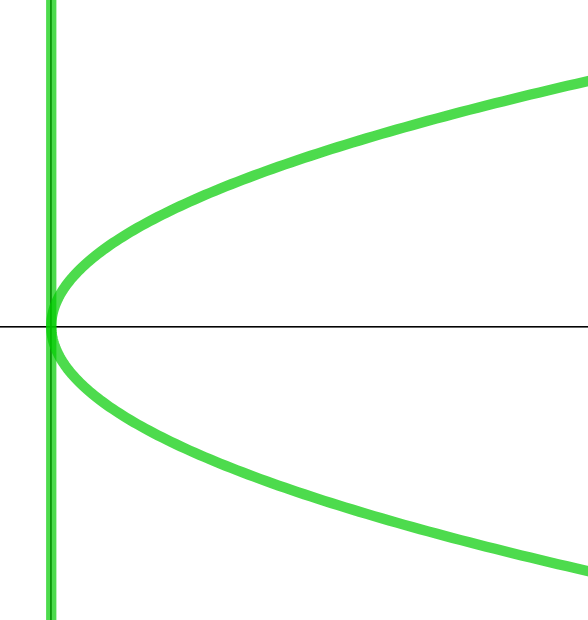}%
  }%
  \hskip18ex
  \subcaptionbox*{}[.23\linewidth]{%
    \includegraphics[width=\linewidth]{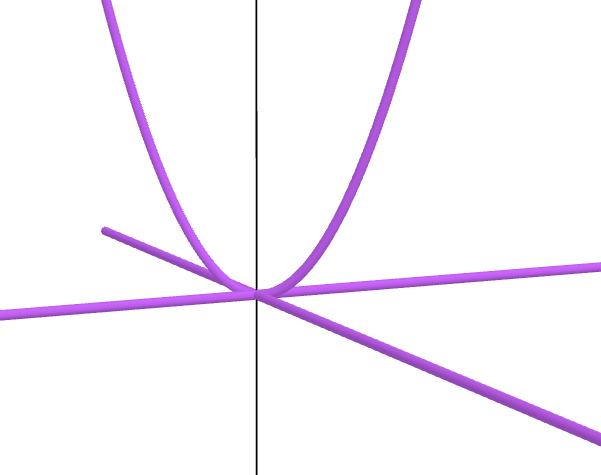}%
  }
  \caption{The model singularity, from Example \ref{exa:genus-1-with-many-eff}, associated to two points (left) or three points (right) on a curve of genus $1$.}
  \label{fig: genus1 with several points}
\end{figure}

In this case, the model ghost condition says that the first derivatives 
\begin{align*}
    d(f|_\sE)_{p_i}:T_{\sE,p_i}\to T_{X,q}
\end{align*} 
have linearly dependent images, where $q\in X$ is the image of the ghost $C$ under $f$. This is consistent with Figure \ref{fig: genus1 with several points}, where the branches of the model singularity have linearly dependent tangent lines at the origin.
    
\end{Example}

\begin{Example}[Genus $2$ with 3 points]\label{exa:genus-2-with-many-eff}
    Assume $g(C) = 2$ and $n = 3$. Choose nonzero vectors $x_i\in T_{C,p_i}$ for $1\le i\le 3$. After re-ordering $p_1,p_2,p_3$, we may assume that the images of $x_1$ and $x_2$, under the map from \eqref{eqn:ghost-with-many-eff}, form a basis of $H^1(C,\clO_C)$. This means that $p_1$ and $p_2$ are not exchanged by the hyperelliptic involution of $C$. Now, a basis of $\fA_1$ is given by the element
    \begin{align*}
        x_3 - a_1x_1 - a_2x_2
    \end{align*}
    with $a_1,a_2\in k$ being the coefficients in the linear expression for the image of $x_3$ in $H^1(C,\clO_C)$ in terms of the images of $x_1,x_2$. We now use this information to characterize $\fA$ as a sub-algebra of $Q = k[x_1,x_2,x_3]/(x_1x_2,x_1x_3,x_2x_3)$.

    The general situation is when $a_1$ and $a_2$ are both nonzero, i.e., no pair of points among $p_1,p_2,p_3$ is exchanged by the hyperelliptic involution. In this case, we scale $x_1$ and $x_2$ to assume that $a_1 = a_2 = -1$, i.e., $\fA_1$ is spanned by $x_1 + x_2 + x_3$. Thus, $\fA$ is the sub-algebra of $Q$ generated by the elements $x_1^2,x_2^2$, and $x_1+x_2+x_3$. Geometrically, this means that $\sY_0$ is embedded in $\bA^3$ and its three irreducible components are the the images of the morphisms $\bA^1\to\bA^3$ given by
    \begin{align*}
        x_1&\mapsto (x_1^2,0,x_1), \\
        x_2&\mapsto (0,x_2^2,x_2), \\
        x_3&\mapsto (0,0,x_3).
    \end{align*}

    \begin{figure}[ht]
    \centering
\includegraphics[width=2cm]{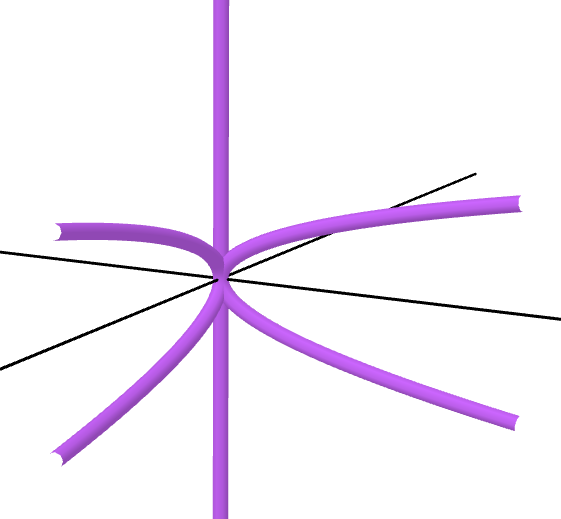}
    \caption{The model singularity, from Example \ref{exa:genus-2-with-many-eff}, associated to three general points on a curve of genus $2$.}
    \label{fig: Example genus2 with 3 points}
\end{figure}
In this general situation, the model ghost condition says that the images of the first derivatives $d(f|_{\sE})_{p_i}:T_{\sE,p_i}\to T_{X,q}$ are pairwise linearly dependent, for $1\le i\le 3$, where $q\in X$ is the image of the ghost $C$ under $f$. This is consistent with Figure \ref{fig: Example genus2 with 3 points}, where the three branches of the model singularity share a common tangent line at the origin.

    The special situation is when one of $a_1$ and $a_2$ is zero. Note that we can't have $a_1 = a_2 = 0$ since $g(C) > 0$. After re-ordering $p_1,p_2$ and scaling $x_1,x_2$, we may assume that $a_1 = 1$ and $a_2 = 0$. Thus, $\fA$ is the sub-algebra of $Q$ generated by the elements $x_1-x_3,x_1^2,x_2^2$, and $x_2^3$. Geometrically, we get $\sY_0\hookrightarrow\bA^4$ and its three irreducible components are the images of the morphisms $\bA^1\to\bA^4$ given by
    \begin{align*}
        x_1 &\mapsto (x_1,x_1^2,0,0) \\
        x_2 &\mapsto (0,0,x_2^2,x_2^3) \\
        x_3&\mapsto (-x_3,0,0,0).
    \end{align*}
    
    It is interesting to note that the first and third components lie in $\bA^2\times\{(0,0)\}$ and together recover the $n=2$ case of Example \ref{exa:genus-1-with-many-eff}, which is shown on the left half of Figure \ref{fig: genus1 with several points}. Similarly, the second component lies in $\{(0,0)\}\times\bA^2$ and recovers the ordinary cusp from Example \ref{exa:non-W-pt}, which is shown on the left half of Figure \ref{fig: non-W}. 
    
    In this special situation, the model ghost condition says that the first derivatives $d(f|_{\sE})_{p_1}$ and $d(f|_{\sE})_{p_3}$ have linearly dependent images in $T_{X,q}$ while the first derivative $d(f|_{\sE})_{p_2}$ vanishes. Here, $q\in X$ is the image of the ghost $C$ under $f$. Observe that this is a degenerate version of the model ghost condition appearing in the general situation considered above.
\end{Example}

%%%%%%%%%%%%%%%%%%%%%%%%%%%%%%%%%%%%%%%%%%%%%%%%%%%%%%%%%%
\subsubsection{Marked points form a canonical divisor}\label{subsubsec:ghost-with-special-eff}

Suppose $g(C)\ge 2$, $n = 2\cdot g(C) - 2$ and $d_1=\cdots=d_n=1$. If $\Delta = \sum p_i$ is not a canonical divisor, then we get $H^1(C,\clO_C(\Delta)) = 0$ by Serre duality and the situation is essentially identical to \textsection\ref{subsubsec:ghost-with-many-eff}. Therefore, we will assume that $\Delta$ is a canonical divisor. Pick a global regular $1$-form $\varphi$ on $C$ (unique up to scaling) such that $\Delta$ is its divisor of zeros. We have $\fA_0 = Q_0 = k$, and using Corollary \ref{cor:coord-ring-of-curve-sing-final} and $n = 2\cdot g(C) - 2$, we also get $\fA_m = Q_m = \bigoplus_{1\le i\le n} (T_{C,p_i})^{\otimes m}$ for $m\ge 3$. To complete the description of $\fA$, we must determine $\fA_1$ and $\fA_2$. Use the long exact sequences in cohomology associated to \eqref{eqn:ses-Q} with $m=1$ and $m=2$ to obtain the following two exact sequences.
\begin{align}\label{eqn:ghost-with-special-eff-1}
    0\to\fA_1\to\bigoplus_{1\le i\le n} T_{C,p_i}\to H^1(C,\clO_C)\to H^1(C,\clO_C(\Delta))\to 0.\\
    \label{eqn:ghost-with-special-eff-2}
    0\to\fA_2\to\bigoplus_{1\le i\le n} (T_{C,p_i})^{\otimes 2}\to H^1(C,\clO_C(\Delta))\to 0.
\end{align}
Note that $H^1(C,\clO_C(\Delta))$ is the quotient of $H^1(C,\clO_C)$ by the subspace annihilated by $\varphi$ under the Serre duality pairing. In particular, $H^1(C,\clO_C(\Delta))$ is $1$-dimensional. Exactly as in \textsection\ref{subsubsec:ghost-with-many-eff}, from \eqref{eqn:ghost-with-special-eff-1}, we see that a tuple $(v_i)\in\bigoplus_{1\le i\le n}T_{C,p_i}$ lies in the linear subspace $\fA_1$ of codimension $g(C)-1$ if and only if we have $\sum_i\alpha_{p_i}(v_i) = 0$ for all global regular $1$-forms $\alpha$ on $C$. Finally, using \eqref{eqn:ghost-with-special-eff-2} and Serre duality, it follows that a tuple $(w_i)\in \bigoplus_{1\le i\le n} (T_{C,p_i})^{\otimes 2}$ lies in the linear subspace $\fA_2$ of codimension $1$ if and only if we have 
\begin{align}
    \sum_{1\le i\le n} (\nabla_{p_i}\varphi)(w_i) = 0
\end{align}
where, for $1\le i\le n$, we define the element $\nabla_{p_i}\varphi\in (T_{C,p_i}^\vee)^{\otimes 2}$ by considering $\varphi$ as a section of $\omega_C(-\Delta)$ and taking its image in $\omega_C(-\Delta)\otimes_{\clO_C}(\clO_{C,p_i}/\fm_{C,p_i}) = (T_{C,p_i}^\vee)^{\otimes 2}$.

\begin{Example}[Genus $2$ with hyperelliptic conjugate points]\label{exa:genus-2-with-hyperelliptic-conjugate-points}
    Assume $g(C) = 2$, $n=2$ and that $p_1 + p_2$ is a canonical divisor. This means that $p_1$ and $p_2$ are exchanged by the hyperelliptic involution of $C$. Choose nonzero vectors $x_i\in T_{C,p_i}$ for $i=1,2$ such that the image of $x_1 + x_2$, under the map from \eqref{eqn:ghost-with-special-eff-1}, is zero in $H^1(C,\clO_C)$. Thus, $\fA_1$ has basis given by $x_1 + x_2$. 
    
    It follows that $\fA$ is the sub-algebra of $Q = \Spec k[x_1,x_2]/(x_1x_2)$ generated by the elements $x_1+x_2$ and $x_1^3$. Geometrically, this means that $\sY_0$ embeds in $\bA^2$ and its two irreducible components are the images of the morphisms $\bA^1\to\bA^2$ given by
    \begin{align*}
        x_1&\mapsto(x_1,x_1^3)\\
        x_2&\mapsto(x_2,0).
    \end{align*}
\begin{figure}[ht]
    \centering
\includegraphics[width=2.3cm]{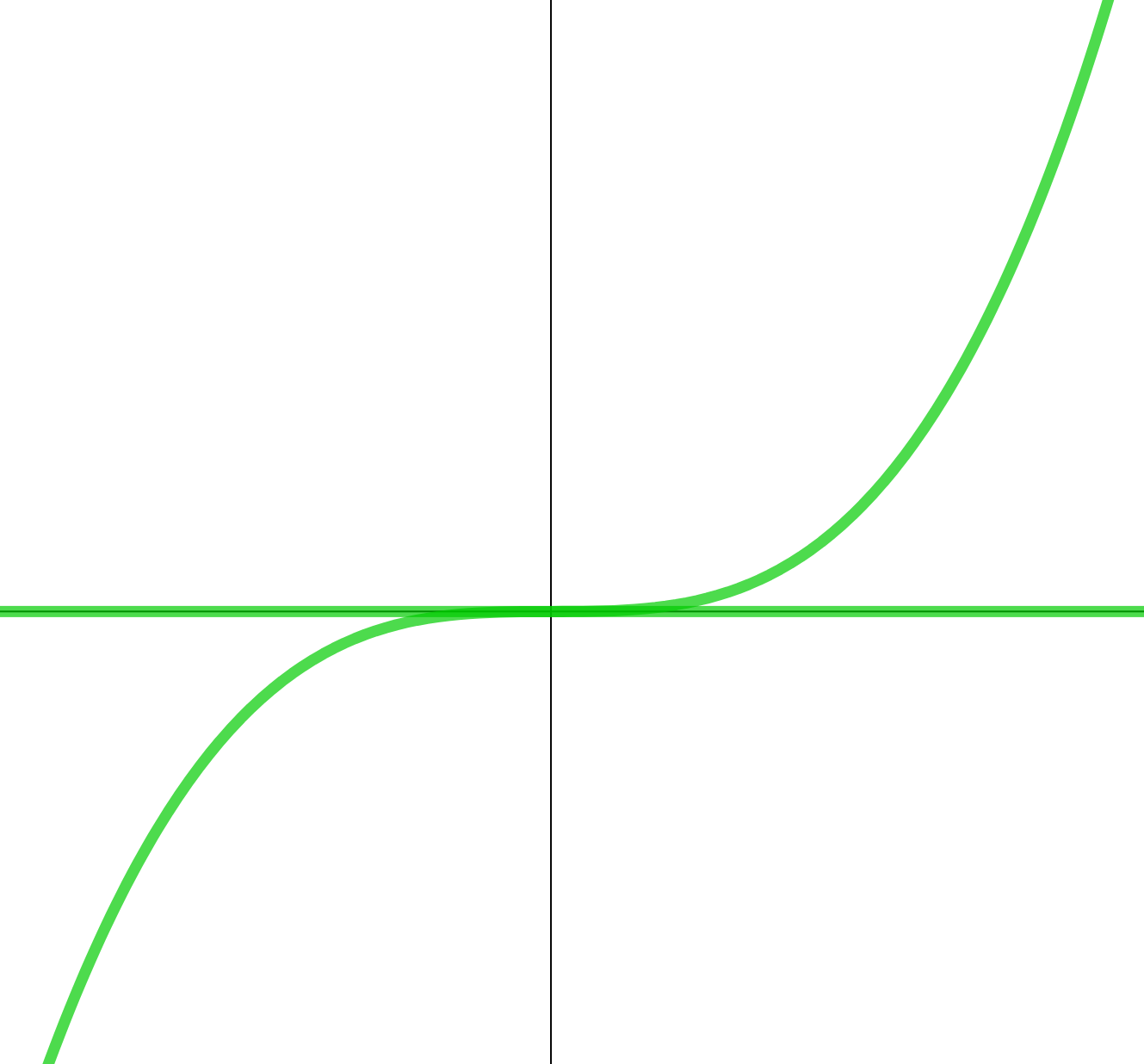}
    \caption{The model singularity, from Example \ref{exa:genus-2-with-hyperelliptic-conjugate-points}, associated to a pair of points on a genus $2$ curve which are exchanged by the hyperelliptic involution.}
    \label{fig:genus-2-with-hyperelliptic-conjugates}
\end{figure}
In this case, the model ghost condition says that there are tangent vectors $v_i\in T_{\sE,p_i}$, for $i=1,2$, such that the first two derivatives of $f$ at $p_1,p_2$ satisfy the following identities, where $q\in X$ is the image of the ghost $C$ under $f$.
\begin{align}
    d(f|_{\sE})_{p_1}(v_1) &+ d(f|_{\sE})_{p_2}(v_2) = 0 \in T_{X,q}.\\
    \label{eqn:genus-2-with-hyp-conj-2} d^2(f|_{\sE})_{p_1}(v_1^{\otimes 2}) &+ d^2(f|_{\sE})_{p_2}(v_2^{\otimes 2}) = 0\in T_{X,q}.
\end{align}
Informally, these identities say that the two branches of $f|_{\sE}$, at $p_1$ and $p_2$, have a second order tangency at $q\in X$. This is consistent with the behavior of the two branches of the model singularity shown in Figure \ref{fig:genus-2-with-hyperelliptic-conjugates}.

\end{Example}

%%%%%%%%%%%%%%%%%%%%%%%%%%%%%%%%%%%%%%%%%%%%%%%%%%%%%%%%%%
\subsubsection{Suspension}\label{subsubsec:suspension}

We give a procedure to create new model singularities out of old ones.

Start with $(C,p_1,\ldots,p_n)$ and $d_1,\ldots,d_n\ge 1$ as before and let $\fA$ be the coordinate ring of the model singularity (with $n$ branches) associated to this data. Using Corollary \ref{cor:coord-ring-of-curve-sing-final}, define $M\ge 1$ to be the smallest integer for which we have $\fA_m = Q_m$ for all $m\ge M$. Now, choose an integer $r\ge 1$ and $r$ distinct points $p_{n+1},\ldots,p_{n+r}\in C\setminus\{p_1,\ldots,p_n\}$ and integers $d_{n+1},\ldots,d_{n+r}\ge M$. Let $\fA'$ be the coordinate ring of the model singularity (with $n+r$ branches) associated to the data of $(C,p_1,\ldots,p_{n+r})$ and $d_1,\ldots,d_{n+r}$. Let $Q'$ be the coordinate ring of its seminormalization (as in Definition \ref{def:partial-norm}) with $\fA'\subset Q'$ being the induced inclusion of graded algebras. We can completely describe $\fA'$ in terms of $\fA$.

Since we chose $d_{n+1},\ldots,d_{n+r}\ge M$, we have $\fA'_m = \fA_m$ and $Q'_m = Q_m$ for $0\le m< M$. Corollary \ref{cor:coord-ring-of-curve-sing-final} shows that $Q/\fA = \bigoplus_{0\le m<M}Q_m/\fA_m$ and $Q'/\fA'$ are both $g(C)$-dimensional, implying that we must have $\fA'_m = Q'_m$ for $m\ge M$. We refer to the model singularity $\Spec\fA'$ as a \emph{suspension} of the model singularity $\Spec\fA$.

\begin{remark}[Explicit geometric description of a suspension]\label{rem:geom-suspension}
    Given a closed embedding $\sY_0 = \Spec\fA\subset\bA^N$ with $\overline{c}\in\sY_0$ mapping to $0\in\bA^N$, then we can describe the suspension $\Spec\fA'\subset\bA^N\times\bA^r = \bA^{N+r}$ as the union of the following two pieces. The first is the product of $\sY_0\subset\bA^N$ with $\{0\}\subset\bA^r$ while the second is the product of $\{0\}\subset\bA^N$ with the union of the $r$ coordinate axes in $\bA^r$.
\end{remark}

\begin{Example}[Suspending a cusp]\label{exa:suspend-cusp}
    Let $\fA$ be the coordinate ring of the model singularity associated to $C$ of genus $1$, $n=1$, a point $p_1\in C$ and $d_1 = 1$. The model singularity in this case was computed in Example \ref{exa:non-W-pt} and is the ordinary cusp $\Spec k[u,v]/(v^2-u^3)$. The model ghost condition says that $f|_{\sE}$ has vanishing first derivative at $p_1$.
    
    Now, let us suspend this using $r=1$, a point $p_2\in C\setminus\{p_1\}$ and $d_2 = 2$. Then, the suspended model singularity $\Spec\fA'$ is embedded in $\bA^3 = \Spec k[u,v,w]$ as the union of the $w$-axis and the cusp $\Spec\fA$, with the latter contained in the $uv$-plane. 
    
    In this case, the model ghost condition says that $f|_\sE$ has vanishing first derivative at $p_1$ and there is \emph{no additional condition} on $f|_\sE$ near $p_2$.
    \begin{figure}[ht]
  \subcaptionbox*{}[.18\linewidth]{%
    \includegraphics[width=\linewidth]{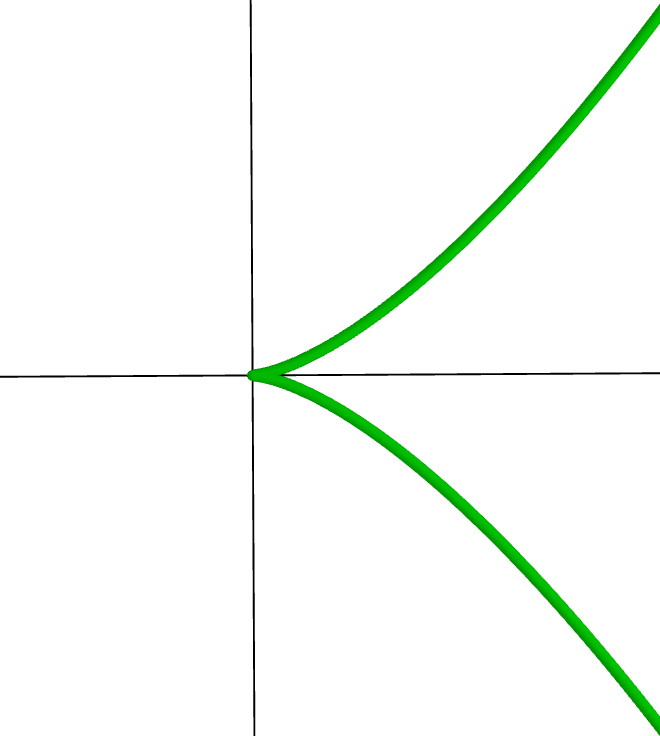}%
  }%
 \hskip18ex
  \subcaptionbox*{}[.25\linewidth]{%
    \includegraphics[width=\linewidth]{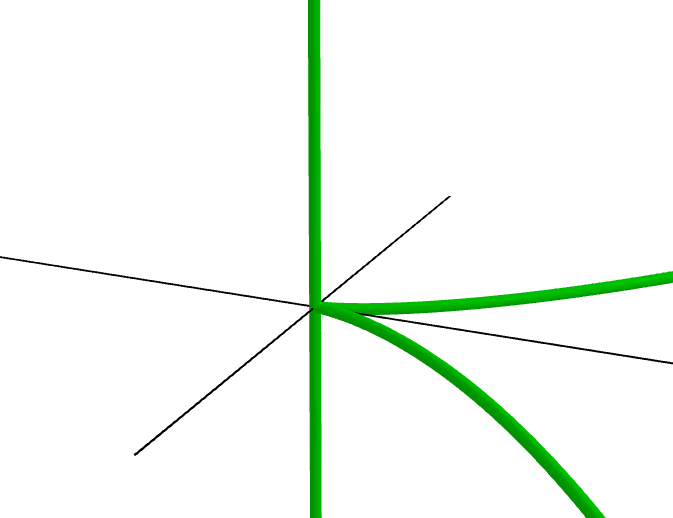}%
  }
  \caption{The ordinary cusp $v^2=u^3$ (left), and the model singularity, from Example \ref{exa:suspend-cusp}, obtained by suspending it (right).}
  \label{fig: Susp-cusp}
\end{figure}
\end{Example}

\begin{Example}[Suspending a tacnode]\label{exa:suspend-tacnode}
    Let $\fA$ be the coordinate ring of the model singularity associated to $C$ of genus $1$, $n=2$, distinct points $p_1,p_2\in C$, and $d_1 = d_2 = 1$. The model singularity in this case was computed in Example \ref{exa:genus-1-with-many-eff} and is a tacnode given by $\Spec k[u,v]/(u^2-uv^2)$, shown on the left half of Figure \ref{fig: genus1 with several points}. The model ghost condition says that the first derivatives of $f|_{\sE}$ at the points $p_1,p_2$ have linearly dependent images.
    
    Now, let us suspend this using $r=1$, a point $p_3\in C\setminus\{p_1,p_2\}$ and $d_3 = 2$. Then, the suspended model singularity $\Spec\fA'$ is embedded in $\bA^3 = \Spec k[u,v,w]$ as the union of the $w$-axis and the tacnode $\Spec\fA$, with the latter contained in the $uv$-plane. 
    
    In this case, the model ghost condition says that the first derivatives at of $f|_\sE$ at $p_1,p_2$ have linearly dependent images and there is \emph{no additional condition} on $f|_\sE$ near $p_3$.

        \begin{figure}[ht]
  \subcaptionbox*{}[.20\linewidth]{%
    \includegraphics[width=\linewidth]{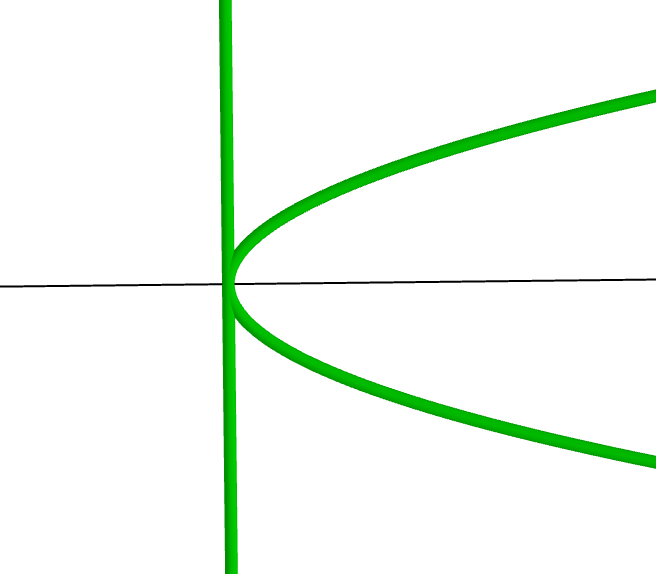}%
  }%
 \hskip18ex
  \subcaptionbox*{}[.22\linewidth]{%
    \includegraphics[width=\linewidth]{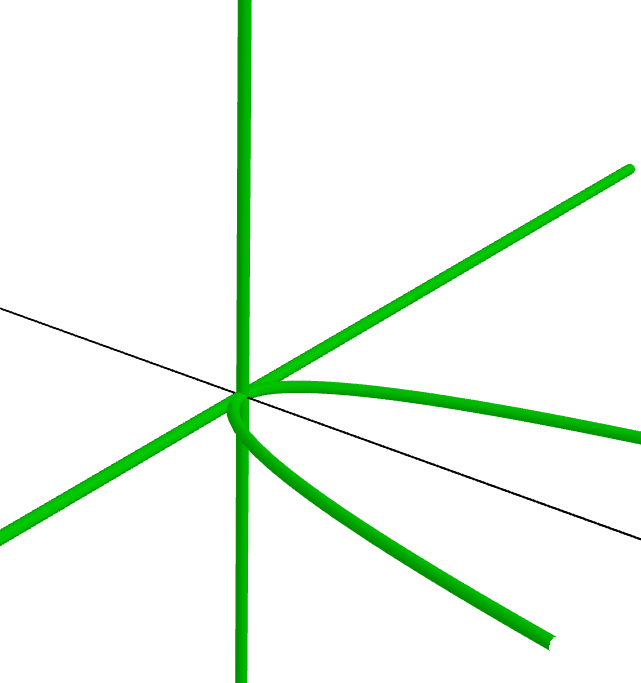}%
  }
  \caption{The ordinary tacnode $u^2=uv^2$ (left), and the model singularity, from Example \ref{exa:suspend-tacnode}, obtained by suspending it (right).}
  \label{fig: tacnode-cusp}
\end{figure}
\end{Example}

\begin{remark}[Model ghost condition for suspension]\label{rem:model-ghost-for-suspension}
    We can generalize the observation of Examples \ref{exa:suspend-cusp} and \ref{exa:suspend-tacnode} as follows. The model ghost condition for the suspension $\Spec\fA'$ amounts to the same condition on $f|_\sE$ near $p_1,\ldots,p_n$ as for $\Spec\fA$ and \emph{no additional condition} on $f|_\sE$ near $p_{n+1},\ldots,p_{n+r}$.
    
    This can be informally explained as follows. A ghost component modeled on $\Spec \fA'$ may be smoothed in two steps as follows. To begin, we smooth the nodes at $p_1,\ldots,p_n$, using the model $\Spec\fA$. Now that the ghost is gone, we are free to smooth the remaining nodes $p_{n+1},\ldots,p_{n+r}$ without any additional conditions. To carry out this $2$-step procedure within a single $1$-parameter family, we simply \emph{postpone} smoothing the nodes $p_{n+1},\ldots,p_{n+r}$ until the smoothing of the nodes $p_1,\ldots,p_n$ is done, i.e., we take the rates of smoothing $1/d_{n+1},\ldots,1/d_{n+r}$ to be small in comparison to $1/d_1,\ldots,1/d_n$.
\end{remark}

\begin{remark}\label{rem:suspension-increases-d}
    The examples given so far (with the exception of Examples \ref{exa:suspend-cusp} and \ref{exa:suspend-tacnode}) had $d_i = 1$ for $1\le i\le n$. We can use suspension to convert any of the earlier examples into new examples having $d_i>1$ for some $i$. Example \ref{exa:genus-3-with-2-points-d=2,3} below exhibits a model singularity, not arising from a suspension, which has $d_i>1$ for all $i$.
\end{remark}

%%%%%%%%%%%%%%%%%%%%%%%%%%%%%%%%%%%%%%%%%%%%%%%%%%%%%%%%%
\subsubsection{An example with all $d_i>1$}\label{subsubsec:all-d-bigger-than-1}
Let $C\subset\bP^2$ be a non-singular curve of genus $3$. Then, $C$ is a canonical curve. Fix two general points $p_1,p_2\in C$. Explicitly, we require that the tangent line to $C\subset\bP^2$ at $p_1$ doesn't contain $p_2$ (and vice versa). Choose nonzero vectors $x_i\in T_{C,p_i}$, for $i=1,2$. We will consider the model singularity associated to $n=2$, $p_1,p_2\in C$, and $(d_1,d_2)=(2,3)$. Before this, it is interesting to work out the case $(d_1,d_2)=(1,1)$.

\begin{Example}\label{exa:genus-3-with-2-points-d=1}
    Take $d_1 = d_2 = 1$. Then, we have $\Delta = p_1 + p_2$. We determine the coordinate ring $\fA$ of the associated model singularity $\sY_0$ as follows.
    
    There is a unique line in $\bP^2$ joining $p_1$ and $p_2$. We thus have $\dim H^1(C,\clO_C(\Delta)) = \dim H^0(C,\omega_C(-\Delta)) = 1$, by Serre duality. The exact sequence \eqref{eqn:ghost-with-special-eff-1} now gives $\fA_1 = 0$. Since $p_1,p_2\in C$ are general points, the tangent lines to $C\subset\bP^2$ at these two points are distinct. Again, a Serre duality argument shows that $H^1(C,\clO_C(m\Delta)) = 0$ for $m\ge 2$. Thus, we have $\fA_m = Q_m = \bigoplus_{1\le i\le 2}T_{C,p_i}^{\otimes m}$ for $m\ge 3$. The exact sequence \eqref{eqn:ghost-with-special-eff-2} shows that $\fA_2$ is identified with the $1$-dimensional kernel of the surjection
    \begin{align}\label{eqn:genus-3-with-2-points-d=1}
        T_{C,p_1}^{\otimes 2}\oplus T_{C,p_2}^{\otimes 2}\twoheadrightarrow H^1(C,\clO_C(\Delta)).
    \end{align}
    Since $p_2$ is not on the tangent line to $C\subset\bP^2$ at $p_1$, Serre duality implies that we have $\dim H^1(C,\clO_C(2p_1+p_2)) = \dim H^0(C,\omega_C(-2p_1-p_2)) = 0$. Using the long exact sequence in cohomology associated to the short exact sequence
    \begin{align*}
        0\to\clO_C(\Delta)\to\clO_C(2p_1+p_2)\to T_{C,p_1}^{\otimes 2}\to 0,
    \end{align*}
    we find that the restriction of \eqref{eqn:genus-3-with-2-points-d=1} to $T_{C,p_1}^{\otimes 2}$ is an isomorphism. By symmetry, the restriction of \eqref{eqn:genus-3-with-2-points-d=1} to $T_{C,p_2}^{\otimes 2}$ is also an isomorphism. Thus, after replacing $x_1$ and $x_2$ by suitable scalar multiples, we may assume that $\fA_2$ has basis $x_1^2 + x_2^2$.

    Thus, $\fA$ is the sub-algebra of $Q = k[x_1,x_2]/(x_1x_2)$ generated by the elements $x_1^2 + x_2^2$, $x_1^3$, $x_2^3$, and $x_1^4$. Geometrically, this means that $\sY_0$ is embedded in $\bA^4$ and its two irreducible components are the images of the morphisms $\bA^1\to\bA^4$ given by
    \begin{align*}
        x_1&\mapsto(x_1^2,x_1^3,0,x_1^4)\\
        x_2&\mapsto(x_2^2,0,x_2^3,0).
    \end{align*}
    
    In this case, the model ghost condition can be divided into two assertions. First, for $i=1,2$, we have the vanishing of the first derivatives $d(f|_\sE)_{p_i}$. Second, there are vectors $v_i\in T_{\sE,p_i}$, for $i=1,2$, such that the identity \eqref{eqn:genus-2-with-hyp-conj-2} for second derivatives holds.
\end{Example}

\begin{Example}\label{exa:genus-3-with-2-points-d=2,3}
    Take $d_1 = 2$ and $d_2 = 3$. Then, we have $\Delta = \frac12p_1 + \frac13p_2$. The table below computes the divisor $\lfloor m\Delta\rfloor$ for $0\le m\le 6$.
    \begin{center}
    \begin{tabular}{|p{0.9cm}||p{0.5cm}|p{0.5cm}|p{0.5cm}|p{1.2cm}|p{1.3cm}|p{1.3cm}|p{1.3cm}|}
        \hline
        $m$ & $0$ & $1$ & $2$ & $3$ & $4$ & $5$ & $6$\\
        \hline
        $\lfloor m\Delta\rfloor$ & $0$ & $0$ & $p_1$ & $p_1 + p_2$ & $2p_1+p_2$ & $2p_1 + p_2$ & $3p_1 + 2p_2$ \\
        \hline
    \end{tabular}
    \end{center}
    We get $\fA_1 = \fA_5 = 0$ from the table. We have $\dim H^1(C,\clO_C(\lfloor2\Delta\rfloor)) = 2$, since this is identified with $\dim H^0(C,\omega_C(-p_1))$ by Serre duality. The computations in Example \ref{exa:genus-3-with-2-points-d=1} show that $\dim H^1(C,\clO_C(\lfloor3\Delta\rfloor)) = 1$ and $\dim H^1(C,\clO_C(\lfloor4\Delta\rfloor)) = 0$. We now get $\dim H^0(\clO_C(\lfloor m\Delta\rfloor)) = 1$ for $1\le m\le 4$, by the Riemann--Roch formula and it follows that $\fA_2 = \fA_3 = \fA_4 = 0$. By Corollary \ref{cor:coord-ring-of-curve-sing-final}, we have $\dim Q/\fA = 3$, and therefore, it must be the case that we have $\fA_m = Q_m$ for $m\ge 6$.

    Thus, $\fA$ is the sub-algebra of $Q = k[x_1,x_2]/(x_1x_2)$ generated by the elements $x_1^3,x_1^4,x_1^5,x_2^2$ and $x_2^3$. Geometrically, this means that $\sY_0$ is embedded in $\bA^5$ and its two irreducible components are the images of the morphisms $\bA^1\to\bA^5$ given by
    \begin{align*}
        x_1&\mapsto(x_1^3,x_1^4,x_1^5,0,0)\\
        x_2&\mapsto(0,0,0,x_2^2,x_2^3).
    \end{align*}
    
    It is interesting to note that the first component lies in $\bA^3\times\{(0,0)\}$ and recovers the $g=2$ case of Example \ref{exa:non-W-pt}, which is shown on the right half of Figure \ref{fig: non-W}. Similarly, the second component lies in $\{(0,0,0)\}\times\bA^2$ and recovers the $g=1$ case of Example \ref{exa:non-W-pt}, which is shown on the left half of Figure \ref{fig: non-W}.

    In this case, the model ghost condition can be divided into two assertions. First, we have the vanishing of the first two derivatives of $f|_\sE$ at $p_1$. Second, we have the vanishing of the first derivative of $f|_\sE$ at $p_2$. Observe that this is a degenerate version of the model ghost condition in Example \ref{exa:genus-3-with-2-points-d=1}.
\end{Example}
\appendix
\setcounter{section}{0}
\section{Analysis of model smoothing}\label{appendix:analysis-model-smoothing} 

In this appendix, we carry out a detailed study of the geometry of the family $\pi_\sY:\sY\to\bA^1$ introduced in Definition \ref{def:model-smoothing}. This involves the construction of another family $\pi_\sX:\sX\to\bA^1$, such that $\pi_\sX^{-1}(0)$ has only ordinary nodes as singularities, and a proper birational morphism $\Phi:\sX\to\sY$ of $\bA^1$-schemes. The key statements from this appendix which are needed in the main body of the paper (specifically, in \textsection\ref{subsec:model-sing-and-smoothing}) are Lemma \ref{lem:partial-normalization-ring-map} and Proposition \ref{prop:local-smoothability-of-model-pinching-appx}.

%%%%%%%%%%%%%%%%%%%%%%%%%%%%%%%%%%%%%%%%%%%%%%%%%%%%%%%%
\subsection*{Notation for this appendix}
\begin{longtable}{ r l }
   $C$& Smooth projective curve of genus $g(C)\ge 1$ (Notation \ref{not:model-sing-input}).\\
   $n$ & Positive integer denoting number of points on $C$ (Notation \ref{not:model-sing-input}).\\
   $p_i$& Distinct closed points on $C$ for $1\le i\le n$ (Notation \ref{not:model-sing-input}).\\
   $d_i$& Relatively prime positive integers for $1\le i\le n$ (Notation \ref{not:model-sing-input}).\\
   $\Delta$& $\bQ$-divisor $\sum\frac 1{d_i}\cdot p_i$ on $C$ (Notation \ref{not:model-sing-input}).\\
   $A$  & Graded algebra $\bigoplus_{m\ge 0}A_m$ with $A_m = H^0(C,\clO_C(\lfloor m\Delta\rfloor)$ (Definition \ref{def:model-smoothing}).\\
   & $A$ is a finitely generated $k$-algebra (Lemma \ref{lem:A-finite-generation}).\\   
   $t$ & Fixed element of $A_1\subset A$ and coordinate on $\bA^1$ (Definition \ref{def:model-smoothing}). \\
   $\pi_\sY:\sY\to\bA^1$ & Model smoothing with $\sY = \Spec A$ (Definition \ref{def:model-smoothing}, Figure \ref{fig:appendixA}).\\
   & $\sY$ is a normal surface (Corollary \ref{cor:contraction-normal}).\\
   $\fA$ & Graded algebra $\bigoplus_{m\ge 0}\fA_m$ with $\fA_m = A_m/A_{m-1}$ (Definition \ref{def:model-sing}).\\
   $\sY_0$ & Model singularity, $\sY_0=\Spec(\fA)$ (Definition \ref{def:model-sing}, Figure \ref{fig:Partial normalization}).\\
   &  $\sY_0=\pi_\sY^{-1}(0)$ (Lemma \ref{lem:model-sing}). $\sY_0$ inherits $\bG_m$-action from $\sY$.\\
   $A_+$ & Maximal ideal $\bigoplus_{m>0}A_m$ of $A$ (Definition \ref{def:cone-point}).\\
   $\overline{c}$ & Closed point of $\sY_0$ corresponding to $A_+\subset A$ (Definition \ref{def:cone-point}).\\
    $\gamma:E\to\sY_0$ & Seminormalization of $\sY_0$ (Definition \ref{def:partial-norm}).\\
   $Q$ & Graded coordinate ring $\bigoplus_{m\ge 0}Q_m$ of $E$ (Definition \ref{def:partial-norm}).\\
   & $\fA\subset Q$ is a graded sub-algebra and $\dim Q/\fA = g(C)$ (Corollary \ref{cor:coord-ring-of-curve-sing-final}).\\
    $U_i$ & Element of a fixed affine open cover of $C$ for $0\le i\le n$ (Notation \ref{not:C-cover}). \\
   & For $1\le i,j\le n$, $p_j\in U_i$ iff $i=j$. Moreover, $U_0 = C\setminus\{p_1,\ldots,p_n\}$. \\
   $R_i$ & Coordinate ring of $U_i$ for $0\le i\le n$ (Notation \ref{not:C-cover}).\\
   $z_i$ & Fixed element of $R_i$ vanishing precisely at $p_i$ for $1\le i\le n$ (Notation \ref{not:C-cover}).\\
   
   $\pi_{\overline\sX}:\overline\sX\to\bA^1$ &  $\bG_m$-equivariant degeneration of $C$ (Definition \ref{def:main-family}, Remark \ref{rem:Gm-equivariance}).\\
   $\overline{E}_i$ & Inverse image of $(p_i,0)$ under the morphism $\overline\sX\to C\times\bA^1$ (Definition \ref{def:main-family}).\\
   & $\overline{E}_i$ is the projectivization of $T_{C,p_i}\oplus (T_{\bA^1,0})^{\otimes d_i}$ (Footnote \ref{fn:excep}).\\
   $\pi_{\sX}:\sX\to\bA^1$ &  Restriction of $\pi_{\overline\sX}$ to the $\bG_m$-invariant open subset\\
   & $\sX\subset\overline\sX$ (Definition \ref{def:main-family-int}, Figure  \ref{fig:appendixA}).\\
   $C_0$ & Proper transform of $C\times\{0\}$ contained in $\sX$ (Definition \ref{def:main-family-int}).\\ $E_i$ & Open subset $\overline{E}_i\cap\sX$ of $\overline{E}_i$ (Definition \ref{def:main-family-int}).\\
   $V_i$ & Element of a fixed affine open cover of $\sX$ for $0\le i\le n$ (Notation \ref{not:X-int-cover}).\\
   & $V_0 = U_0\times\bA^1 = \Spec (R_0[t])$ and $V_i = \sX|_{U_i\times\bA^1}$ for $1\le i\le n$.\\
   
   $\Phi:\sX\to\sY$ & Contraction morphism with $\Phi(C_0) = \{\overline{c}\}$ (Definition \ref{def:contraction}, Figure  \ref{fig:appendixA}).\\
   & $\Phi$ induces an isomorphism $\sX\setminus C_0 \simeq \sY\setminus\{\overline{c}\}$ (Lemma \ref{lem:contraction}).
\end{longtable}

\begin{notation}[Affine open cover of $C$]\label{not:C-cover}
    For each $1\le i\le n$, fix an open affine neighborhood $p_i\in U_i = \Spec R_i\subset C$ such that $p_j\not\in U_i$ for $j\ne i$ and an element $z_i\in R_i$ whose scheme-theoretic zero locus is the reduced point $p_i\in U_i$. Note that the image of $z_i$ in $\clO_{C,p_i}$ is a local coordinate. Also recall $U_0 = C\setminus\{p_1,\ldots,p_n\} = \Spec R_0$ from Notation \ref{not:model-sing-input}. The collection $\{U_i\}_{0\le i\le n}$ forms an affine open cover of $C$. This will be used for coordinate computations.    
\end{notation}

\begin{Definition}[$1$-parameter degeneration $\overline\sX$ of $C$]\label{def:main-family}
    Define the family
    \begin{align*}
        \pi_{\overline\sX}:\overline{\sX}\to\bA^1 = \Spec k[t]
    \end{align*}
    as follows. $\overline\sX$ is the blow-up of $C\times\bA^1$ along the disjoint union of the $0$-dimensional closed subschemes $\{p_i\}\times\Spec k[t]/(t^{d_i})$ for $1\le i\le n$. The projection $\pi_{\overline\sX}$ is given by the blow-down morphism $\overline\sX\to C\times\bA^1$ followed by the coordinate projection $C\times\bA^1\to\bA^1$. For $1\le i\le n$, denote by $\overline{E}_i\subset\overline\sX$ the inverse image of the point $(p_i,0)\in C\times\bA^1$.
\end{Definition}

We can describe $\pi_{\overline\sX}:\overline\sX\to\bA^1$ in coordinates as follows. Using the open cover fixed in Notation \ref{not:C-cover}, we get the explicit description\footnote{Recall that, for an affine variety $T = \Spec S$ and an ideal $I = (f_0,\ldots,f_r)\subset S$, the blow-up $\text{Bl}_Z(T) = \Proj\bigoplus_{m\ge 0} I^m$ of $T$ along the closed subscheme $Z = \Spec S/I$ has a closed embedding into $\bP^r_S = \Proj S[x_0,\ldots,x_r]$ given by $x_i\mapsto f_i$, for $0\le i\le r$. The image of this embedding is the closure of the graph of the rational map $[f_0:\cdots:f_n]:T\dashrightarrow\bP^r$.} 
\begin{align}\label{eqn:main-family-coord}
  \overline\sX|_{U_i\times\bA^1} = \Proj\left(\frac{R_i[t,u_i,v_i]}{(z_iv_i-t^{d_i}u_i)}\right)\subset U_i\times\bA^1\times\bP^1
\end{align}
for $1\le i\le n$, where the $\Proj$ construction is done with respect to the grading which assigns degree $0$ to $R_i[t]$ and degree $1$ to the variables $u_i,v_i$. From \eqref{eqn:main-family-coord}, we get the identification $\overline{E}_i = \Proj(k[u_i,v_i])$. The blow-down morphism $\overline\sX\to C\times\bA^1$ is projective, birational, and is an isomorphism over the complement of $\{p_1,\ldots,p_n\}\times\{0\}\subset C\times\bA^1$.

\begin{remark}\label{rem:Gm-equivariance}
    The natural action of $\bG_m = \Spec k[t^{\pm1}]$ on $C\times\bA^1$ (trivial on $C$ and by scaling on $\bA^1$) lifts to an action on $\overline\sX$, making $\pi_{\overline\sX}$ a $\bG_m$-equivariant morphism. This is because the closed subscheme of $C\times\bA^1$ which is blown up to define $\overline\sX$ is itself $\bG_m$-invariant. $\bG_m$-equivariance of $\pi_{\overline\sX}$ shows that its fibres over all closed points with $t\ne 0$ are isomorphic to $C$. The fibre over $t = 0$ is the prestable curve whose irreducible components consist of the proper transform of $C\times\{0\}$ along with the rational curves $\overline{E}_i$ for $1\le i\le n$.
\end{remark}

\begin{remark}[Flatness of $\pi_{\overline\sX}$]\label{rem:X-flatness}
    By Remark \ref{rem:flatness-over-dvr}, the morphism $\pi_{\overline\sX}$ is flat.
\end{remark}

\begin{Definition}[Interior $\sX$ of $\overline\sX$]\label{def:main-family-int}
    Define the open subvariety $\sX\subset\overline\sX$ to be the complement in $\overline\sX$ of the proper transforms $H_i$ of $\{p_i\}\times\bA^1\subset C\times\bA^1$ for $1\le i\le n$. Let $\pi_{\sX}:\sX\to\bA^1$ denote the restriction of $\pi_{\overline\sX}$. Denote the proper transform of $C\times\{0\}\subset C\times\bA^1$ by $C_0\subset\sX$. Write $E_i := \overline{E}_i\cap\sX$ and $\{q_i\} := E_i\cap C_0$ for $1\le i\le n$. 
\end{Definition}

In terms of the coordinate presentation \eqref{eqn:main-family-coord} of $\overline\sX$, the proper transform $H_i$ of $\{p_i\}\times\bA^1$ is given by the equation $\{u_i = 0\}$ and $\sX|_{U_i\times\bA^1} = (\overline\sX\setminus H_i)|_{U_i\times\bA^1}$ for $1\le i\le n$. Introducing $w_i = v_i/u_i$, we get the explicit description
\begin{align}\label{eqn:main-family-int-coord}
    \sX|_{U_i\times\bA^1} = \Spec\left(\frac{R_i[t,w_i]}{(z_iw_i-t^{d_i})}\right)
\end{align}
for $1\le i\le n$. Note that $\sX|_{U_i\times\bA^1}$ is $\bG_m$-invariant and the action of $\bG_m$ is encoded by the grading which assigns degree $0$ to $R_i$, degree $1$ to $t$, and degree $d_i$ to $w_i$. In terms of \eqref{eqn:main-family-int-coord}, the curves $C_0$ and $E_i$ are given by the respective ideals $(w_i, t)$ and $(z_i,t)$. It follows from \eqref{eqn:main-family-int-coord} that the surface $\sX$ has an $A_{d_i-1}$  singularity\footnote{For $m\ge 0$, a point $p$ on a surface is an $A_m$ singularity if it is \'etale locally modeled on the normal toric variety $\Spec k[z,w,t]/(zw-t^{m+1})$, with $p$ given by $(z,w,t)$. An \emph{$A_0$ singularity} is a non-singular point.} at $q_i$ for $1\le i\le n$ and is non-singular away from these $n$ points. Normality is local in the smooth topology by \cite[\href{https://stacks.math.columbia.edu/tag/034F}{Tag 034F}]{stacks-project}, which shows $\sX$ is normal. Additionally, since $\overline{E}_i$ is the projectivization\footnote{Observe that the conormal module of $\Spec k[t]/(t^{d_i})\subset\bA^1$ is given by $(t^{d_i})/(t^{2d_i})$ and it restricts, over the reduced point $0\in\bA^1$, to the vector space $(t^{d_i})/(t^{d_i+1}) = (T_{\bA^1,0}^\vee)^{\otimes d_i}$.\label{fn:excep}} of the $\bG_m$-representation $T_{C,p_i}\oplus (T_{\bA^1,0})^{\otimes d_i}$, we get the identification $E_i =  T_{C,p_i}^\vee$ under which $q_i = 0$.

\begin{notation}[Affine open cover of $\sX$]\label{not:X-int-cover}
    Fix the affine open cover $\{V_i\}_{0\le i\le n}$ of $\sX$ given by $V_0 := U_0\times\bA^1 = \Spec(R_0[t])$, and for $1\le i\le n$, $V_i:=\sX|_{U_i\times\bA^1}$ described in \eqref{eqn:main-family-int-coord}. This cover by $\bG_m$-invariant open subsets will be used in coordinate computations.
\end{notation}

We would like to obtain a (singular) surface by contracting the curve $C_0\subset\sX$ to a point. In general, contracting a curve in a quasi-projective surface may take us outside of the category of quasi-projective varieties, e.g., see \cite[Example 3.1]{Badescu-book}. When $k=\bC$, it is possible to obtain the contraction as a complex analytic space \cite{Grauert-contraction}, and in general, it can be obtained as an algebraic space \cite{Artin-modifications}. Fortunately, in our special situation, it will turn out that contracting $C_0\subset\sX$ to a point results in $\sY$ (see Lemma \ref{lem:contraction}).

We first need to construct the morphism $\Phi:\sX\to\sY$ which exhibits $\sY$ as the result of contracting $C_0\subset\sX$ to a point. Since $\sY$ is an affine variety, this amounts to constructing the corresponding ring map $A\to\Gamma(\sX,\clO_{\sX})$. The next lemma accomplishes this task by constructing a canonical isomorphism of these rings. To understand its statement, note that we have an inclusion $\Gamma(\sX,\clO_{\sX})\subset R_0[t]$ of graded $k[t]$-algebras (via the birational \emph{morphism} $U_0\times\bA^1\hookrightarrow\sX$), and also recall Remark \ref{rem:model-smoothing-subalg} which gives the inclusion $A\subset R_0[t]$.

\begin{Lemma}\label{lem:coord-ring-iso}
    As graded $k[t]$-sub-algebras of $R_0[t]$, we have
    \begin{align}
        \label{eqn:coord-ring-iso}
        A = \Gamma(\sX,\clO_{\sX}).
    \end{align}
\end{Lemma}
\begin{proof}
  The affine open cover $\{V_i\}_{0\le i\le n}$ of $\sX$ fixed in Notation \ref{not:X-int-cover} consists of $\bG_m$-invariant open subsets and can therefore be used to compute $B=\Gamma(\sX,\clO_{\sX})$ as a graded $k[t]$-algebra. Since each $\Gamma(V_i,\clO_{V_i})$ is non-negatively graded, the same is true for $B$. To follow the computations below, recall that we fixed an open cover $\{U_i\}_{0\le i\le n}$ along with a local coordinate $z_i$ at $p_i$ for $1\le i\le n$ in Notation \ref{not:C-cover}.
    
    Fix $m\ge 0$ and denote by $B_m$ the summand of $B$ in grading $m$. By restricting to the open subset $V_0$, we may write any element $g_m\in B_m$ uniquely as $f_m t^m$, for some $f_m\in R_0$. Thus, the image of the (injective) map $B_m\to R_0$ (given by $g_m\mapsto f_m$) consists of those $F\in R_0$ for which the expression $Ft^m$ on $V_0$ extends as a regular function on the remaining open subsets $V_i$ for $1\le i\le n$. With this in mind, fix $1\le i\le n$, and regarding $F$ as a rational function on $U_i = \Spec R_i$, let $\alpha_i\ge 0$ be the smallest non-negative integer such that $h_i = z_i^{\alpha_i}F\in R_i$. Such an $\alpha_i$ exists since $U_i$ does not contain any $p_j$ with $j\ne i$ and $F$ is already regular on $U_0\cap U_i = U_i\setminus\{p_i\}$. Now, we may write $F t^m = h_i t^mz_i^{-\alpha_i}$. Recalling that the rational function $w_i = t^{d_i}z_i^{-1}$ is regular on $V_i$, we see that $Ft^m$ is regular on $V_i$ if and only if $d_i\alpha_i\le m$, i.e., $\alpha_i\le\lfloor \frac m{d_i}\rfloor$. Thus, we get $B_m = A_m\subset R_0$ as required.
\end{proof}

\begin{Corollary}\label{cor:contraction-normal}
    $\sY$ is normal.
\end{Corollary}
\begin{proof}
     If $f/g$ is integral over $A$ with $f,g\in A$ and $g\ne 0$, then we get a similar statement over $\sX$ using $A = \Gamma(\sX,\clO_{\sX})$. From the discussion below \eqref{eqn:main-family-int-coord}, recall that $\sX$ is a normal surface. It follows that the rational function $f/g$ is everywhere locally (and therefore globally) regular on $\sX$, i.e., there exists $h\in\Gamma(\sX,\clO_{\sX}) = A$ with $f = gh$. Thus, $f/g=h\in A$ and we conclude that $\sY$ is normal.
\end{proof}

\begin{Definition}[Contraction morphism]\label{def:contraction}
    Define
    \begin{align*}
        \Phi:\sX\to\sY
    \end{align*}
    to be the morphism corresponding to the identification \eqref{eqn:coord-ring-iso} of graded $k[t]$-algebras. By construction, $\Phi$ is $\bG_m$-equivariant and compatible with the projections $\pi_\sX$ and $\pi_\sY$ to $\bA^1$.
\end{Definition}

\begin{remark}\label{rem:affine-initial-object}
    $\Phi$ is an initial object in the category of morphisms from $\sX$ to affine schemes. Indeed, if $\sX\to\Spec S$ is a morphism, then we get a ring map $S\to\Gamma(\sX,\clO_{\sX}) = A$. By the correspondence between morphisms of affine schemes and ring maps, this yields a morphism $\sY = \Spec A\to\Spec S$ factoring the original morphism $\sX\to\Spec S$ through $\Phi$. 
\end{remark}

\begin{figure}[h]
    \centering
\includegraphics[width=12.3cm]{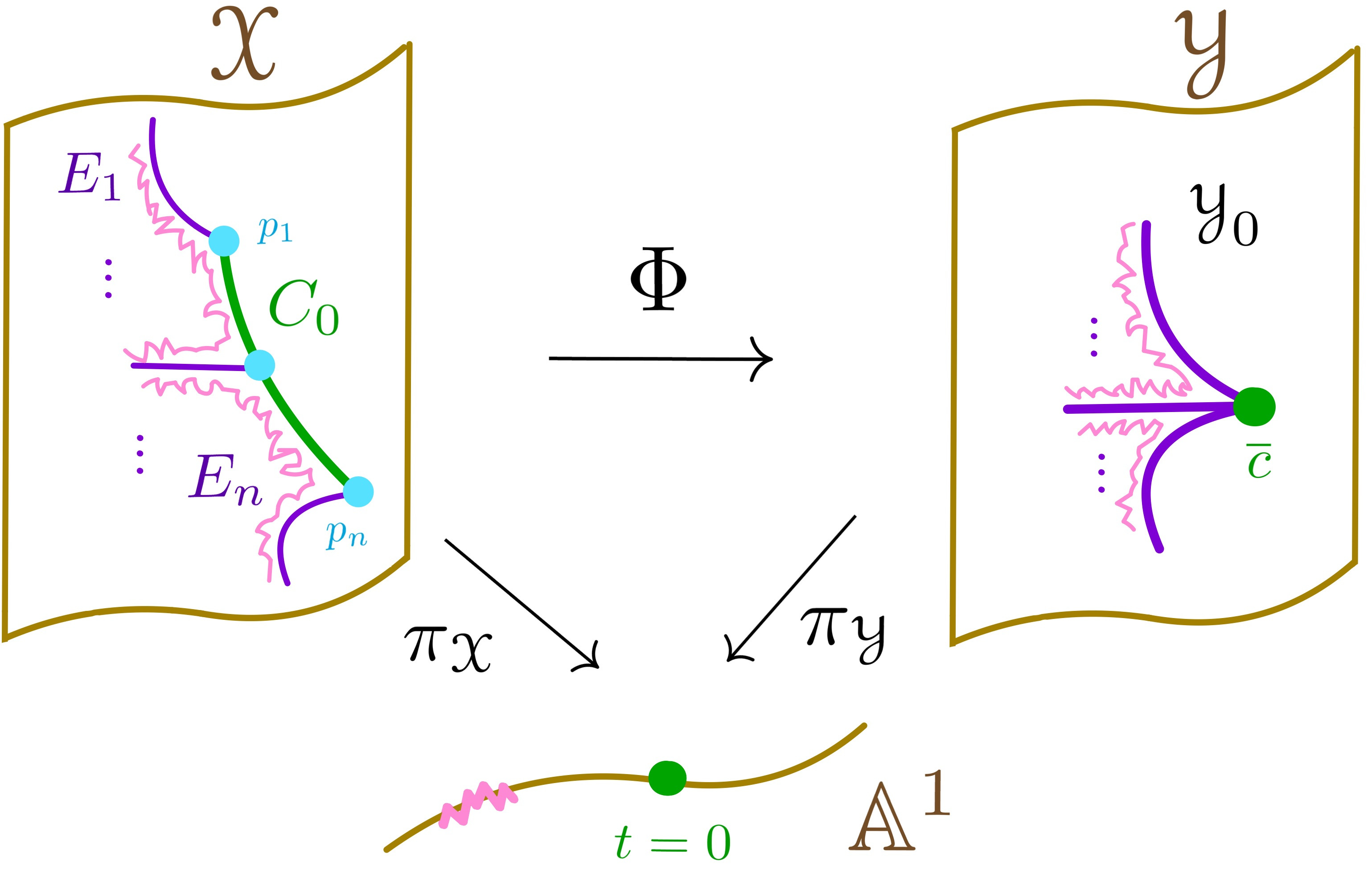}
    \caption{The contraction morphism $\Phi:\sX\to\sY$, from Definition \ref{def:contraction}.}
    \label{fig:appendixA}
\end{figure}

The following technical observation will be useful later.

\begin{Lemma}\label{lem:affine-technical}
    If $0\ne g\in A_+$ is any homogeneous element in grading $m\ge 1$, then the complement $\sX_g\subset\sX$ of the zero locus of $g$ is affine.
\end{Lemma}
\begin{proof}
    We ask the reader to consult the list of notations provided at the beginning of this appendix to follow the computations below. Exactly as in the proof of Lemma \ref{lem:coord-ring-iso}, write $g = F t^m$ over $V_0$ for some regular function $F$ on $U_0$ with poles of order at most $\lfloor\frac m{d_i}\rfloor$ at the points $p_i$, for $1\le i\le n$. For $1\le i\le n$, write $g = h_i w_i^{\alpha_i}t^{r_i}$ over $V_i$ with $0\le\alpha_i$ being the smallest non-negative integer such that $h_i = z_i^{\alpha_i}F$ is regular on $U_i$ and $r_i = m - d_i\alpha_i\ge 0$. Define $I:=\{1\le i\le n\,|\,r_i=0\}$. If $i\in I$, then we must have $\alpha_i = \frac{m}{d_i}>0$ since $m\ge 1$.

    In this paragraph, we describe the zero locus of $g$. We do so set-theoretically, rather than scheme-theoretically, since we are only interested in its complement. Let $D\subset U_0$ denote the zero locus of $F|_{U_0}$. Then, the zero locus of $g$ on $V_0$ is the union of $D\times\bA^1$ and $U_0\times\{0\}$. Now, fix $1\le i\le n$ and consider the zero locus of $g = h_iw_i^{\alpha_i}t^{r_i}$ on $V_i$. If $i\in I$, then $\alpha_i>0$ and the zero locus of $w_i^{\alpha_i}$ is $C_0\cap V_i$. If $i\not\in I$, then the zero locus of $t^{r_i}$ is the union of $C_0\cap V_i$ and $E_i$. Finally, we consider the zero locus of $h_i$ on $V_i$. Since $h_i = z_i^{\alpha_i}F_i$, the zero locus of $h_i$ on $V_i$ contains $(U_i\cap D)\times\bA^1$. If $\alpha_i > 0$, then this containment is an equality since $h_i = z_i^{\alpha_i}F$ is non-vanishing at $p_i$ (owing to the minimality of $\alpha_i$), $z_i$ has no zeros on $U_i\setminus\{p_i\}$, and $p_j\not\in U_i$ for $j\ne i$. If $\alpha_i = 0$, then $i\not\in I$ and any zeros of $h_i$ on $V_i$ which are not in $(U_i\cap D)\times\bA^1$ must lie over $p_i\in U_i$ are therefore contained in $E_i$ (which appears in the zero locus of $t^{r_i}$ as seen above). Putting all this together, we find that the zero locus of $g$ on $\sX$ consists of the union of $D\times\bA^1$, $C_0$, and those $E_i$ for which $i\not\in I$.

    To show that $\sX_g$ is affine, we argue as follows. To begin with, we claim that the union of $D$ and $\{p_i\}_{i\not\in I}$ is a non-empty set. Indeed, if $I=\{1,\ldots,n\}$, then the rational function $F$ \emph{must} have a pole of order $\alpha_i = m/{d_i}>0$ at each $p_i$, and therefore, $D$ \emph{must} be non-empty. It follows that the open subset $C_\text{aff}\subset C$, defined as the complement of the (non-empty) union of $D$ and $\{p_i\}_{i\not\in I}$, is affine. Write $C_\text{aff} = \Spec R_\text{aff}$. Observe that we still have $\{p_i\}_{i\in I}\subset C_\text{aff}$ and $\sX_g$ is obtained by taking the blow-up of $C_\text{aff}\times\bA^1$ along the disjoint union of the closed subschemes $\{p_i\}\times\Spec k[t]/(t^{d_i})$ for $i\in I$ and excising from it the proper transforms of $C_\text{aff}\times\{0\}$ and $\{p_i\}\times\bA^1$ for $i\in I$. Observe now that $G = 1/F \in R_\text{aff}$ is a regular function on $C_\text{aff}$ whose scheme-theoretic zero locus is $\sum_{i\in I} \alpha_i\cdot p_i$. Blowing up $C_\text{aff}\times\bA^1$ along the closed subscheme defined by the ideal $(G,t^m)$ and excising the proper transforms of $C_\text{aff}\times\{0\}$ and $\{p_i\}\times\bA^1$ for $i\in I$ yields the affine scheme 
    \begin{align}\label{eqn:aux-blowup}
        \sV := \Spec\left(\frac{R_\text{aff}[t,w^{\pm1}]}{(wG-t^m)}\right)
    \end{align}
    where $w$ is a new variable. For $i\in I$, let $S_i$ denote the coordinate ring of $C_\text{aff}\cap U_i$, i.e., we have $C_\text{aff}\cap U_i = \Spec S_i$. We can then write $G|_{C_\text{aff}\cap U_i} = z_i^{\alpha_i}u_i$ for some $u_i\in S_i^\times$. As a result, the image of the ideal $(G,t^m)$ in $S_i[t]$ can be re-expressed as $(z_i^{\alpha_i},t^m)$. Since $\alpha_i = m/{d_i}$ (for $i\in I$), we get a morphism $\sX_g\to\sV$ (respecting the projection to $C_\text{aff}\times\bA^1$, and therefore, birational) via the universal property of blowing up \cite[\href{https://stacks.math.columbia.edu/tag/0806}{Tag 0806}]{stacks-project}. For $i\in I$, this morphism is explicitly described over $C_\text{aff}\cap U_i$ by the $S_i[t]$-algebra map
    \begin{align}\label{eqn:map-between-blowups}
        \frac{S_i[t,w^{\pm1}]}{(wG-t^m)}\to\frac{S_i[t,w_i^{\pm1}]}{(z_iw_i-t^{d_i})}
    \end{align}
    defined by $w\mapsto w_i^{\alpha_i}/u_i$. Here, we have used \eqref{eqn:aux-blowup} and \eqref{eqn:main-family-int-coord} to determine the source and target of \eqref{eqn:map-between-blowups}. The relation $w_i^{\alpha_i} = wu_i$ shows that the target of \eqref{eqn:map-between-blowups} is a finite type module over the source of \eqref{eqn:map-between-blowups}. Moreover, $\sX_g\to\sV$ is an isomorphism over $C_\text{aff}\cap U_0$. Being a finite morphism is local on the target by \cite[\href{https://stacks.math.columbia.edu/tag/01WI}{Tag 01WI}]{stacks-project}. Thus, the preceding argument shows that the morphism $\sX_g\to\sV$ is finite. Finite morphisms are affine morphisms by \cite[\href{https://stacks.math.columbia.edu/tag/01WH}{Tag 01WH}]{stacks-project}. Since $\sV$ is affine, the same must be true for $\sX_g$.
\end{proof}

The next lemma formalizes the assertion that $\sY$ is obtained from $\sX$ by contracting $C_0$ to the point $\overline{c}\in\sY$ (Definition \ref{def:cone-point}) via the morphism $\Phi$. See Figure \ref{fig:appendixA}.

\begin{Lemma}[Contraction property]\label{lem:contraction}
    The set-theoretic inverse image of $\overline{c}\in\sY$ under $\Phi$ is $C_0\subset\sX$. Moreover, $\Phi$ restricts to an isomorphism 
    \begin{align}\label{eqn:iso-away-from-core}
        \sX\setminus C_0\xrightarrow{\simeq}\sY\setminus\{\overline{c}\}
    \end{align}
    on the complement.
\end{Lemma}
\begin{proof}
    Let $g\in A_+$ be any homogeneous element in grading $m\ge 1$. As in the proof of Lemma \ref{lem:coord-ring-iso}, write $g|_{V_0} = Ft^m$ for a regular function $F$ on $U_0$. This shows that $U_0\times\{0\}$, and therefore its closure $C_0$, lies in the zero locus of $g$ in $\sX$. As $g$ was arbitrary, it follows that $C_0\subset\Phi^{-1}(\overline{c})$.
    
    If $q\in\sX\setminus C_0$ is a closed point, then in order to show that $\Phi(q)\ne\overline{c}$, we must produce an element of $A_+$ which is non-vanishing at $q$. Let $p\in C$ be the image of $q$ under the projection $\sX\to C\times\bA^1\to C$. If $p\in U_0$, then $t\in A_1\subset A_+$ is already non-vanishing at $p$. If $p = p_i$ for some $1\le i\le n$, then take a large integer $d\gg 1$ which is divisible by all the $d_i$. The divisor $d\Delta$ is then base point free and we can find a regular function $F$ on $U_0$, corresponding to a section $\sigma$ of $\clO_C(d\Delta)$ which generates its stalk at $p_i$, such that $g = Ft^d$ is regular on $\sX$. Over $U_i$, we can write $F = h_i/z_i^{\alpha_i}$ with $\alpha_i = d/d_i$ being an integer, and $h_i$ being a regular function on $U_i$ such that $h_i(p_i)\ne 0$. (This is possible since $\sigma$ generates $\clO_C(d\Delta)$ at $p_i$.) This shows that $g = h_iw_i^{\alpha_i}$ on $V_i = \sX|_{U_i\times\bA^1}$. Since $w_i(q)\ne 0$, we get $g(q)\ne 0$.

    To establish the isomorphism \eqref{eqn:iso-away-from-core}, it suffices to show that for any homogeneous element $0\ne g\in A_+$ in grading $m\ge 1$, $\Phi$ induces an isomorphism between the complements of zero loci of $g$ in $\sX$ and $\sY$. (Indeed, since such elements $g$ generate $A_+$, these complements form an affine open cover of $\sY\setminus\{\overline{c}\}$.)
    Denote these open complements by $\sX_g\subset\sX$ and $\sY_g = \Spec A_g\subset\sY$. Note that $A_g$ denotes the localization of the ring $A$ at the element $g$, i.e., it is obtained from $A$ by inverting $g$. We then have the identifications 
    \begin{align}\label{eqn:localize-coord-ring}
        A_g = \Gamma(\sX,\clO_{\sX})_g = \Gamma(\sX_g,\clO_{\sX_g}).
    \end{align} 
    The first equality in \eqref{eqn:localize-coord-ring} is obtained from \eqref{eqn:coord-ring-iso} by localizing at the element $g$. To justify the second equality in \eqref{eqn:localize-coord-ring}, consider the two sides as sub-rings of the function field of $\sX$. The inclusion of the left side into the right side is clear. The opposite inclusion follows by noting that any regular function on $\sX_g$, after multiplication by a sufficiently high power of $g$, becomes a regular function on $\sX$. Under \eqref{eqn:localize-coord-ring}, the morphism $\sX_g\to\sY_g$ induced by $\Phi$ is simply the tautological morphism $\sX_g\to\Spec \Gamma(\sX_g,\clO_{\sX_g})$ which, by Lemma \ref{lem:affine-technical}, is an isomorphism.
\end{proof}

Combining Lemma \ref{lem:contraction} and Remark \ref{rem:affine-initial-object} leads to the following.

\begin{Corollary}[Universal property of contraction]\label{cor:univ-prop-contraction}
    Any morphism of schemes $\sX\to T$ which maps $C_0$ to a point in $T$ must factor (necessarily uniquely) through $\Phi$. Moreover, we have an isomorphism $\clO_{\sY} = \Phi_*\clO_{\sX}$ of sheaves.
\end{Corollary}

The next statement describes how to recover the $\sX$ from $\sY$, i.e., it describes the inverse of the contraction process.

\begin{Lemma}[Inverse of contraction]\label{lem:inverse-contraction}
    Define the subvariety 
    \begin{align*}
        \Gamma\subset\sY\times C
    \end{align*}
    to be the closure of the graph of the rational map $\sY\dashrightarrow C$ given by $\Phi^{-1}$ followed by the projection $\sX\to C\times\bA^1\to C$. Then, the natural birational morphism 
    \begin{align}\label{eqn:blowup-to-graph}
        \sX\to\Gamma
    \end{align}
    is finite, and therefore, projective. In particular, $\Phi:\sX\to\sY$ is projective and \eqref{eqn:blowup-to-graph} is the normalization morphism of $\Gamma$.
\end{Lemma}
\begin{proof}
    The natural morphism
    \begin{align*}
        \sX\to\Gamma\subset\sY\times C,    
    \end{align*}
    given by $\Phi$ on the first factor and $\sX\subset\overline\sX\to C$ on the second factor, is bijective on closed points (and is even an isomorphism away from $\{\overline{c}\}\times C$). This follows from Lemma \ref{lem:contraction} and the fact that $C_0\subset\sX\to C$ is an isomorphism. Showing that \eqref{eqn:blowup-to-graph} is finite amounts to showing that $\sX\to\sY\times C$ is finite. Note that, for $0\le i\le n$, the restriction of $\sX\to\sY\times C$ over the open subset $U_i\subset C$ is given by the morphism $V_i\to \sY\times U_i$ of affine varieties. 
    
    When $i=0$, this corresponds to the ring map $A\otimes R_0\to R_0[t]$. Here, $R_0\subset R_0[t]$ is the evident inclusion, while $A\subset R_0[t]$ comes from Remark \ref{rem:model-smoothing-subalg}. Since $t\in A_+\subset A$, this ring map is surjective and $V_0\to\sY\times U_0$ is a closed embedding. Now, fix $1\le i\le n$. The morphism $V_i\to\sY\times U_i$ corresponds to the ring map
    \begin{align}\label{eqn:blowup-to-graph-ring-map}
        A\otimes R_i\to \frac{R_i[t,w_i]}{(z_iw_i-t^{d_i})}
    \end{align}
    where we have used \eqref{eqn:main-family-int-coord} to determine the coordinate ring of $V_i$. Now, take an integer $d\gg 1$ divisible by $d_i$ and write $\alpha_i = d/d_i$. We can then find a global section $\sigma_i$ of the invertible sheaf $\clO_C(\alpha_i\cdot p_i)$ which generates its stalk at $p_i$. The section $\sigma_i$ corresponds to a rational function $F_i$ on $C$, which is regular away from $p_i$ and has a pole of order $\alpha_i$ at $p_i$. Let $U_i' = \Spec R_i'\subset U_i$ be an affine open neighborhood of $p_i$ on which $h_i:=z_i^{\alpha_i}F_i$ is a nowhere vanishing regular function, i.e., $h_i\in(R_i')^\times$. Localizing the map \eqref{eqn:blowup-to-graph-ring-map} from $U_i = \Spec R_i$ to $U_i' = \Spec R_i'$, we obtain the map
    \begin{align}\label{eqn:blowup-to-graph-localized-ring-map}
        A\otimes R'_i\to \frac{R'_i[t,w_i]}{(z_iw_i-t^{d_i})}.
    \end{align}
    We claim that \eqref{eqn:blowup-to-graph-localized-ring-map} is a finite ring map. Since $t\in A_+\subset A$, it suffices to check that a positive power of $w_i$ lies in the image of \eqref{eqn:blowup-to-graph-localized-ring-map}. For this, note that the expression $F_it^d$ defines a global regular function on $\sX$, i.e., an element of $\Gamma(\sX,\clO_{\sX}) = A$. Moreover, its image in the coordinate ring of $V_i$ can be written as $h_iw_i^{\alpha_i}$, using the relation $z_iw_i=t^{d_i}$. Thus, we find that \eqref{eqn:blowup-to-graph-localized-ring-map} maps the element $F_it^d\otimes(1/h_i)\in A\otimes R_i'$ to $w_i^{\alpha_i}$.
    
    We have shown that $\sX\to\sY\times C$ is a finite morphism when restricted over the open cover $\{U_0\}\cup\{U_i'\}_{1\le i\le n}$ of $C$. Being a finite morphism is local on the target by \cite[\href{https://stacks.math.columbia.edu/tag/01WI}{Tag 01WI}]{stacks-project}. This shows that $\sX\to\sY\times C$ is finite. By the discussion below \eqref{eqn:main-family-int-coord}, $\sX$ is normal. Since \eqref{eqn:blowup-to-graph} is finite and birational, we conclude that $\sX$ is the normalization of $\Gamma$.
\end{proof}

\begin{remark}[Comparison with \cite{Pinkham-surf-sing}]\label{rem:pinkham-comparison}
    As indicated in Remark \ref{rem:model-motivation-pinkham-surf-sing}, we take (a special case of) the \emph{result} of \cite[Theorem 5.1]{Pinkham-surf-sing} as the basis for our \emph{definition} of the model smoothing $\pi_\sY:\sY\to\bA^1$. The sequence of statements in this appendix (leading up to Lemma \ref{lem:inverse-contraction}) provide an alternate approach to \cite[\textsection 3]{Pinkham-surf-sing}, which constructs an explicit $\bG_m$-equivariant resolution of any $\bG_m$-equivariant affine normal surface singularity. The equivariant resolution is used to determine the graded coordinate ring of the singularity in \cite[Theorem 5.1]{Pinkham-surf-sing}. This is parallel to our Lemma \ref{lem:coord-ring-iso}.

    The construction in \cite[\textsection 3]{Pinkham-surf-sing} involves a transcendental step which uses the topological fundamental group $\pi_1$, and therefore, is written under the assumption $k=\bC$. Our approach in this appendix is a purely algebraic version of this construction (which works over any algebraically closed $k$ of characteristic $0$) in the special case of interest to us.
\end{remark}

The next lemma proves the explicit description of the (semi)normalization of $\sY_0$ stated in Lemma \ref{lem:justifying-partial-norm}. Before this, we need a brief preliminary discussion.

Note that $\sY_0$ is the image of $\bigsqcup_{1\le i\le n} E_i\subset\sX$ under the contraction morphism $\Phi$. By Lemma \ref{lem:contraction}, for $1\le i\le n$, the point $q_i\in E_i$ maps to $\overline{c}\in\sY_0$ while $E_i\setminus\{q_i\}$ maps isomorphically onto its image in $\sY_0\setminus\{\overline{c}\}$. Thus, the $\bG_m$-equivariant morphism 
\begin{align*}
    \bigsqcup_{1\le i\le n} E_i\to\sY_0    
\end{align*}
induced by $\Phi$ is the normalization morphism of $\sY_0$. Recall from the discussion below \eqref{eqn:main-family-int-coord} that we have identifications $E_i = T^\vee_{C,p_i}$ with $q_i = 0$ (and $\bG_m$ acting with weight $d_i$) for $1\le i\le n$. Thus, gluing the lines $E_i$ along the points $q_i$ for $1\le i\le n$ recovers the scheme $E = \Spec Q$ introduced in Definition \ref{def:partial-norm}. Since the points $q_i\in E_i$ for $1\le i\le n$ all map to $\overline{c}\in\sY_0$, we get a factorization
\begin{align}\label{eqn:norm-factor}
    \bigsqcup_{1\le i\le n} E_i\to E\to\sY_0
\end{align}
of the normalization morphism of $\sY_0$ (see Figure \ref{fig:Partial normalization}).

\begin{Lemma}[Alternate descriptions of seminormalization]\label{lem:partial-normalization-ring-map}
    The morphisms $E\to\sY_0$ in \eqref{eqn:norm-factor} and \eqref{eqn:partial-normalization} are the same.
\end{Lemma}
\begin{proof}
    To understand the proof, the reader should recall the notations introduced in Definition \ref{def:partial-norm}. We need to show that the inclusion $\fA\subset Q = \bigoplus_{m\ge 0}Q_m$ induced by \eqref{eqn:norm-factor} and the inclusion $\fA\subset\bigoplus_{m\ge 0}H^0(C,\clQ_m)$ from \eqref{eqn:curve-sing-incl-in-Q} coincide under the canonical identification $Q_m = H^0(C,\clQ_m)$ from \eqref{eqn:coord-ring-E-and-Q}. It suffices to check this after composing with the natural surjection $A\to\fA$ (coming from Lemma \ref{lem:model-sing}). Moreover, we can check this one grading at a time. There is nothing to check for grading $m=0$. Now, let $0\ne g\in A_+$ be a homogeneous element in grading $m\ge 1$ and let $1\le i\le n$. We will show that the image of $g$ in $(\clQ_m)_{p_i}$, defined via the exact sequence \eqref{eqn:ses-Q}, corresponds to the image of $g$ in $Q_{m,i}$, defined by restriction of functions from $\sX$ to $E_i$, under the identification provided by \eqref{eqn:coord-ring-E-and-Q}. Recall from Definition \ref{def:partial-norm} that $Q_{m,i}$ is the summand of $Q_m$ corresponding to $1\le i\le n$.

    We use the notation of Lemma \ref{lem:coord-ring-iso}. Over $V_0 = U_0\times\bA^1$, write $g = Ft^m$ for some regular function $F$ on $U_0$. Let $\alpha_i\ge 0$ be the smallest non-negative integer such that $h_i = z_i^{\alpha_i}F$ becomes regular on $U_i$ and define $r_i = m - d_i\alpha_i\ge 0$. Using the coordinate presentation \eqref{eqn:main-family-int-coord}, we rewrite $g = h_iw_i^{\alpha_i}t^{r_i}$ on $V_i$. Recalling from the discussion below \eqref{eqn:main-family-int-coord} that $E_i\subset\sX$ is defined by the ideal $(z_i,t)$, we see that $g|_{E_i}$ vanishes unless $r_i = 0$. Similarly, when we regard $F$ as a section of $\clO_C(\lfloor m\Delta\rfloor)$, we see that it also has vanishing image in $(\clQ_m)_{p_i}$ unless $r_i = 0$. Therefore, we will suppose that $r_i = 0$ for the remainder of the argument.
    
    In this case, since $F$ has a pole of order $\alpha_i = m/d_i>0$ at $p_i$, we get $h_i(p_i)\ne 0$. This shows that $g|_{E_i}$ is simply the function $h_i(p_i)w_i^{\alpha_i}$ in the coordinate ring $k[w_i]$ of $E_i$. On the other hand, the image of $F$ in $(\clQ_m)_{p_i}$ is the same as the image of the local section $h_i(p_i)z_i^{-\alpha_i}$ in $(\clQ_m)_{p_i}$ since their difference $F - h_i(p_i)z_i^{-\alpha_i}$ has (at worst) a pole of order strictly less than $\alpha_i$ at $p_i$. Unwinding the computation \eqref{eqn:Q-local-stalks}, we see that the monomial $w_i^{\alpha_i}$ in $Q_m$ corresponds to $z_i^{-\alpha_i}$ in $(\clQ_m)_{p_i}$. Thus, the images of $g$ in $Q_m$ and $(\clQ_m)_{p_i}$ coincide under the identification \eqref{eqn:coord-ring-E-and-Q}.
\end{proof}

We conclude this appendix with the following proposition, which is the key to proving local smoothability of the model pinching (Proposition \ref{prop:local-smoothability-of-model-pinching}). Its proof just consists of putting together everything we have already proved in this appendix.

\begin{Proposition}\label{prop:local-smoothability-of-model-pinching-appx}
    The morphism $\Phi:\sX\to\sY$, produced by Definitions \ref{def:main-family}, \ref{def:main-family-int} and \ref{def:contraction}, is a proper birational morphism and has the following properties.
    \begin{enumerate}[\normalfont(i)]
        \item The morphism $\pi_\sX = \pi_\sY\circ\Phi:\sX\to\bA^1$ is flat and $\Phi\times_{\bA^1}\{0\}$ is identified with the morphism $\nu:\sX_0\to\sY_0$ produced by Definition \ref{def:model-pinching}.
        \item $\Phi$ restricts to an isomorphism over $\sY\setminus\{\overline c\}$.
    \end{enumerate}
\end{Proposition}
\begin{proof}
    Property (ii) follows from Lemma \ref{lem:contraction}. Birationality of $\Phi$ is a consequence of property (ii). Properness of $\Phi$ follows from Lemma \ref{lem:inverse-contraction}. By the construction of $\Phi$ in Definition \ref{def:contraction}, we get $\pi_\sY\circ\Phi = \pi_\sX$. The morphism $\pi_{\overline\sX}$ is flat by Remark \ref{rem:X-flatness}. Being the restriction of $\pi_{\overline\sX}$ to the open subvariety $\sX\subset\overline\sX$, the morphism $\pi_\sX$ is also flat. By Remark \ref{rem:Gm-equivariance}, the fibre $\pi_\sX^{-1}(0)$ is the union of the curves $E_i$ for $1\le i\le n$ and $C_0$ (these curves were introduced in Definition \ref{def:main-family-int}). From the discussion below \eqref{eqn:main-family-int-coord}, we have the identification $E_i = T_{C,p_i}^\vee$. Therefore, we get the desired identification $\pi_{\sX}^{-1}(0) = \sX_0$. We are left with checking that, under this identification, $\Phi\times_{\bA^1}\{0\}$ is the same as $\nu:\sX_0\to\sY_0$. This is an immediate consequence of Lemma \ref{lem:partial-normalization-ring-map}.
\end{proof}

%%%%%%%%%%%%%%%%%%%%%%%%%%%%%%%%%%%%%%%%%%%%%%%%%%%%
\section{Local smoothability of genus 0 pinchings}\label{appendix:genus-0-pinchings}

In this short appendix, we prove the following proposition, which is the key to establishing the local smoothability of a genus $0$ pinching (Proposition \ref{prop:local-smoothability-genus-0-pinching}).

\begin{Proposition}\label{prop:local-smoothability-genus-0-pinching-appx}
    Let $\nu:\sC\to\sS$ be the pinching of a prestable genus $0$ curve at a connected sub-curve $C\subsetneq\sC$, in the sense of Definition \ref{def:pinching}. Moreover, assume that the closure of $\sC\setminus C$ is a disjoint union of finitely many copies of $\bP^1$. Then, $\nu$ is locally smoothable, in the sense of Definition \ref{def:locally-smoothable-pinching}.
\end{Proposition}

The remainder of this appendix gives the proof of this proposition.

\begin{figure}[ht]
    \centering
\includegraphics[width=12.3cm]{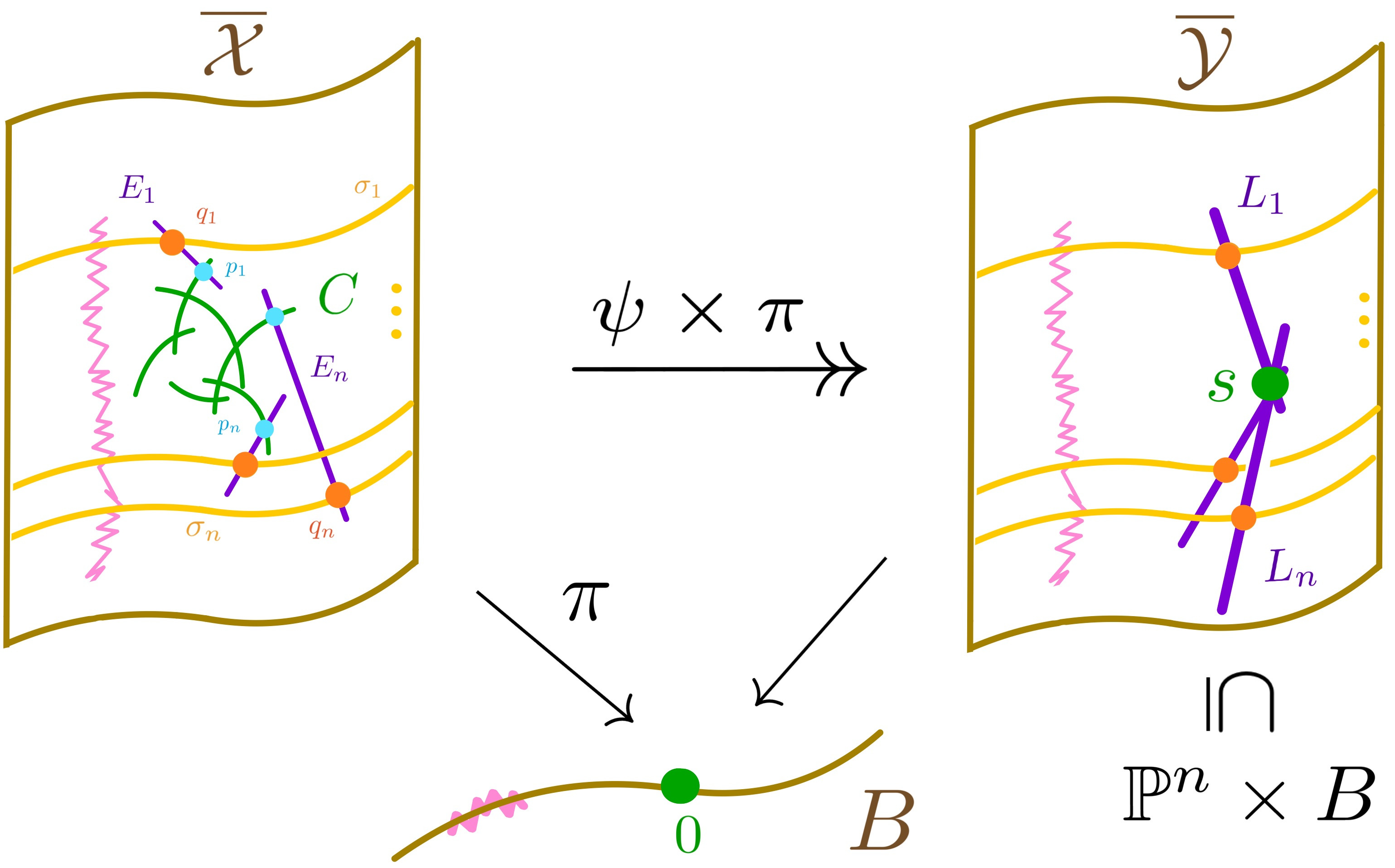}
    \caption{Depiction of the various objects appearing in the proof of Proposition \ref{prop:local-smoothability-genus-0-pinching-appx}.}
    \label{fig:appendixB}
\end{figure}

\begin{notation}\label{not:genus-0-pinching-data}
Enumerate the $\bP^1$ components of the closure of $\sC\setminus C$ as $E_1,\ldots,E_n$. For $1\le i\le n$, let $p_i$ be the unique point in $C\cap E_i$. Fix choices of points $q_i\in E_i\setminus\{p_i\}$ for $1\le i\le n$. Now, choose the data of 
\begin{enumerate}[(i)]
    \item a non-singular affine curve $B$ with a closed point $0\in B$,
    \item a flat, projective morphism $\pi:\overline\clX\to B$ which is smooth over $B\setminus\{0\}$, such that $\overline\clX$ is a non-singular irreducible surface with an identification $\overline\clX\times_B\{0\} = \sC$, and 
    \item pairwise disjoint sections $\sigma_i:B\to\overline\clX$ of $\pi$ with $\sigma_i(0) = q_i$ for $1\le i\le n$.
\end{enumerate}
It follows that the fibre of $\pi:\overline{\clX}\to B$ over any closed point in $B\setminus\{0\}$ is isomorphic to the projective line $\bP^1$. Denote by
\begin{align*}
    \clL = \clO_{\overline\clX}(\sum\sigma_i)
\end{align*}
the line bundle on $\clX$ determined by the smooth divisor $\sum\sigma_i(B)\subset\overline\clX$. Finally, for any closed point $b\in B$, we will write $\overline\clX_b:=\pi^{-1}(b)$ and $\clL_b:=\clL|_{\overline\clX_b}$. In particular, $\overline\clX_0 = \sC$.
\end{notation}

\begin{remark}
    One way to choose the data $0\in B$, $\pi:\overline\clX\to B$ and $\{\sigma_i\}_{1\le i\le n}$, as in Notation \ref{not:genus-0-pinching-data}, is the following. By suitably adding marked points (including $q_1,\ldots,q_n$) to $\sC$, we may view $\sC$ as a point of the smooth variety $\Mbar_{0,m}$ for some $m\ge 3$. We may now choose a non-singular affine curve $B\subset\Mbar_{0,m}$ which passes through this point (which we label as $0\in B$) and is transverse to each irreducible component of the boundary divisor of $\Mbar_{0,m}$. The family $\pi:\overline\clX\to B$ can be defined by restricting the universal family (of stable $m$-pointed genus $0$ curves) on $\Mbar_{0,m}$ to $B$. Pulling back the sections of the universal family corresponding to the marked points $q_1,\ldots,q_n\in\sC$ gives the desired sections $\sigma_1,\ldots,\sigma_n$ of $\pi$.
\end{remark}

The following standard result, concerning cohomology and base change, will be used in the proof of Lemma \ref{lem:nice-linear-system} below.
\begin{Lemma}[{\cite[Theorem III.12.11]{Har77}}]\label{lem:hartshorne-cohom-bc}
    Let $f:X\to Y$ be a projective morphism of noetherian schemes, and let $\clF$ be a coherent sheaf on $X$, which is flat over $Y$. Let $y$ be a point of $Y$. Then, we have the following.
    \begin{enumerate}[\normalfont(a)]
    \item If the natural map
        \begin{align*}
            \phi^i(y):R^if_*(\clF)\otimes_{\clO_{Y,y}} (\clO_{Y,y}/\fm_{Y,y})\to H^i(X_y,\clF_y)
        \end{align*} 
        is surjective, then it is an isomorphism, and the same is true for all points $y'$ in a suitable neighborhood of $y$.
        
    \item Assume that $\phi^i(y)$ is surjective. Then, the following two conditions are equivalent.
        \begin{enumerate}[\normalfont(i)]
        \item The map $\phi^{i-1}(y)$ is surjective.
        \item $R^if_*(\clF)$ is locally free in a neighborhood of $y$.
        \end{enumerate}
    \end{enumerate}
\end{Lemma}

\begin{Lemma}\label{lem:nice-linear-system}
    The coherent sheaf $\pi_*\clL$ is locally free of rank $n+1$ on $B$. The fibre of the resulting vector bundle at any closed point $b\in B$ is identified with $H^0(\overline\clX_b,\clL_b)$.
\end{Lemma}
\begin{proof}
    Note that if a line bundle on a genus $0$ prestable curve has non-negative degree on each irreducible component, then it is globally generated and its $H^1$ group vanishes. Moreover, if the total degree of such a line bundle is $d$, then its $H^0$ group has dimension $d+1$. For $\bP^1$, this is because every line bundle on $\bP^1$ is a power of $\clO_{\bP^1}(1)$, and in general, this follows by induction on the number of irreducible components. Thus, $H^1(\overline\clX_b,\clL_b) = 0$ and $\dim H^0(\overline\clX_b,\clL_b) = n+1$, for any closed point $b\in B$. 

    We will now repeatedly apply Lemma \ref{lem:hartshorne-cohom-bc} to finish the proof. Applying Lemma \ref{lem:hartshorne-cohom-bc}(a) to the $H^1$ computation from the previous paragraph implies that $R^1\pi_*(\clL)$ is the zero sheaf;  in particular, it is locally free of rank $0$. Lemma \ref{lem:hartshorne-cohom-bc}(b), again applied to the above $H^1$ computation, now implies that the natural map 
    \begin{align}\label{eqn:cohom-bc-map}
        \pi_*\clL\otimes_{\clO_{B,b}}(\clO_{B,b}/\fm_{B,b})\to H^0(\overline\clX_b,\clL_b)
    \end{align} 
    is surjective, for all closed points $b\in B$. Lemma \ref{lem:hartshorne-cohom-bc}(a), applied to the surjectivity of \eqref{eqn:cohom-bc-map}, shows that \eqref{eqn:cohom-bc-map} is actually an isomorphism. Now, Lemma \ref{lem:hartshorne-cohom-bc}(b) implies that $\pi_*\clL$ is a locally free sheaf. The fibre of the resulting vector bundle is determined by the isomorphism \eqref{eqn:cohom-bc-map}. Its rank is seen to be $n+1$ by the $H^0$ computation from the previous paragraph.
\end{proof}

\begin{notation}
    After possibly replacing $B$ by an affine open neighborhood of $0\in B$, use Lemma \ref{lem:nice-linear-system} to fix an identification $\clO_B^{\oplus n+1} = \pi_*\clL$, defined by $n+1$ global sections $\psi_0,\ldots,\psi_n$. These sections restrict to a basis of $H^0(\overline\clX_b,\clL_b)$ for each closed point $b\in B$. In particular, the sections $\psi_0,\ldots,\psi_n$ give a base point free linear system on $\overline\clX$ and we get a morphism
    \begin{align*}
        \psi=[\psi_0:\cdots:\psi_n]:\overline\clX\to\bP^n.
    \end{align*}
    Since $\clL$ is trivial over $C\subset\sC\subset\overline\clX$, the morphism $\psi$ maps $C$ to a point, denoted by $s\in\bP^n$. Let $\overline\clY\subset\bP^n\times B$ the closed subvariety which is defined to be the image of the proper morphism $\psi\times\pi$. The projection $\overline\clY\to B$ is surjective and so, it is flat by Remark \ref{rem:flatness-over-dvr}. It therefore defines a flat family $\{\overline\clY_b\subset\bP^n\}_{b\in B}$ of closed subschemes, parametrized by the affine curve $B$. Finally, let
    \begin{align*}
        \psi|_b:\overline\clX_b\to\overline\clY_b\subset\bP^n
    \end{align*}
    be the restriction of $\psi$ over any closed point $b\in B$. Note that $\psi|_b$ is defined by the linear system $\psi_0|_{\overline\clX_b},\ldots,\psi_n|_{\overline\clX_b}$ obtained by restricting the sections $\psi_0,\ldots,\psi_n$ to $\overline\clX_b$.
\end{notation}

Consider any closed point $b\in B$. If $b\ne 0$, then $(\overline\clX_b,\clL_b)$ is isomorphic to $(\bP^1,\clO_{\bP^1}(n))$; therefore, $\psi|_b:\overline\clX_b\to\overline\clY_b$ is an isomorphism.  Thus, whenever $b\ne 0$, the Hilbert polynomial\footnote{Recall that the Hilbert polynomial of a coherent sheaf $\clF$ on $\bP^n$ is given by $t\in\bZ\mapsto \chi(\clF\otimes_{\clO_{\bP^n}}\clO_{\bP^n}(t))$. The Hilbert polynomial of closed subscheme $S\subset\bP^n$ is, by definition, the Hilbert polynomial of $\clO_S$.} of the subscheme $\overline\clY_b\subset\bP^n$ is given by $P(t) = nt+1$. Since $\overline\clY\to B$ is flat, we conclude that the subscheme $\overline\clY_0\subset\bP^n$ also has the same Hilbert polynomial.

\begin{Lemma}
    The scheme $\overline\clY_0$ is reduced. Moreover, $\psi|_0:\overline\clX_0\to\overline\clY_0$ is identified with the pinching $\nu:\sC\to\sS$ of $\sC$ at $C$, and $\psi$ induces an isomorphism $\overline\clX\setminus C\to\overline\clY\setminus\{(s,0)\}$.  
\end{Lemma}
\begin{proof}
    Since $\psi|_0$ is the morphism defined by the complete linear system of $\clL_0$ on $\overline\clX_0 = \sC$, we see that, for $1\le i\le n$, the restriction $\psi|_{E_i}$ is an isomorphism onto a linearly embedded projective line $L_i\subset\bP^n$ passing through $s$ (which is the image of $C$ under $\psi$). Moreover, we have an equality $\bigoplus_{1\le i\le n}T_{L_i,s} = T_{\bP^n,s}$ of tangent spaces. Define $Z\subset\bP^n$ to be the reduced subscheme given by the union $L_1\cup\cdots\cup L_n$. A direct computation shows that the Hilbert polynomial of $Z$ is also given by $P(t) = nt+1$. Now, $\overline\clY_0$ contains $Z$ as a subscheme and they both have the \emph{same} Hilbert polynomial. This means $\ker(\clO_{\overline\clY_0}\twoheadrightarrow\clO_{Z})$ has zero Hilbert polynomial and is, therefore, the zero sheaf. Thus, $\overline\clY_0 = Z$ is reduced.

    By the previous paragraph and Example \ref{exa:genus-0-pinching}, we see that $\psi|_0:\overline\clX_0\to\overline\clY_0$ is the \emph{unique} pinching of $\sC$ at $C$. Since $\nu:\sC\to\sS$ is another such pinching, the two must be identified. The final assertion (that $\overline\clX\setminus C\to\overline\clY\setminus\{(s,0)\}$ is an isomorphism) follows since $\psi|_{\overline\clX\setminus C}$ is a closed embedding into $\bP^n\times B\setminus\{(s,0)\}$.
\end{proof}

To complete the proof of Proposition \ref{prop:local-smoothability-genus-0-pinching-appx}, let us define $\clX := \overline\clX\setminus\bigcup_{1\le i\le n}\sigma_i(B)$ and $\clY := \overline\clY\setminus\bigcup_{1\le i\le n}\sigma_i(B)$. The identity morphism of $\clY_0 = \sS\setminus\{q_1,\ldots,q_r\}$, the pointed curve $0\in B$, and the $B$-morphism $\psi:\clX\to\clY$ can now be taken to play the respective roles of $\varphi_i:U_i\to V_i$, $0\in B_i$, and $\psi_i:\clX_i\to\clY_i$ in Definition \ref{def:locally-smoothable-pinching}. This concludes the proof.\qed

\bibliography{references} 

% \bib, bibdiv, biblist are defined by the amsrefs package.
\begin{bibdiv}
\begin{biblist}

\bib{Artin-modifications}{article}{
      author={Artin, M.},
       title={Algebraization of formal moduli. {II}. {E}xistence of
  modifications},
        date={1970},
        ISSN={0003-486X},
     journal={Ann. of Math. (2)},
      volume={91},
       pages={88\ndash 135},
         url={https://doi.org/10.2307/1970602},
      review={\MR{260747}},
}

\bib{Artin-tifr-lectures}{book}{
      author={Artin, M.},
       title={Lectures on deformations of singularities. {Notes} by {C}. {S}.
  {Seshadri}, {Allen} {Tannenbaum}},
    language={English},
      series={Lect. Math. Phys., Math., Tata Inst. Fundam. Res.},
   publisher={Springer, Berlin; Tata Inst. of Fundamental Research, Bombay},
        date={1976},
      volume={54},
}

\bib{Badescu-book}{book}{
      author={B\u{a}descu, L.},
       title={Algebraic surfaces},
      series={Universitext},
   publisher={Springer-Verlag, New York},
        date={2001},
        ISBN={0-387-98668-5},
         url={https://doi.org/10.1007/978-1-4757-3512-3},
        note={Translated from the 1981 Romanian original by Vladimir Ma\c{s}ek
  and revised by the author},
      review={\MR{1805816}},
}

\bib{Battistella-Carocci}{article}{
      author={Battistella, Luca},
      author={Carocci, Francesca},
       title={A smooth compactification of the space of genus two curves in
  projective space: via logarithmic geometry and {G}orenstein curves},
        date={2023},
        ISSN={1465-3060,1364-0380},
     journal={Geom. Topol.},
      volume={27},
      number={3},
       pages={1203\ndash 1272},
         url={https://doi.org/10.2140/gt.2023.27.1203},
      review={\MR{4599312}},
}

\bib{behrend-manin}{article}{
      author={Behrend, K.},
      author={Manin, Yu.},
       title={Stacks of stable maps and {G}romov-{W}itten invariants},
        date={1996},
        ISSN={0012-7094,1547-7398},
     journal={Duke Math. J.},
      volume={85},
      number={1},
       pages={1\ndash 60},
         url={https://doi.org/10.1215/S0012-7094-96-08501-4},
      review={\MR{1412436}},
}

\bib{NMSP3}{article}{
      author={Chang, Huai-Liang},
      author={Guo, Shuai},
      author={Li, Jun},
       title={{B}{C}{O}{V}'s {F}eynman rule of quintic 3-folds},
        date={2018},
     journal={arXiv preprint arXiv:1810.00394},
}

\bib{NMSP2}{article}{
      author={Chang, Huai-Liang},
      author={Guo, Shuai},
      author={Li, Jun},
       title={Polynomial structure of {G}romov-{W}itten potential of quintic
  {$3$}-folds},
        date={2021},
        ISSN={0003-486X,1939-8980},
     journal={Ann. of Math. (2)},
      volume={194},
      number={3},
       pages={585\ndash 645},
         url={https://doi.org/10.4007/annals.2021.194.3.1},
      review={\MR{4334973}},
}

\bib{NMSP1}{article}{
      author={Chang, Huai-Liang},
      author={Guo, Shuai},
      author={Li, Jun},
      author={Li, Wei-Ping},
       title={The theory of {$N$}-mixed-spin-{$P$} fields},
        date={2021},
        ISSN={1465-3060,1364-0380},
     journal={Geom. Topol.},
      volume={25},
      number={2},
       pages={775\ndash 811},
         url={https://doi.org/10.2140/gt.2021.25.775},
      review={\MR{4251436}},
}

\bib{log-GLSM1}{article}{
      author={Chen, Qile},
      author={Janda, Felix},
      author={Ruan, Yongbin},
       title={The logarithmic gauged linear sigma model},
        date={2021},
        ISSN={0020-9910,1432-1297},
     journal={Invent. Math.},
      volume={225},
      number={3},
       pages={1077\ndash 1154},
         url={https://doi.org/10.1007/s00222-021-01044-2},
      review={\MR{4296354}},
}

\bib{log-GLSM2}{article}{
      author={Chen, Qile},
      author={Janda, Felix},
      author={Ruan, Yongbin},
       title={Punctured logarithmic {R}-maps},
        date={2022},
     journal={arXiv preprint arXiv:2208.04519},
}

\bib{log-GLSM3}{misc}{
      author={Chen, Qile},
      author={Janda, Felix},
      author={Ruan, Yongbin},
       title={The structural formulae of virtual cycles in logarithmic gauged
  linear sigma models, \emph{in preparation}},
}

\bib{CLLL}{article}{
      author={Chang, Huai-Liang},
      author={Li, Jun},
      author={Li, Wei-Ping},
      author={Liu, Chiu-Chu~Melissa},
       title={Mixed-spin-{P} fields of {F}ermat polynomials},
        date={2019},
        ISSN={2168-0930,2168-0949},
     journal={Camb. J. Math.},
      volume={7},
      number={3},
       pages={319\ndash 364},
         url={https://doi.org/10.4310/CJM.2019.v7.n3.a3},
      review={\MR{4010064}},
}

\bib{CLLL2}{article}{
      author={Chang, Huai-Liang},
      author={Li, Jun},
      author={Li, Wei-Ping},
      author={Liu, Chiu-Chu~Melissa},
       title={An effective theory of {GW} and {FJRW} invariants of quintic
  {C}alabi-{Y}au manifolds},
        date={2022},
        ISSN={0022-040X,1945-743X},
     journal={J. Differential Geom.},
      volume={120},
      number={2},
       pages={251\ndash 306},
         url={https://doi.org/10.4310/jdg/1645207466},
      review={\MR{4385118}},
}

\bib{DW-counting}{article}{
      author={Doan, Aleksander},
      author={Walpuski, Thomas},
       title={Counting embedded curves in symplectic 6-manifolds},
        date={2023},
        ISSN={0010-2571,1420-8946},
     journal={Comment. Math. Helv.},
      volume={98},
      number={4},
       pages={693\ndash 769},
         url={https://doi.org/10.4171/cmh/556},
      review={\MR{4680501}},
}

\bib{Eisenbud-Harris-Wpoint}{article}{
      author={Eisenbud, D.},
      author={Harris, J.},
       title={Existence, decomposition, and limits of certain {W}eierstrass
  points},
        date={1987},
        ISSN={0020-9910,1432-1297},
     journal={Invent. Math.},
      volume={87},
      number={3},
       pages={495\ndash 515},
         url={https://doi.org/10.1007/BF01389240},
      review={\MR{874034}},
}

\bib{ekholm-shende-ghost}{article}{
      author={Ekholm, Tobias},
      author={Shende, Vivek},
       title={Ghost bubble censorship},
        date={2022},
     journal={arXiv preprint arXiv:2212.05835},
}

\bib{Guo-Janda-Ruan-BCOV-g2}{article}{
      author={Guo, S.},
      author={Janda, F.},
      author={Ruan, Y.},
       title={A mirror theorem for genus two {G}romov-{W}itten invariants of
  quintic threefolds},
        date={2017},
     journal={arXiv preprint arXiv:1709.07392},
}

\bib{Guo-Janda-Ruan-BCOV-higher-g}{article}{
      author={Guo, S.},
      author={Janda, F.},
      author={Ruan, Y.},
       title={{S}tructure of higher genus {G}romov-{W}itten invariants of
  quintic 3-folds},
        date={2018},
     journal={arXiv preprint arXiv:1812.11908},
}

\bib{Grauert-contraction}{article}{
      author={Grauert, H.},
       title={\"{U}ber {M}odifikationen und exzeptionelle analytische
  {M}engen},
        date={1962},
        ISSN={0025-5831,1432-1807},
     journal={Math. Ann.},
      volume={146},
       pages={331\ndash 368},
         url={https://doi.org/10.1007/BF01441136},
      review={\MR{137127}},
}

\bib{M2}{misc}{
      author={Grayson, D.~R.},
      author={Stillman, M.~E.},
       title={Macaulay2, a software system for research in algebraic geometry},
        note={Available at \url{http://www.math.uiuc.edu/Macaulay2/}},
}

\bib{Har-DT}{book}{
      author={Hartshorne, R.},
       title={Deformation theory},
      series={Graduate Texts in Mathematics},
   publisher={Springer, New York},
        date={2010},
      volume={257},
        ISBN={978-1-4419-1595-5},
         url={https://doi.org/10.1007/978-1-4419-1596-2},
      review={\MR{2583634}},
}

\bib{Har77}{book}{
      author={Hartshorne, R.},
       title={Algebraic geometry},
      series={Graduate Texts in Mathematics, No. 52},
   publisher={Springer-Verlag, New York-Heidelberg},
        date={1977},
        ISBN={0-387-90244-9},
      review={\MR{0463157}},
}

\bib{geogebra5}{misc}{
      author={Hohenwarter, M.},
      author={Borcherds, M.},
      author={Ancsin, G.},
      author={Bencze, B.},
      author={Blossier, M.},
      author={\'Eli\'as, J.},
      author={Frank, K.},
      author={G\'al, L.},
      author={Hofst\"atter, A.},
      author={Jordan, F.},
      author={Kone\v{c}n\'y, Z.},
      author={Kov\'acs, Z.},
      author={Lettner, E.},
      author={Lizelfelner, S.},
      author={Parisse, B.},
      author={Solyom-Gecse, C.},
      author={Stadlbauer, C.},
      author={Tomaschko, M.},
       title={{G}eo{G}ebra 5.0.507.0},
        date={2018},
        note={Available at \url{http://www.geogebra.org}},
}

\bib{hu-li-niu-genus-2}{article}{
      author={Hu, Y.},
      author={Li, J.},
      author={Niu, J.},
       title={Genus two stable maps, local equations and modular resolutions},
        date={2012},
     journal={arXiv preprint arXiv:1201.2427},
}

\bib{ionel-j-inv}{article}{
      author={Ionel, E.-N.},
       title={Genus {$1$} enumerative invariants in {$\mathbf{P}^n$} with fixed
  {$j$} invariant},
        date={1998},
        ISSN={0012-7094,1547-7398},
     journal={Duke Math. J.},
      volume={94},
      number={2},
       pages={279\ndash 324},
         url={https://doi.org/10.1215/S0012-7094-98-09414-5},
      review={\MR{1638587}},
}

\bib{KKO}{article}{
      author={Kim, Bumsig},
      author={Kresch, Andrew},
      author={Oh, Yong-Geun},
       title={A compactification of the space of maps from curves},
        date={2014},
        ISSN={0002-9947,1088-6850},
     journal={Trans. Amer. Math. Soc.},
      volume={366},
      number={1},
       pages={51\ndash 74},
         url={https://doi.org/10.1090/S0002-9947-2013-05845-X},
      review={\MR{3118390}},
}

\bib{mumford-path-4}{article}{
      author={Mumford, D.},
       title={Pathologies {IV}},
        date={1975},
        ISSN={0002-9327,1080-6377},
     journal={Amer. J. Math.},
      volume={97},
      number={3},
       pages={847\ndash 849},
         url={https://doi.org/10.2307/2373780},
      review={\MR{460338}},
}

\bib{Niu-thesis}{book}{
      author={Niu, J.},
       title={Refined {C}onvergence for {G}enus-{T}wo {P}seudo-{H}olomorphic
  {M}aps},
   publisher={ProQuest LLC, Ann Arbor, MI},
        date={2016},
        ISBN={978-1369-18729-8},
  url={http://gateway.proquest.com/openurl?url_ver=Z39.88-2004&rft_val_fmt=info:ofi/fmt:kev:mtx:dissertation&res_dat=xri:pqm&rft_dat=xri:pqdiss:10164143},
        note={Thesis (Ph.D.)--State University of New York at Stony Brook},
      review={\MR{3597726}},
}

\bib{Pinkham-thesis}{book}{
      author={Pinkham, H.},
       title={Deformations of algebraic varieties with {$G\sb{m}$} action.},
   publisher={Soci\'{e}t\'{e} Math\'{e}matique de France, Paris,},
        date={1974},
      review={\MR{376672}},
}

\bib{Pinkham-surf-sing}{article}{
      author={Pinkham, H.},
       title={Normal surface singularities with {$C\sp*$} action},
        date={1977},
        ISSN={0025-5831,1432-1807},
     journal={Math. Ann.},
      volume={227},
      number={2},
       pages={183\ndash 193},
         url={https://doi.org/10.1007/BF01350195},
      review={\MR{432636}},
}

\bib{PT-13-over-2}{incollection}{
      author={Pandharipande, R.},
      author={Thomas, R.~P.},
       title={13/2 ways of counting curves},
        date={2014},
   booktitle={Moduli spaces},
      series={London Math. Soc. Lecture Note Ser.},
      volume={411},
   publisher={Cambridge Univ. Press, Cambridge},
       pages={282\ndash 333},
      review={\MR{3221298}},
}

\bib{sequel}{article}{
      author={Rezaee, F.},
      author={Swaminathan, M.},
       title={An obstructing to smoothing stable maps},
        date={2024},
     journal={arXiv preprint arXiv:2407.01845},
}

\bib{Ranganathan--Santos-Parker--Wise}{article}{
      author={Ranganathan, D.},
      author={Santos-Parker, K.},
      author={Wise, J.},
       title={Moduli of stable maps in genus one and logarithmic geometry {I}},
        date={2019},
     journal={Geometry \& Topology},
      volume={23},
       pages={3315\ndash 3366},
}

\bib{stacks-project}{misc}{
      author={{Stacks project authors}, The},
       title={The stacks project},
         how={\url{https://stacks.math.columbia.edu}},
        date={2023},
}

\bib{Vakil-genus01}{article}{
      author={Vakil, Ravi},
       title={The enumerative geometry of rational and elliptic curves in
  projective space},
        date={2000},
        ISSN={0075-4102,1435-5345},
     journal={J. Reine Angew. Math.},
      volume={529},
       pages={101\ndash 153},
         url={https://doi.org/10.1515/crll.2000.094},
      review={\MR{1799935}},
}

\bib{Vakil-stable-red}{incollection}{
      author={Vakil, R.},
       title={A tool for stable reduction of curves on surfaces},
        date={2001},
   booktitle={Advances in algebraic geometry motivated by physics ({L}owell,
  {MA}, 2000)},
      series={Contemp. Math.},
      volume={276},
   publisher={Amer. Math. Soc., Providence, RI},
       pages={145\ndash 154},
         url={https://doi.org/10.1090/conm/276/04517},
      review={\MR{1837115}},
}

\bib{VZ-desing}{article}{
      author={Vakil, Ravi},
      author={Zinger, Aleksey},
       title={A desingularization of the main component of the moduli space of
  genus-one stable maps into {$\Bbb P^n$}},
        date={2008},
        ISSN={1465-3060,1364-0380},
     journal={Geom. Topol.},
      volume={12},
      number={1},
       pages={1\ndash 95},
         url={https://doi.org/10.2140/gt.2008.12.1},
      review={\MR{2377245}},
}

\bib{Zinger-sharp-compactness}{article}{
      author={Zinger, A.},
       title={A sharp compactness theorem for genus-one pseudo-holomorphic
  maps},
        date={2009},
        ISSN={1465-3060,1364-0380},
     journal={Geom. Topol.},
      volume={13},
      number={5},
       pages={2427\ndash 2522},
         url={https://doi.org/10.2140/gt.2009.13.2427},
      review={\MR{2529940}},
}

\bib{Zinger-BCOV-g1}{article}{
      author={Zinger, Aleksey},
       title={The reduced genus 1 {G}romov-{W}itten invariants of
  {C}alabi-{Y}au hypersurfaces},
        date={2009},
        ISSN={0894-0347,1088-6834},
     journal={J. Amer. Math. Soc.},
      volume={22},
      number={3},
       pages={691\ndash 737},
         url={https://doi.org/10.1090/S0894-0347-08-00625-5},
      review={\MR{2505298}},
}

\end{biblist}
\end{bibdiv}

\end{document}